\documentclass[11pt]{amsart}

\usepackage{fancyhdr}
\usepackage{amsmath}
\usepackage{amssymb}
\usepackage{amsthm}
\usepackage{bbm}
\usepackage{hyperref}
\usepackage{esint}
\usepackage[margin=1in]{geometry}
\usepackage{arydshln}
\usepackage{xcolor}
\usepackage{tikz}
	\usetikzlibrary{cd}

\newcommand{\cc}[1]{\mathcal{#1}}

\newcommand{\bC}{{\mathbb C}}
\newcommand{\bZ}{{\mathbb Z}}
\newcommand{\bN}{{\mathbb N}}

\newcommand{\bR}{{\mathbb R}}

\newcommand{\abs}[1]{\left|{#1}\right|}
\newcommand{\norm}[1]{\lVert#1\rVert}

\newcommand{\vol}{\operatorname{vol}}

\newcommand{\Mod}{\operatorname{Mod}}
\newcommand{\dMod}{\operatorname{dMod}}

\newcommand{\ip}[2]								
{\left< #1 , #2 \right>}

\newcommand{\N}{\mathbb{N}}						
\newcommand{\Z}{\mathbb{Z}}						
\newcommand{\R}{\mathbb{R}}						
\newcommand{\K}{\mathbb{K}}						

\renewcommand{\H}{\mathbb{H}}

\newcommand{\eps}{\varepsilon}					

\newcommand{\dd}								
{\mathop{}\!\mathrm{d}}						
\newcommand{\ddn}[1]							
{\mathop{}\!\mathrm{d^{#1}}}
\newcommand{\hodge}{\mathtt{\star}\hspace{1pt}}	

\newcommand{\smallnorm}[1]						
{\lVert #1 \rVert}

\DeclareMathOperator{\id}{id}					
\DeclareMathOperator{\intr}{int}				

\DeclareMathOperator{\tor}{Tor}					
\DeclareMathOperator{\spt}{spt}					

\DeclareMathOperator{\lip}{lip}

\DeclareMathOperator{\maxfun}{{\mathbf{M}}}		

\newcommand{\loc}{\mathrm{loc}}					

\def\XXint#1#2#3{{\setbox0=\hbox{$#1{#2#3}{\int}$}
		\vcenter{\hbox{$#2#3$}}\kern-.5\wd0}}

\newcommand{\moddens}{\operatorname{Mod}}
\newcommand{\modform}{\operatorname{dMod}}
\newcommand{\adm}{\operatorname{adm}}
\newcommand{\essadm}{\operatorname{ess\,adm}}

\newcommand{\cA}{\mathcal{A}} \newcommand{\cB}{\mathcal{B}}
\newcommand{\cC}{\mathcal{C}} 
\newcommand{\cE}{\mathcal{E}} \newcommand{\cF}{\mathcal{F}}

 \newcommand{\cL}{\mathcal{L}}
\newcommand{\cM}{\mathcal{M}} 
\newcommand{\cR}{\mathcal{R}} \newcommand{\cS}{\mathcal{S}}
 
 \newcommand{\cW}{\mathcal{W}}

\newtheorem{theorem}{Theorem}[section]
\newtheorem{cor}[theorem]{Corollary}
\newtheorem{lemma}[theorem]{Lemma}
\newtheorem{prop}[theorem]{Proposition}
\newtheorem{definition}[theorem]{Definition}
\newtheorem{remark}[theorem]{Remark}

\newtheorem*{namedtheorem}{\theoremname}
\newcommand{\theoremname}{testing}
\newenvironment{named}[1]{\renewcommand{\theoremname}{#1}\begin{namedtheorem}}{\end{namedtheorem}}

\begin{document}
	
\title{On the Moduli of Lipschitz Homology Classes}

\author{Ilmari Kangasniemi}
\address{Department of Mathematics, Syracuse University, Syracuse,
NY 13244, USA }
\email{kikangas@syr.edu}

\author{Eden Prywes}
\address{Department of Mathematics, Princeton University, Princeton, NJ 08544, USA}
\email{eprywes@princeton.edu}
	
\thanks{Eden Prywes was partially supported by the NSF grant RTG-DMS-1502424.}
	
\begin{abstract}
    We define a type of modulus $\modform_p$ for Lipschitz surfaces based on $L^p$-integrable measurable differential forms, generalizing the vector modulus of Aikawa and Ohtsuka. We show that this modulus satisfies a homological duality theorem, where for H\"older conjugate exponents $p, q \in (1, \infty)$, every relative Lipschitz $k$-homology class $c$ has a unique dual Lipschitz $(n-k)$-homology class $c'$ such that $\modform_p^{1/p}(c) \modform_q^{1/q}(c') = 1$ and the Poincar\'e dual of $c$ maps $c'$ to 1.  As $\modform_p$ is larger than the classical surface modulus $\moddens_p$, we immediately recover a more general version of the estimate $\moddens_p^{1/p}(c) \moddens_q^{1/q}(c') \leq 1$, which appears in works by Freedman and He and by Lohvansuu. Our theory is formulated in the general setting of Lipschitz Riemannian manifolds, though our results appear new in the smooth setting as well. We also provide a characterization of closed and exact Sobolev forms on Lipschitz manifolds based on integration over Lipschitz $k$-chains.
\end{abstract}
	
\maketitle
	
\section{Introduction}
In \cite{Fuglede_surface-modulus}, Fuglede defined the modulus for a family of measures.  He used this to define the modulus for a family of Lipschitz surfaces, extending the classical definition for a family of paths.  In this paper, we give a definition for a new version of the modulus of a family of surfaces using differential forms, which we denote $\dMod$.  We are able to prove that our definition possesses duality properties that Fuglede's classical definition of the modulus of surfaces does not satisfy. The main theorem of the paper is as follows:
\begin{theorem}\label{thm:mainthm}
    Let $M$ be a closed Lipschitz Riemannian manifold.  Let $c \in H_k^{\lip}(M;\bZ)$ be a non-torsion element of the rank $k$ Lipschitz homology group.  For all $p \in (1,\infty)$, there exists a unique element $c' \in H^{\lip}_{n-k}(M;\bR)$ so that the Poincar\'e dual of $c$ maps $c'$ to $1$ and
    \begin{align*}
        \dMod_p(c)^{1/p}\dMod_q(c')^{1/q} = 1,
    \end{align*}
    where $p^{-1}+q^{-1} = 1$.
\end{theorem}

\subsection{Modulus and duality}
The study of path modulus was initiated by Beurling and Ahlfors \cite{ahlforsbeurling1950} in order to study quasiconformal maps.  If $\Gamma$ is a family of paths in $\Omega \subset \bC$, then the \emph{$p$-modulus} is defined as
\begin{align*}
    \Mod_p(\Gamma) = \inf \bigg \{ \int_\Omega \rho^p : \rho \in L^p(\Omega) \text{ and }\int_\gamma \rho \ge 1 \text{ for } \gamma \in \Gamma\bigg \}.
\end{align*}
The $2$-modulus of path families is preserved by planar conformal maps and is preserved up to a constant by planar quasiconformal maps.  
It is also connected to capacity problems and the calculus of variations. 

The $p$-modulus satisfies a duality property for topological quadrilaterals.  Let $Q$ be a topological quadrilateral with sides $E_1,E_2,E_3$ and $E_4$ arranged cyclically.  If $\Gamma_1$ is the family of paths connecting $E_1$ to $E_3$ in $Q$ and $\Gamma_2$ is the family of paths connecting $E_2$ to $E_4$ in $Q$, then
\begin{align*}
    \Mod_p(\Gamma_1)^{1/p}\Mod_q(\Gamma_2)^{1/q} = 1,
\end{align*}
where $p^{-1}+q^{-1} = 1$.

Fuglede \cite{Fuglede_surface-modulus} defined the modulus for a family of measures in an analogous way.  Let $\cc M$ be a family of measures on a set $\Omega \subset \bR^n$.  The $p$-modulus of $\cc M$ is defined as
\begin{align*}
    \Mod_p(\cc M) = \inf \bigg \{ \int_\Omega \rho^p : \rho \in L^p(\Omega) \text{ and }\int_\Omega \rho d\mu \ge 1 \text{ for } \mu \in \cc M\bigg \}.
\end{align*}
Since Lipschitz surfaces always have corresponding surface measures, this definition also gives a definition for a family of Lipschitz surfaces.  Fuglede was interested in duality results that resemble the duality for quadrilaterals in the curves setting and was able to prove such a theorem for the $2$-modulus of surfaces. In particular, he proved that the collection of all curves that connect a compact set $K \subset \R^n$ to $\infty$ has dual $2$-modulus to the collection of Lipschitz hypersurfaces that separate $K$ from $\infty$ \cite[Theorem 9]{Fuglede_surface-modulus}.

Gehring \cite{gehring1961} and Ziemer \cite{Ziemer_Modulus-duality} proved a similar duality results for the exponent $p = n$.  Gehring showed, under smoothness assumptions on the surfaces, that the $n$-modulus of a family of curves connecting two compact continua is dual to the $n/(n-1)$-modulus of the family of surfaces separating them.
Ziemer extended the result to remove the smoothness assumption on the surface family. 

A general duality theory for families of surfaces does not hold.  This was observed by Freedman and He in \cite{Freedman-He_Modulus-duality}.  However, they pose similar capacity problems, which do satisfy duality properties in the cases they are interested in, namely solid tori in $3$-dimensional space. They are also able to show that for families of surfaces $A$ and $B$ which are dual in a homological sense, one has
\begin{align*}
    \Mod_p(A)^{1/p}\Mod_q(B)^{1/q} \le 1,
\end{align*}
where $p^{-1}+q^{-1} = 1$.

In \cite{Lohvansuu_duality}, Lohvansuu studied the modulus of surfaces in Lipschitz cubes.  Let $Q = h(I^n)$, where $I = [0,1]$ and $h$ is a bilipschitz map. Let $A = h(\partial I^k\times  I^{n-k})$ and $B = h( I^k\times \partial I^{n-k})$.  In this case, the relative homology satisfies $H^k(Q,A;\bZ) \cong H^k(Q,B;\bZ) \cong \bZ$.  If $\Gamma_A$ and $\Gamma_B$ are the generators of the homology groups, Lohvansuu proved that
\begin{align*}
    \Mod_p(\Gamma_A)^{1/p}\Mod_q(\Gamma_B)^{1/q} \le 1,
\end{align*}
where $p^{-1}+q^{-1} = 1$.

\subsection{Modulus with differential forms}
The above results indicate that finding a modulus theory that gives a duality property is not trivial even in settings where the topology is not complicated, like a cube or torus.  For this reason, in this paper we focus on a modulus that uses differential forms instead of densities.
In \cite{Aikawa-Ohtsuka_vector-modulus}, Aikawa and Ohtsuka defined a vector modulus that assigned a modulus to a family of vector valued measures.  This is very similar to the modulus we define, however they only consider families of curves.  They are able to show duality properties in this setting and relate their theory to capacity properties of sets.

Let $\cS$ be a family of Lipschitz $k$-chains in $\Omega \subset \bR^n$.  We define the \emph{$p$-differential form modulus} as
\begin{align*}
    \dMod_p(\cS) = \inf \int_\Omega |\omega|^p,
\end{align*}
where the infimum is taken over all $k$-forms $\omega \in L^p(\Omega)$ such that
\begin{align*}
    \int_\sigma \omega \ge 1,
\end{align*}
for $\sigma \in \cS$, outside a set of zero $p$-modulus in $\cS$.  Excluding an exceptional set in $\cS$ is necessary, as noted in \cite{Aikawa-Ohtsuka_vector-modulus} and discussed in Section \ref{sec:moddifforms}.  The differential form modulus is particularly suited to proving duality theorems.  

We prove Theorem \ref{thm:mainthm} by proving a more general theorem that includes manifolds with boundary.  Let $M$ be a closed $n$-dimensional Lipschitz submanifold of an oriented boundaryless Lipschitz Riemannian reference $n$-manifold $\cR$. Additionally, suppose that $\partial M = D \cup E$, where $D$ and $E$ are closed $(n-1)$-dimensional Lipschitz submanifolds of $\partial M$ with $\partial D = \partial E$.

\begin{theorem}\label{thm:mainthmrelative}
    Let $p,q \in (1,\infty)$ satisfy $p^{-1} + q^{-1} = 1$.  If $c \in  H_k^{\lip}(M,D;\bZ)$ is a non-torsion element, then there exists a unique element $c' \in H_{n-k}^{\lip}(M,E;\bR)$ such that the Poincar\'e dual of $c$ maps $c'$ to $1$ and
    \begin{align*}
        \dMod_p(c)^{1/p}\dMod_q(c')^{1/q} = 1.
    \end{align*}
\end{theorem}

Theorem \ref{thm:mainthm} follows from Theorem \ref{thm:mainthmrelative} by letting $M = \cR$ and having $\partial M =  D = E = \emptyset$.  In the case where $H_k^{\lip}(M,D;\bZ) \cong H_{n-k}^{\lip}(M,E;\bZ) \cong \bZ$, as in the setting of \cite{Lohvansuu_duality}, we can strengthen our theorem to ensure that $c' \in H_{n-k}^{\lip}(M,E;\bZ)$.
\begin{theorem}\label{thm:mainthmZ}
	Let $p,q \in (1,\infty)$ satisfy $p^{-1} + q^{-1} = 1$.  If $H_k^{\lip}(M,D;\bZ) \cong H_{n-k}^{\lip}(M,E;\bZ) \cong \bZ$ and $c \in  H_k^{\lip}(M,D;\bZ)$ and $c' \in H_{n-k}^{\lip}(M,E;\bZ)$ are the generators of the homology groups, then
    \begin{align*}
        \dMod_p(c)^{1/p}\dMod_q(c')^{1/q} = 1.
    \end{align*}
\end{theorem}

An immediate consequence of Theorem \ref{thm:mainthmZ} is the following corollary, which generalizes the main result in \cite{Lohvansuu_duality}.
\begin{cor}\label{cor:modineq}
	Let $p, q \in (1, \infty)$ with $p^{-1} + q^{-1} = 1$. Suppose that $H_{k}^{\lip}(M,D; \Z) \cong H_{n-k}^{\lip}(M,E; \Z) \cong \Z$ for some $k \in \{1, \dots, n-1\}$. Let $c$ and $c'$ be the generators of $H_{k}^{\lip}(M; \Z)$ and $H_{n-k}^{\lip}(M; \Z)$, respectively. Then 
	\[
		\moddens_p(c)^{\frac{1}{p}} \moddens_q(c')^{\frac{1}{q}} \leq 1.
	\]
\end{cor}

As mentioned above, a version of Theorem \ref{thm:mainthmZ} does not in general hold for the classical surface modulus defined by Fuglede, due to a counterexample given by Freedman and He in \cite{Freedman-He_Modulus-duality}. This counterexample is a solid torus with a $180^{\circ}$ twist, the slices of which are dumbbell-shaped surfaces isometric to $S = B^2((1,0), 2^{-1}) \cup B^2((-1,0), 2^{-1}) \cup ([-1, 1] \times [-\eps, \eps])$. See Figure \ref{fig:Freedman-He_example} for an illustration. A flat version of this example can be defined by  $M = (S \times [0, 1])/\sim$, with the equivalence relation $\sim$ given by $(z, 0) \sim (-z, 1)$.

\begin{figure}[h]
    \centering
    \includegraphics[width=10cm]{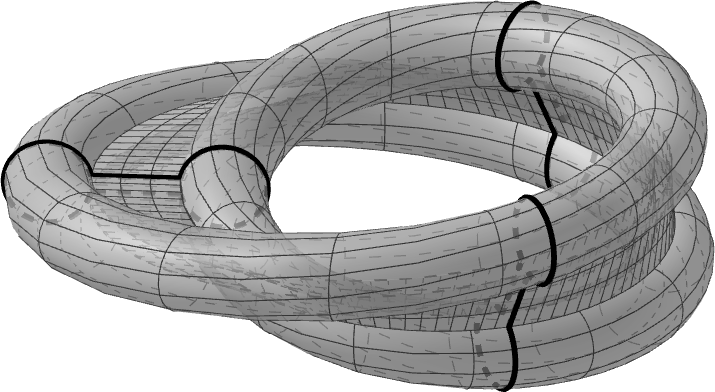}
    \caption{The dumbbell-shaped twisted solid torus of Freedman and He that acts as a counterexample to a version of Theorem \ref{thm:mainthmZ} for the classical modulus.}
    \label{fig:Freedman-He_example}
\end{figure}

Indeed, in the context of Theorem \ref{thm:mainthmZ}, $M$ is an oriented Lipschitz Riemannian manifold with $H_1(M,\emptyset;\bZ) \cong H_2(M,\partial M;\bZ) \cong \bZ$. If $c \in H_1(M,\emptyset;\bZ)$ is a generator, then the density $\rho$ which is $1$ in the handle $[-2^{-1}, 2^{-1}] \times [-\eps, \eps] \times [0, 1]$ and $0$ elsewhere will be admissible for $c$, as any loop winding around $M$ exactly once that doesn't stay in the handle has to cross it. Hence, $\moddens_p(c) \leq \eps$. However, for a generator $c'$ of $H_2(M,\partial M;\bZ)$, by taking a constant density $\rho'$ in the end-part $(B^2((1,0), 2^{-1}) \cup B^2((-1,0), 2^{-1})) \times [0, 1]$ of the dumbbell, one finds an upper bound for $\moddens_q(c')$ independent of $\eps$. Thus, for small enough $\eps$, one has $\moddens_p(c)^{1/p} \moddens_q(c')^{1/q} < 1$.

Similarly to the planar case discussed earlier, the classical path modulus is connected to the theory of quasiconformal maps in that quasiconformal maps on $\bR^n$ preserve the $n$-modulus, up to a multiplicative constant. For surface modulus, only partial results have been shown.  For example, Rajala in \cite{Rajala_smallK} showed a result for the image of sphere families by quasiregular maps. It is an interesting question whether there exists a class of $k$-surfaces that is invariant under quasiconformal transformations and has a reasonable definition for the $n/k$-modulus that is quasi-preserved by quasiconformal maps. A similar question can also be asked for the differential form modulus.

\subsection{Ideas of the proofs}

We now outline the proofs of the main theorems in this paper. Since our setting is that of Lipschitz manifolds, there are technical difficulties in defining the integration theory for surfaces. This is mainly due to the fact that the Riemannian metric $g$ for a Lipschitz manifold $M$ is defined almost everywhere.  Thus, the surface measure of an arbitrary Lipschitz $k$-chain $\sigma$ on $M$, which is used to define the $p$-modulus and consequently $p$-exceptional sets, is not necessarily defined in an intrinsic way. We are however able to define $p$-exceptional families of Lipschitz $k$-chains on $M$ without first defining the $p$-modulus of surfaces on $M$, essentially relying on the general modulus of measures and the quasi-preservation of surface measures under bilipschitz maps. 

This definition of $p$-exceptional sets unlocks the requisite integration theory of $L^p$-forms over $p$-almost every Lipschitz $k$-chain on $M$ and moreover allows us to define the classical $p$-modulus on $M$ using the fact that $p$-almost every Lipschitz $k$-chain has a surface measure. Following this, we obtain versions of the de Rham theorem and Stokes' theorem in the setting of Sobolev forms and Lipschitz chains. Our arguments take advantage of the integration theory of flat forms by Whitney \cite{Whitney_book} and Federer \cite{Federer_book}, which also applies in our setting of Lipschitz manifolds.

We then prove a characterization of weak exactness and closedness for Sobolev differential forms via integration over Lipschitz $k$-chains. This characterization is similar in spirit to Fuglede's original motivation for defining $p$-modulus. In \cite[Section 3]{Fuglede_surface-modulus}, he showed that the $p$-modulus can be used to detect when vector fields have zero divergence or curl. In particular, our version generalizes Fuglede's results to differential forms of arbitrary degree, applies on Lipschitz manifolds and also includes the characterization of weakly vanishing boundary values. The following theorem is a simplified version of this characterization not involving boundary values; the full statement can be found in Section \ref{sec:stokes}. While the theorem does not use differential form modulus, it is an essential ingredient in its study.

\begin{theorem}\label{thm:characterization_thmintro}
    Let $M$ be an oriented Lipschitz Riemannian $n$-manifold without boundary. Let $p \in (1,\infty)$ and $\omega \in L^p(\wedge^kT^*M)$. The form $\omega$ is weakly closed if and only if 
    \begin{align*}
        \int_\sigma \omega = 0,
    \end{align*}
    for all $\sigma \in \partial C_{k+1}^{\lip}(M;\bZ)$ outside a $p$-exceptional set. Additionally, $\omega$ is weakly exact if and only if
    \begin{align*}
        \int_\sigma \omega = 0,
    \end{align*}
    for all $\sigma \in C_{k}^{\lip}(M;\bZ) \cap \ker(\partial)$ outside a $p$-exceptional set.
\end{theorem}

Following this, we define the modulus $\dMod_p$ of differential forms and study admissible forms for a homology class $c$. An $L^p$-differential form $\omega$ is \emph{admissible} for $c$ if
\begin{align*}
    \int_\sigma \omega \ge 1
\end{align*}
for $\sigma \in c$, outside family that is $p$-exceptional in terms of the classical modulus of surfaces. The key observation is that in this case, we can use the characterization theory outlined in Theorem \ref{thm:characterization_thmintro} to show that $\omega$ is in fact weakly closed, despite us having no regularity assumptions on $\omega$ beyond $L^p$-integrability. This directly connects admissibility of forms for homology classes to Sobolev-de Rham cohomology theory.

With these preliminaries we are finally able to show the duality theorems, Theorems \ref{thm:mainthm} and \ref{thm:mainthmrelative}. The proof uses the connection between $p$-differential form modulus and the capacity of differential forms, which was observed in \cite{Freedman-He_Modulus-duality}.
We first show that for a Lipschitz homology class $c$, the modulus problem admits a unique minimizer $\omega$. In particular, we prove in Section \ref{sec:moddifforms} that this minimizer satisfies the following properties.

\begin{theorem}\label{thm:modcapconnection}
    Let $M$ be a closed $n$-dimensional Lipschitz submanifold in an oriented Lipschitz Riemannian reference $n$-manifold $\cR$ and let $D$ be a closed $(n-1)$-dimensional Lipschitz submanifold of $\partial M$. Let $p \in (1,\infty)$ and let $c \in H^{\lip}_k(M,D;\bZ)$, with $0 < k < n$.  If $c$ is not a torsion element, then there exists a unique differential form $\omega$ such that $\omega$ is $L^p$-integrable, weakly closed, $p$-harmonic, $\omega$ has weakly vanishing tangential part in $D$, $|\omega|^{p-2}\omega$ has weakly vanishing normal part in $E = \partial M \setminus D$ and
    \begin{align*}
        \dMod_p(c) = \int_M |\omega|^p \vol_M = \inf_\tau \int_M |\omega + d\tau|^p \vol_M,
    \end{align*}
    where the infimum is over all $L^p$-integrable $\tau$, with an $L^p$-integrable weak exterior derivative $d\tau$ and with weakly vanishing tangential part in $D$.
\end{theorem}

We then connect the form $\omega$ to Poincar\'e duality. In particular, if $M$ is compact, then the Poincar\'e dual of $c$ corresponds to a Sobolev-de Rham cohomology class $[\zeta]$, where we may assume that $\zeta$ solves the capacity problem,
\begin{align*}
    \int_M \abs{\zeta}^q \vol_M = \inf_\tau \int_M \abs{\zeta + d\tau}^q \vol_M,
\end{align*}
with $p^{-1} + q^{-1} = 1$ and with the infimum taken over Sobolev forms $\tau$. We prove that $\zeta$ and $\omega$ are in fact related to each other via the formula
\begin{align*}
    \omega = \frac{|\zeta|^{q-2}\hodge \zeta}{\|\zeta\|_q^q}.
\end{align*} 
This indicates the connection between capacity problems and the differential form modulus. The dual homology class $c'$ of $c$ is then found as the Poincar\'e dual of $[ \omega]$ up to a sign and we prove that $\zeta$ is a $q$-modulus minimizing admissible form for $c'$. The class $c'$ a priori has real coefficients, but in the case where $H_{k}^{\lip}(M,D; \Z) \cong H_{n-k}^{\lip}(M,E; \Z) \cong \Z$, we are able to show that $c'$ in fact corresponds to a generator of $H_{n-k}^{\lip}(M,E; \Z)$.

We remark that key parts of our theory rely on a technical result that Lipschitz homology classes on Lipschitz manifolds are never $p$-exceptional in the sense of the classical modulus of surfaces. Our proof of this fact ended up sufficiently involved to deserve its own section. Indeed, even though the ideas behind this are clear locally, $p$-exceptionality is a global property, so most standard local proofs appear to fail.  Our approach uses a technique based on mollification and leads to a working proof that is mostly local, although the method becomes increasingly technical in the case involving relative homology. 

\subsection*{Structure of this paper}

In Section \ref{sec:modulusofmeasures}, we recall Fuglede's modulus and present the relevant theorems.  In Section \ref{sec:lipschitzmfds}, we recall the necessary background information regarding Lipschitz Riemannian manifolds and in Section \ref{sec:lipintegration} we discuss integration theory on Lipschitz Riemannian manifolds. In Section \ref{sec:sobolevhomology}, we recall the construction of local Sobolev cohomology on Lipschitz manifolds, as well as the related de Rham isomorphisms. In Section \ref{sec:pexceptional}, we prove that Lipschitz homology classes are not $p$-exceptional, a technical result we require for the duality theory. In Section \ref{sec:stokes}, we prove the characterization of weak exactness and closedness of differential forms outlined in Theorem \ref{thm:characterization_thmintro}.  In Section \ref{sec:moddifforms}, we define the differential form modulus and prove its basic properties. This includes a proof for Theorem \ref{thm:modcapconnection}. Section \ref{sec:duality} then includes the proofs of the main duality theorems for differential form modulus.

\subsection*{Acknowledgements} The authors would like to thank the AMS's Mathematical Research Communities, where this project was initiated in the Analysis on Metric Spaces session. They would also like to thank the participants of the session for helpful conversations on the topic, especially Mario Bonk and Sylvester Eriksson-Bique. The first author would like to thank Princeton University for hosting his visit, which was of significant use in completing this paper. 

\section{Modulus of measures}\label{sec:modulusofmeasures}

We recall Fuglede's definition for the modulus of a family of measures (see \cite{Fuglede_surface-modulus}).

\begin{definition}\label{def:fuglede_modulus}
	Let $(X, \mu)$ be a complete measure space and let $\cM_\mu$ be the family of all measures $\nu$ on $X$ such that every $\mu$-measurable set is $\nu$-measurable for all $\nu \in \cM_\mu$. If $\Sigma \subset \cM_\mu$, then a measurable function $f \colon X \to [0, \infty]$ is called \emph{admissible for $\Sigma$} if
	\begin{equation}\label{eq:admissibility}
		\int_X f \dd \nu \geq 1
	\end{equation}
	for every $\nu \in \Sigma$. We denote the set of admissible functions for $\Sigma$ by $\adm(\Sigma)$. The \emph{$p$-modulus} of $\Sigma$ for $p \in [1, \infty)$ is given by
	\[
	\moddens_{p}(\Sigma) = \inf_{f \in \adm(\Sigma)} \int_X \abs{f}^p \dd\mu.
	\]
\end{definition}

A family of measures $\Sigma$ on $(X, \mu)$ is called \emph{$p$-exceptional} if $\moddens_p(\Sigma) = 0$. A property of measures is said to hold for \emph{$p$-almost every ($p$-a.e.) $\nu \in \cM_\mu$} if it holds outside a $p$-exceptional subset of $\cM_\mu$. 

We recall some basic properties of the $p$-modulus (see \cite[Chapter I]{Fuglede_surface-modulus}). Given a measure $\nu$, we use $\overline{\nu}$ to denote its completion. Moreover, if $\Sigma, \Sigma' \subset \cM_\mu$, then we denote $\Sigma \geq \Sigma'$ if, for every $\sigma \in \Sigma$, there exists $\sigma' \in \Sigma'$ such that $\sigma(A) \geq \sigma'(A)$ for all $\mu$-measurable $A \subset X$. In this case, we say that $\Sigma'$ \emph{minorizes} $\Sigma$.

\begin{lemma}\label{lem:Fuglede_properties}
	Let $(X, \mu)$ be a complete measure space and let $p \in [1, \infty)$. 
	\begin{enumerate}
		\item \label{enum:fuglede_outer_measure} The modulus $\moddens_p$ is an outer measure on $\cM_\mu$.
		\item \label{enum:fuglede_minorization} If $\Sigma, \Sigma' \subset \cM_\mu$ and $\Sigma \geq \Sigma'$, then $\moddens_p(\Sigma) \leq \moddens_p(\Sigma')$.
		\item \label{enum:fuglede_measure_zero} If $E \subset X$ is such that $\mu(E) = 0$, then $\nu(E) = 0$ for $p$-a.e.\ $\nu \in \cM_\mu$.
		\item \label{enum:fuglede_integrability} If $f \in L^p(X, \mu)$, then $f \in L^1(X, \overline{\nu})$ for $p$-a.e.\ $\nu \in \cM_\mu$.
		\item \label{enum:fuglede_convergence} If $f \in L^p(X, \mu)$ and $f_i \to f$ in $L^p(X, \mu)$, then there exists a subsequence $f_{i_j}$ of $f_i$ such that $f_{i_j} \to f$ in $L^1(X, \overline{\nu})$ for $p$-a.e.\ $\nu \in \cM_\mu$.
	\end{enumerate}
\end{lemma}

The following lemma gives a characterization for $p$-exceptional sets.

\begin{lemma}[{\cite[Theorem 2]{Fuglede_surface-modulus}}]\label{lem:Fuglede_exceptional_lemma}
	Let $(X, \mu)$ be a complete measure space, let $\Sigma \subset \cM_\mu$ and let $p \in [1, \infty)$. The set $\Sigma$ is $p$-exceptional if and only if there exists a function $f \in L^p(X, \mu)$ such that $f \geq 0$ and
	\[
		\int_X f \dd \sigma = \infty
	\]
	for every $\sigma \in \Sigma$.
\end{lemma}

For two sets $\Sigma, \Sigma' \subset \cM_\mu$, we say that $\Sigma'$ \emph{weakly minorizes} $\Sigma$ if, for every $\sigma \in \Sigma$, there exists $\sigma' \in \Sigma'$ and $c_\sigma \in (0, 1)$ such that $\sigma(A) \geq c_\sigma \sigma'(A)$ for all $\mu$-measurable $A \subset X$. In this case, we denote $\Sigma \gtrsim \Sigma'$. Since the constant $c_\sigma$ in the definition of weak minorization can depend on $\sigma$, there is no estimate on $\Mod_p(\Sigma)$ in terms of $\Mod_p(\Sigma')$ if $\Sigma \gtrsim \Sigma'$. However, we do get the following result for exceptional sets as an immediate corollary of Lemma \ref{lem:Fuglede_exceptional_lemma}.

\begin{cor}\label{cor:weak_minorization_p-exc}
	Let $(X, \mu)$ be a complete measure space, let $\Sigma, \Sigma' \subset \cM_\mu$ and let $p \in [1, \infty)$. If $\Sigma'$ is $p$-exceptional and $\Sigma \gtrsim \Sigma'$, then $\Sigma$ is $p$-exceptional.
\end{cor}

A measurable function $f \colon X \to [0, \infty]$ is \emph{$p$-weakly admissible} for $\Sigma \subset \cM_\mu$ if \eqref{eq:admissibility} holds for $p$-almost every $\nu \in \cM_\mu$.
The set of such functions is denoted by $\essadm_p(\Sigma)$. 
It follows that
\begin{equation}\label{eq:modulus_essential_def}
	\moddens_{p}(\Sigma) = \inf_{f \in \essadm_p(\Sigma)} \int_X \abs{f}^p \dd\mu.
\end{equation}
If $p > 1$, the infimum in \eqref{eq:modulus_essential_def} is in fact a minimum with a unique minimizer (see the discussion in \cite[pp.\ 181--182]{Fuglede_surface-modulus}).

\section{Lipschitz manifolds}\label{sec:lipschitzmfds}

\subsection{Lipschitz manifolds without boundary}\label{subsect:lip_mfld_wo_bdry}

Let $\cR$ be a topological $n$-manifold without boundary, where we assume all manifolds are Hausdorff and second-countable. An atlas $\cA$ of charts $\phi \colon \Omega_{\phi} \to U_{\phi}$ with $\Omega_\phi \subset \R^n$ and $U_\phi \subset \cR$ is a \emph{Lipschitz atlas} if, for any two charts $\phi, \psi \in \cA$, the transition map $\phi^{-1} \circ \psi \colon \psi^{-1}(U_\phi \cap U_\psi) \to \phi^{-1}(U_\phi \cap U_\psi)$ is bilipschitz. A topological manifold $\cR$ equipped with a Lipschitz atlas $\cA$ is called a \emph{Lipschitz manifold}. The charts $\phi \in \cA$ are called \emph{bilipschitz charts}.

A \emph{Lipschitz Riemannian metric} $g$ on $\cR$ is a choice of a measurable Riemannian metric $\phi^* g \in \Gamma(T^*\Omega_\phi \otimes T^*\Omega_\phi)$ for every chart $\phi \in \cA$. .  The metric satisfies that for all $\phi,\psi \in \cA$ and $v \in \bR^n$, there exists $C_\phi \in [1,\infty)$ so that $C_\phi^{-1} \abs{v}^2 \leq \phi^* g(v, v) \leq C_\phi \abs{v}^2$ a.e.\ on $\Omega_\phi$ and $(\phi^{-1} \circ \psi)^* (\phi^* g) = \psi^* g$ a.e.\ on $\psi^{-1}(U_\phi \cap U_\psi)$.
A Lipschitz manifold $\cR$ equipped with a Lipschitz Riemannian metric $g$ is a \emph{Lipschitz Riemannian manifold}.

With a Lipschitz structure, we may define the set of measurable differential forms, $\Gamma(\wedge^kT^*M)$ in the usual way, for $M \subset \cR$ measurable (see e.g.\ \cite{Teleman_LipManifoldIndex}).  If $M$ is orientable, we may also define a volume form, denoted $\vol_g$.
If $f \colon \cR \to [0, \infty]$ is a measurable function we can define the usual integral of $f$ against the volume form using a continuous partition of unity. In particular, $\vol_g$ induces a measure on $\cR$. Moreover, if $\omega, \omega' \in \Gamma(\wedge^k T^* \cR)$, we obtain a well-defined measurable inner product $\ip{\omega}{\omega'} \colon \cR \to \R$. 
This allows us to define point-wise norms $\abs{\omega} \colon \cR \to [0, \infty)$ for measurable $k$-forms a.e.\ on $\cR$. We thus obtain the spaces $L^p(\wedge^k T^*D)$ and $L^p_\loc(\wedge^k T^*D)$ for any open set $D \subset \cR$ in the usual way, where the $L^p$-norm is given by
\[
	\norm{\omega}_{L^p} = \left(\int_{\cR} \abs{\omega}^p \vol_g \right)^\frac{1}{p}.
\] 

\subsection{Sobolev spaces on Lipschitz manifolds}

Let $\Omega$ be a domain in $\R^n$. We recall that if $\omega \in L^1_\loc(\wedge^k T^* \Omega)$, then $d\omega \in L^1_\loc(\wedge^{k+1} T^* \Omega)$ is a \emph{weak differential of $\omega$} if
\begin{equation}\label{eq:smooth_test_form_for_weak_d}
	\int_\Omega d\omega \wedge \eta = (-1)^{k+1} \int_\Omega \omega \wedge d\eta
\end{equation}
for all $\eta \in C^\infty_0(\wedge^{n-k-1} T^* \Omega)$. The Sobolev space of all $\omega \in L^p(\wedge^k T^* \Omega)$ with a weak differential $d\omega \in L^p(\wedge^k T^* \Omega)$ is denoted by $W^{d,p}(\wedge^k T^* \Omega)$. Locally integrable versions $W^{d,p}_\loc(\wedge^k T^* \Omega)$ are defined accordingly.

If $\Upsilon \subset \R^n$ is another domain and $f \colon \Omega \to \Upsilon$ is bilipschitz, then the Lusin $(N^{-1})$-property of $f$ along with the almost everywhere differentiability of $f$ yields a measurable pull-back $f^* \omega$ on $\Omega$ for any measurable $k$-form $\omega$ on $\Upsilon$. We recall the following chain rule for such pull-backs, sketching its proof for the convenience of the reader (See also \cite[Lemma 2.22]{Donaldson-Sullivan_Acta} and \cite[Lemma 3.6]{Iwaniec-Martin_Acta}). 

\begin{lemma}\label{lem:bilip_chain_rule}
	Let $\Omega, \Upsilon \subset \R^n$ be connected open domains, let $f \colon \Omega \to \Upsilon$ be bilipschitz and let $p \in [1, \infty]$. If $\omega \in W^{d, p}(\wedge^k T^*\Upsilon)$, then $f^* \omega \in W^{d, p}(\wedge^k T^* \Omega)$ and $d f^* \omega = f^* d \omega$.
\end{lemma}
\begin{proof}
	Let $L$ be the bilipschitz constant of $f$. If $\omega \in L^p(\wedge^k T^*\Upsilon)$, then $f^* \omega \in L^p(\wedge^k T^*\Omega)$ for $p \in [1, \infty)$ by the following standard change of variables argument:
	\begin{equation}\label{eq:bilip_Lp_estimate}
		\int_\Omega \abs{f^*\omega}^p \leq \int_\Omega (\abs{\omega}^p \circ f) \abs{Df}^{kp}
		\leq L^{n+kp} \int_\Omega (\abs{\omega}^p \circ f) J_f
		= L^{n+kp} \int_{f(\Omega)} \abs{\omega}^p.
	\end{equation}
	The same holds for $p = \infty$ by using $\abs{f^*\omega} \leq L^k (\abs{\omega} \circ f)$ and the Lusin properties of $f$.
	
	Checking that $d f^* \omega = f^* d \omega$ weakly is also a standard approximation argument. Indeed, one selects $\Omega'$ compactly contained in $\Omega$, chooses a compactly supported $\omega' \in W^{d,p}_0(\wedge^k T^* \R^n)$ that coincides with $\omega$ on $f(\Omega')$ and mollifies both $f$ and $\omega'$ to obtain $f_i \in C^\infty(\Omega', \R^n)$ and $\omega_i \in C^\infty_0(\wedge^k T^* \R^n)$. We have $d f_i^* \omega_j = f_i^* d \omega_j$. Since $\abs{f_i - f} \to 0$ uniformly, $\abs{Df_i} \leq L$ a.e.\ on $\Omega'$, $Df_i \to Df$ pointwise a.e.\ on $\Omega$ and each $\omega_j$ and $d\omega_j$ is uniformly continuous, dominated convergence yields that $f_i^*\omega_j \to f^* \omega_j$ and $f_i^* d\omega_j \to f^* d\omega_j$ in the $L^1$-norm on $\Omega'$. Hence, it follows that $f^* d \omega_j = d f^* \omega_j$. Then since $\omega_j \to \omega'$ and $d\omega_j \to d\omega'$ in the $L^1$-norm, \eqref{eq:bilip_Lp_estimate} yields that $f^* \omega_j \to f^* \omega' = f^*\omega$ and $f^* d\omega_j \to f^* d\omega' = f^* d\omega$ in the $L^1$-norm on $\Omega'$. Hence, $d f^* \omega = f^* d \omega$ weakly on $\Omega'$. This generalizes to $\Omega$ by the locality of weak differentials.
\end{proof}

For an open $U$ in an oriented Lipschitz Riemannian manifold, we define that $d\omega \in L^1_\loc(\wedge^{k+1} T^*U)$ is a weak differential of $\omega \in L^1_\loc(\wedge^{k} T^*U)$ if $\phi^* d\omega$ is a weak differential of $\phi^* \omega$, for every bilipschitz chart $\phi$ with $U_\phi \subset U$. This lets us define Sobolev spaces $W^{d,p}(\wedge^k T^* U)$ and $W^{d,p}_\loc(\wedge^k T^* U)$. The space $W^{d,p}(\wedge^k T^* U)$ has the norm $\norm{\omega}_{W^{d,p}} = \norm{\omega}_{L^p} + \norm{d\omega}_{L^p}$. It can be shown that $d\omega$ is a weak differential of $\omega \in L^1_\loc(\wedge^{k} T^*U)$ precisely if it satisfies \eqref{eq:smooth_test_form_for_weak_d} for test forms $\eta \in W^{d, \infty}(\wedge^{n-k-1} T^* U)$ with $\spt \eta \subset U$ compact.

We also state the following approximation results.

\begin{lemma}\label{lem:W_infinity_approximation}
	Let $U$ be an open set in an oriented Lipschitz Riemannian manifold $\cR$ and let $p \in [1, \infty)$.  The space $W^{d,\infty}(\wedge^k T^* U)$ is dense in both $L^p(\wedge^k T^* U)$ and $W^{d,p}(\wedge^k T^* U)$.
\end{lemma}
\begin{proof}
	We sketch the case $\omega \in W^{d,p}(\wedge^k T^* U)$, as the proof of the other case is analogous. Let $\eps > 0$. We take a countable locally finite cover of $U$ with charts $\phi_i \colon \Omega_i \to U_i$, $i \in \bN$ and use a subordinate $W^{d, \infty}$-partition of unity to find $\omega_i \in W^{d,p}(\wedge^k T^* U)$ such that $\omega = \sum_i \omega_i$ and $\spt \omega_i \subset U_i$. We take a smooth convolution approximation $\eta_i$ of $\phi_i^* \omega_i$ such that $\norm{\eta_i - \phi_i^* \omega_i}_{W^{d,p}} <  C_i^{-(k+1)p-n} 2^{-i} \eps$, where $C_i$ is the constant such that $C_i^{-1} \abs{v} \leq \sqrt{(\phi_i^* g)_x(v,v)} \leq  C_i\abs{v}$ for every $v \in \R^n$ and a.e.\ $x \in \Omega_i$. The extension of $(\phi_i^{-1})^* \eta_i$ by zero defines a $W^{d,\infty}$-form on $U$ such that $\norm{\omega_i - (\phi_i^{-1})^* \eta_i}_{W^{d,p}} < 2^{-i} \eps$. The sum of these forms gives the desired approximation of $\omega$.
\end{proof}

\subsection{Lipschitz domains and submanifolds}

A connected, closed region $M \subset \cR$ in a Lipschitz $n$-manifold $\cR$ is called a \emph{(weakly) Lipschitz domain} if it satisfies the following condition: For any $x \in \partial M$, there exists a bilipschitz chart $\phi \in \cA$ such that $x \in U_\phi$, $\phi (\H^n_+ \cap \Omega_\phi) = \intr(M) \cap U_\phi$, $\phi (\H^n_- \cap \Omega_\phi) = (\cR \setminus M) \cap U_\phi$ and $\phi (\partial\H^n_+ \cap \Omega_\phi) = \partial M \cap U_\phi$. Here, $\H^n_+ = \{x \in \R^n : x_n > 0\}$ and $\H^n_- = \{x \in \R^n : x_n < 0\}$ are half-spaces of $\R^n$. A chart $\phi$ as above is denoted a \emph{bilipschitz boundary chart of $M$}. Note that we allow for a Lipschitz domain $M$ to be non-compact. More generally, a \emph{closed Lipschitz $n$-submanifold} $M$ of $\cR$ is a locally finite disjoint union of Lipschitz domains in $\cR$; this is equivalent to the definition of a Lipschitz domain without the connectedness assumption. If a closed Lipschitz $n$-submanifold is compact, we refer to it as a \emph{compact Lipschitz $n$-submanifold}.

The boundary $\partial M$ of a closed Lipschitz $n$-submanifold $M$ is a Lipschitz $(n-1)$-manifold, as the restrictions of $\phi \in \cA$ to $\partial \H^n_+ \cong \R^{n-1}$ provide a Lipschitz atlas $\cA_{\partial M}$. For subsets $A \subset \partial M$, we use $\intr(A)$ and $\partial A$ to refer to the interior and boundary, respectively, in $\partial M$.

We can define $L^p$ and $L^p_\loc$-spaces of differential forms on a closed Lipschitz $n$-submanifold $M$ in the usual way. We consequently define Sobolev spaces $W^{d,p}(\wedge^k T^* M)$ by requiring that $d\omega \in L^p(\wedge^{k+1} T^* M)$ is a weak differential of $\omega \in L^p(\wedge^k T^* M)$ in $\intr(M)$ and also $W^{d,p}_\loc$-spaces with a similar definition. For the rest of this section, we let $M$ be a closed Lipschitz $n$-submanifold in $\cR$, $U'$ be an open subset of $\cR$ and $U = U' \cap M$.

Suppose that $D$ is a closed Lipschitz $(n-1)$-submanifold in $\partial M$. Let $\omega \in W^{d, 1}_\loc(\wedge^k T^* U)$. \emph{The tangential part $\omega_T$ of $\omega$ vanishes weakly on $D$} if the extension of $d\omega$ by zero to $(U' \setminus \partial M) \cup (U \cap \intr(D))$ is the weak differential of the extension of $\omega$ by zero to $(U' \setminus \partial M) \cup (U \cap \intr(D))$. This is only dependent on $U$, $M$ and $D$ and not the choice of $U'$. The space of $\omega \in W^{d,p}(\wedge^k T^* U)$ with a weakly vanishing $\omega_T$ on $D$ is denoted $W^{d,p}_{T(D)}(\wedge^k T^* U)$ and a similar notation $W^{d,p}_{T(D), \loc}(\wedge^k T^* U)$ is used for the locally integrable version.

The Lipschitz Riemannian metric $g$ on $\cR$ defines a Hodge star operator $\hodge_g$ on $\cR$. The weak codifferential is defined by $d^* \omega = (-1)^{n(k-1) + 1} \hodge_g d \hodge_g \omega$, whenever $\omega \in L^1_\loc(\wedge^k T^* U)$ is such that $\hodge_g \omega$ is weakly differentiable. The space of all forms $\omega \in L^p(\wedge^k T^* U)$ with a weak codifferential $d^* \omega \in L^p(\wedge^{k-1} T^* U)$ is denoted $W^{d^*,p}(\wedge^k T^* U)$ and locally integrable versions of the spaces are defined accordingly. 

For $\omega \in W^{d^*,1}_\loc(\wedge^k T^* U)$, the \emph{normal part $\omega_N$ of $\omega$ vanishes weakly on $D$} if the tangential part $(\hodge_g\omega)_T$ of $\hodge_g \omega$ vanishes weakly on $D$. In particular, this occurs if the extension of $d^* \omega$ by zero to $(U' \setminus \partial M) \cup (U \cap \intr(D))$ is the weak codifferential of the extension of $\omega$ by zero to $(U' \setminus \partial M) \cup (U \cap \intr(D))$. We denote the space of Sobolev $k$-forms $\omega \in W^{d^*,p}(\wedge^k T^* U)$ with weakly vanishing normal part on $D$ by $W^{d^*,p}_{N(D)}(\wedge^k T^* U)$ and a corresponding locally integrable version $W^{d^*,p}_{N(D), \loc}(\wedge^k T^* U)$ is also defined.

We note that the spaces $W^{d,p}_{T(D)}(\wedge^k T^* U)$ and $W^{d^*, p}_{N(D)}(\wedge^k T^* U)$ are complete. Indeed, this is a direct consequence of the completeness of $W^{d,p}(\wedge^k T^* U)$ and $W^{d^*, p}(\wedge^k T^* U)$, as taking extensions by zero past $\intr(D)$ has no effect on $L^p$-convergence.

We prove a version of Lemma \ref{lem:W_infinity_approximation} for $W^{d,p}_{T(D)}(\wedge^k T^* M)$. For smooth domains $M$ in smooth manifolds $\cR$, the case $D = \partial M$ can be found in \cite[Corollary 3.8]{Iwaniec-Scott-Stroffolini}. Our proof is somewhat similar to theirs, except the handling of charts that meet $\partial D$ requires some technical modifications.

\begin{lemma}\label{lem:W_infinity_approx_boundary}
	Let $M$ be a closed Lipschitz $n$-submanifold in an oriented Lipschitz Riemannian $n$-manifold $\cR$ without boundary, let $U \subset M$ be open in $M$ and let $D$ be a closed Lipschitz $(n-1)$-submanifold in $\partial M$. If $p \in [1, \infty)$, then every $\omega \in W^{d,p}_{T(D)}(\wedge^k T^* U)$ can be approximated in $\norm{\cdot}_{W^{d,p}}$-norm by $\omega_j \in W^{d,\infty}_{T(D)}(\wedge^k T^* U)$ with $\spt \omega_j \subset U \setminus D$.
\end{lemma}

Before we begin the proof, we recall the following reflection property of $W^{d,p}$-forms in the Euclidean setting. It is discussed in \cite[p.48, 50]{Iwaniec-Scott-Stroffolini}. 

\begin{lemma}\label{lem:reflection_in_Rn}
	Let $B$ be the unit ball $B^n(0, 1)$ in $\R^n$ and suppose that $\omega \in W^{d,p}(\wedge^k T^* (B \cap \H^n_+))$ for some $p \in [1, \infty]$. Extend $\omega$ to $B \cap \H^n_-$ by $\Psi^*(\omega) = -\omega$, where $\Psi \colon \R^n \to \R^n$ is the reflection map defined by $\Psi(x_1, \dots, x_n) = \Psi(x_1, \dots, x_{n-1}, -x_n)$. Then the extended $\omega$ is in $W^{d, p}(\wedge^k T^* B)$.
\end{lemma}

\begin{proof}[Proof of Lemma \ref{lem:W_infinity_approx_boundary}]
	Let $B$ be the unit ball $B^n(0, 1)$ in $\R^n$. We can take a locally finite cover of $U$ with bilipschitz charts which are of the following three types:
	\begin{enumerate}
		\item\label{enum:case_interior} $\phi_i \colon B \to U_i$, where $U_i \subset \intr(M)$;
		\item\label{enum:case_boundary_nonvanishing} $\phi_i \colon B \to U_i$, where $\phi_i(B \cap \H^n_+) \subset \intr M$, $\phi_i(B \cap \H^n_-) \subset \cR \setminus M$ and $\phi_i(B \cap \partial \H^n_+) \subset \intr D$;
		\item\label{enum:case_boundary_vanishing} $\phi_i \colon B \to U_i$, where $\phi_i(B \cap \H^n_+) \subset \intr M$, $\phi_i(B \cap \H^n_-) \subset \cR \setminus M$ and $\phi_i(B \cap \partial \H^n_+) \subset \partial M \setminus D$;
		\item\label{enum:case_boundary_mixed} $\phi_i \colon B \to U_i$, where 
		\begin{itemize}
			\item $\phi_i(B \cap \H^n_+) \subset \intr M$,
			\item $\phi_i(B \cap \H^n_-) \subset \cR \setminus M$,
			\item $\phi_i(B \cap \{x \in \R^n : x_n = 0, x_1 > 0 \}) \subset \intr D$,
			\item $\phi_i(B \cap \{x \in \R^n : x_n = 0, x_1 < 0 \}) \subset \partial M \setminus D$,
			\item and $\phi_i(B \cap \{x \in \R^n : x_n = 0, x_1 = 0 \}) \subset \partial D$.
		\end{itemize}
	\end{enumerate}
	Similarly as in the proof of Lemma \ref{lem:W_infinity_approximation}, we can use a $W^{d, \infty}$ partition of unity to reduce the question into the case where $\spt \omega$ is compact and contained in a single chart neighborhood $U_i$. We only present the last case \eqref{enum:case_boundary_mixed}, as it demonstrates all the required techniques and the other cases can be shown by leaving out parts of the argument.
	
	Suppose we are in the last case. We use another bilipschitz chart to further modify our domain. Define $\phi_i' \colon B \to U_i$ by
	\begin{itemize}
		\item $\phi_i'(B \cap \{x \in \R^n : x_n \geq \max(0, x_1)\}) = M \cap U_i$,
		\item $\phi_i'(B \cap \{x \in \R^n : x_n = 0, x_1 < 0 \}) =  U_i \cap (\partial M \setminus D)$,
		\item $\phi_i'(B \cap \{x \in \R^n : x_n = 0, x_1 = 0 \}) = U_i \cap \partial D$,
		\item $\phi_i'(B \cap \{x \in \R^n : x_n = x_1, x_1 > 0\}) = U_i \cap \intr D$.
	\end{itemize}
	We extend $\omega$ by zero outside $M$ and let $\omega' = (\phi_i')^* \omega$ on $B \cap \H^n_+$. By the weak vanishing of $\omega_T$ on $D$, we have that this extension by zero is in $W^{d,p}(\wedge^k T^* \phi_i(B \cap \H^n_+))$. Consequently, $d\omega' = (\phi_i')^* d\omega$ weakly, $\omega' \in W^{d,p}(\wedge^k T^* (B \cap \H^n_+))$ and $\omega' \equiv 0$ on $\{x \in B\cap \H^n_+ : x_n < x_1\}$.
	
	We then extend $\omega'$ to $B$ via the reflection Lemma \ref{lem:reflection_in_Rn}. The form $\omega'$ is compactly supported in $W^{d, p}(\wedge^k T^* B)$ and $\omega'$ vanishes on the set of points $\{x \in B : \abs{x_n} < x_1\}$. Let $\tau_t \colon \R^n \to \R^n$ denote the translation $x \mapsto (x_1 - t, x_2, x_3, \dots, x_n)$. If $j > 0$, then for small enough $t$, $\tau_t^*\omega'$ and $\tau_t^*d\omega'$ are well-defined compactly supported forms in $B^n(0, 1)$ and
	\[
		\int_{B^n(0, 1)} (\abs{\omega' - \tau_t^* \omega'}^p + \abs{d\omega' - d\tau_t^* \omega'}^p) < \eps_j.
	\]
	Moreover, the translated form $\tau_t^* \omega'$ vanishes in a neighborhood of the set $B \cap \{x : \abs{x_n} \leq x_1\}$. Hence, this remains true for convolutions with a small enough mollifying kernel and we may use a convolution to pick a smooth approximation $\omega_j' \in C^\infty_0(\wedge^k T^* (B \setminus \{x : \abs{x_n} \leq x_1\}))$ with $\smallnorm{\omega_j' - \tau_t^* \omega'}_p + \smallnorm{d\omega_j' - d\tau_t^* \omega'}_p < \eps_j$. Then $((\phi_i')^{-1})^* \omega_j'$ yields the desired approximation in the chart $U_i$, concluding the proof of this case. 
\end{proof}

\begin{remark}
	Note that the same argument as above can be used in the smooth manifold setting to prove a smooth approximation version of Lemma \ref{lem:W_infinity_approx_boundary}. The only change one has to do is to replace the condition $x_n \geq x_1$ in the definition of $\phi_i'$ by a condition of the form $x_n \geq f(x_1)$, where $f$ is a smooth non-decreasing function such that $f(t) = 0$ if $t \leq 0$ and $f(t) > 0$ if $t > 0$.
\end{remark}

We end this section with a lemma about wedge products and vanishing tangential parts, the proof of which takes advantage of Lemma \ref{lem:W_infinity_approx_boundary}.

\begin{lemma}\label{lem:mixedboundarydata}
	Let $M$ be a closed Lipschitz $n$-submanifold in an oriented Lipschitz Riemannian $n$-manifold $\cR$ without boundary, let $p, q, r \in [1, \infty]$ with $p^{-1} + q^{-1} \geq r^{-1}$ and let $U \subset M$ be open in $M$. Let $D \subset \partial M$ be a closed Lipschitz $(n-1)$-submanifold and let $E = \partial M \setminus\intr D$.  If $\omega \in W^{d,p}_{T(D)}(\wedge^kT^*U)$ and $\tau \in W^{d,q}_{T(E)}(\wedge^{l}T^*U)$, where $0 \le l+k \leq n$, then $\omega \wedge \tau \in W^{d,r}_{T(\partial M)}(\wedge^{l+k}T^*U)$.
\end{lemma}
\begin{proof}
	By standard properties of wedge products of weakly differentiable forms, we have $\omega \wedge \tau \in W^{d, r}(\wedge^{k+l} T^* U)$ with $d(\omega \wedge \tau) = d\omega \wedge \tau + (-1)^k \omega \wedge d\tau$. By Lemma \ref{lem:W_infinity_approx_boundary}, we may approximate $\omega$ by $\omega_i \in W^{d,\infty}_{T(D)}(\wedge^kT^*U)$ such that $\spt \omega_i \subset M \setminus D$ and $\tau$ by $\tau_i \in W^{d,\infty}_{T(E)}(\wedge^l T^*U)$ such that $\spt \tau_i \subset M \setminus E$. Consequently, $\spt(\omega_i \wedge \tau_i)$ does not meet $\partial M$ and hence  $\omega_i \wedge \tau_i \in W^{d,\infty}_{T(\partial M)} (\wedge^{k+l} T^* U)$. Since also $\omega_i \wedge \tau_i \to \omega \wedge \tau$ in the $W^{d,r}$-norm on $U$, the claim follows by completeness of $W^{d,r}_{T(\partial M)}(\wedge^{k+l} T^* U)$. 
\end{proof}

\section{Lipschitz chains and integration}\label{sec:lipintegration}

Let $\cR$ be a Lipschitz $n$-manifold without boundary. A map $\sigma \colon \Delta_k \to \cR$ that is Lipschitz, where $\Delta_k \subset \R^k$ is the (closed) standard $k$-simplex, is called a \emph{singular Lipschitz $k$-simplex}. Given a field of coefficients $\K$, a \emph{Lipschitz $k$-chain on $\cR$ with coefficients in $\K$} is a formal sum $\sigma = k_1 \sigma_1 + \dots + k_m \sigma_m$ of singular Lipschitz $k$-simplices $\sigma_i$, where $k_i \in \K$. We denote the space of all such Lipschitz $k$-chains by $C^{\lip}_k(\cR; \K)$. If $U$ is a subset of $\cR$, then $C^{\lip}_k(U; \K)$ is the space of all $\sigma = k_1 \sigma_1 + \dots + k_m \sigma_m \in C^{\lip}_k(\cR; \K)$ such that the image of each $\sigma_i$ is contained in $U$. We also obtain a boundary map $\partial \colon C^{\lip}_{k+1}(U; \K) \to C^{\lip}_{k}(U; \K)$ in the usual way.

\subsection{Lipschitz chains in Euclidean domains}

We first go over the key properties of the space $C^{\lip}_k(\Omega; \R)$, where $\Omega \subset \R^n$ is a connected open domain. 

Suppose that $\sigma \colon \Delta_k \to \Omega$ is a singular Lipschitz $k$-simplex. Rademacher's theorem yields a derivative $D\sigma(x) \colon \R^k \to \R^n$ for almost every $x \in \Delta_k$. We may define a corresponding Borel measure $\nu_{\sigma}$ for $\sigma$ on $\Omega$ by
\begin{equation}\label{eq:nu_sigma_def}
	\int_\Omega \eta \dd\nu_{\sigma} = \int_{\Delta_k} (\eta \circ \sigma) \abs{J_{\sigma}} \vol_k.
\end{equation}
for any Borel measurable function $\eta \colon \Omega \to \R$. Here, the absolute value of the Jacobian $\abs{J_{\sigma}}$ is defined a.e.\ on $\Delta_k$ by
\begin{equation}\label{eq:Jacobian}
	\abs{J_{\sigma}}(x) = \sqrt{\det D^T\sigma(x) D\sigma(x)}.
\end{equation}
Moreover, we can extend this definition to Lipschitz $k$-chains $\sigma = k_1 \sigma_1 + \dots + k_m \sigma_m$ by $\nu_{\sigma} = k_1 \nu_{\sigma_1} + \dots + k_m \nu_{\sigma_m}$. Using the measure $\nu_{\sigma}$, we define the modulus of a family $\cS$ of Lipschitz $k$-chains, denoted 
\[
	\moddens_p(\cS) = \moddens_p\left(\left\{ \overline{\nu_{\sigma}} : \sigma \in \cS \right\}\right).
\]

Let $\sigma \colon \Delta_k \to \Omega$ be a singular Lipschitz $k$-simplex in $\Omega$ and let $\omega$ be a Borel differential $k$-form on $\Omega$ such that $\abs{\omega} \in L^1(\nu_\sigma)$. We may define the integral of $\omega$ over the surface $\sigma$. The pull-back $\sigma^* \omega$ is a well-defined measurable form since $\omega$ is Borel. Moreover, since $\sigma^* \omega$ is a top-dimensional form in $\Delta_k$, we have
\begin{align*}
	\abs{\sigma^* \omega_x}
	&\leq \abs{\omega_{\sigma(x)}} \abs{J_\sigma}
\end{align*}
and
\begin{equation}\label{eq:norm_integral_estimate}
	\int_{\Delta_k} \abs{\sigma^* \omega} \vol_k 
	\leq \int_{\Delta_k} (\abs{\omega} \circ \sigma) \abs{J_\sigma} \vol_k
	= \int_\Omega \abs{\omega} \dd \nu_S < \infty.
\end{equation}
Thus, $\sigma^* \omega$ is integrable. So 
\begin{equation}\label{eq:surface_integral_def}
	\int_{\sigma} \omega := \int_{\Delta_k} \sigma^* \omega.
\end{equation}
The definition extends to $\sigma \in C_k^{\lip}(\Omega; \R)$ linearly.

We crucially obtain a version of the Fuglede lemma for differential forms.

\begin{lemma}\label{lem:fuglede_for_forms_eucl}
	Let $\Omega \subset \R^n$ be a connected open domain and let $p \in [1, \infty)$. If $\omega \in L^p(\wedge^k T^*\Omega)$ and $\omega_i \to \omega$ in $L^p(\wedge^k T^*\Omega)$, then there exists a subsequence $\omega_{i_j}$ of $\omega_i$ such that 
	\[
		\int_\sigma \omega_{i_j} \to \int_\sigma \omega
	\]
	for $p$-a.e.\ $\sigma \in C_k^{\lip}(\Omega; \R)$.
\end{lemma}
\begin{proof}
	The forms $\omega$ and $\omega_i$ are integrable over $\sigma$ for $p$-a.e.\ $\sigma \in C_k^{\lip}(\Omega; \R)$. Lemma \ref{lem:Fuglede_properties}, \eqref{enum:fuglede_convergence} yields a subsequence $\omega_{i_j}$ such that
	\[
		\int_{\Omega} \abs{\omega_{i_j} - \omega} \dd \nu_\sigma \to 0
	\]
	for $p$-a.e.\ $\sigma \in C_k^{\lip}(\Omega; \R)$. It follows using \eqref{eq:norm_integral_estimate} and \eqref{eq:surface_integral_def} that
	\[
		\int_\sigma (\omega_{i_j} - \omega) \to 0
	\]
	for $p$-a.e.\ $\sigma \in C_k^{\lip}(\Omega; \R)$.
\end{proof}

\subsection{Wolfe's theorem} We have a reasonable integral for $\omega \in L^p(\wedge^k T^*\Omega)$ over $p$-a.e.\ Lipschitz $k$-chain $\sigma \in C_k^{\lip}(\Omega; \R)$. However, a much stronger result can be achieved for $\omega \in W^{d, \infty}(\wedge^k T^* \Omega)$. According to \emph{Wolfe's theorem}, there is a definition of an integral of $\omega \in W^{d, \infty}(\wedge^k T^* \Omega)$ over \emph{all} Lipschitz $k$-chains $\sigma \in C_k^{\lip}(\Omega; \R)$, despite the fact that $\omega$ is defined only up to a set of measure zero. Standard references for this can be found in the books of Whitney \cite[Chapter IX]{Whitney_book} and Federer \cite[4.1.19]{Federer_book}.  See also the work of Petit, Rajala and Wenger \cite{Petit-Rajala-Wenger_Sobolev-Wolfe}, which uses methods closer to the setting of this paper.

We state the precise properties of the integral we require. The following Proposition \ref{prop:wolfe_int_properties} and Lemma \ref{lem:bilip_change_of_vars} are reasonably immediate consequences from the results in \cite{Federer_book} and \cite{Whitney_book}. However, for the reader unfamiliar with this theory, we have given a more detailed account of the proof with precise literary references in the appendix of this paper.

\begin{prop}\label{prop:wolfe_int_properties}
	Let $\Omega \subset \R^n$ be a connected open domain. There exists a unique bilinear map
	\[
		W^{d, \infty}(\wedge^k T^* \Omega) \times C_k^{\lip}(\Omega; \R) \to \R, \quad (\omega, \sigma) \mapsto \int_\sigma \omega
	\]
	with the following properties.
	\begin{enumerate}
		\item \label{enum:wolfe_smoothforms} If $\omega \in W^{d, \infty}(\wedge^k T^* \Omega)$ is smooth, then
		\[
			\int_\sigma \omega = \int_{\Delta_k} \sigma^* \omega.
		\] 
		for all singular Lipschitz $k$-simplices $\sigma \colon \Delta_k \to \Omega$.
		\item \label{enum:wolfe_borelforms} If $\omega \in W^{d, \infty}(\wedge^k T^* \Omega)$ is Borel, then for every $p \in [1, \infty)$,
		\[
			\int_\sigma \omega = \int_{\Delta_k} \sigma^* \omega,
		\] 
		for $p$-a.e.\ singular Lipschitz $k$-simplices $\sigma \colon \Delta_k \to \Omega$.
		\item \label{enum:wolfe_stokes} For all $\sigma \in C_k^{\lip}(\Omega; \R)$ and $\omega \in W^{d, \infty}(\wedge^{k-1} T^* \Omega)$, 
		\[
			\int_{\partial \sigma} \omega = \int_{\sigma} d\omega.
		\]
		\item \label{enum:wolfe_convergence} For all $\sigma \in C_k^{\lip}(\Omega; \R)$ and $\omega, \omega_j \in W^{d, \infty}(\wedge^k T^* \Omega)$ where $j \in \N$, if $\omega_j \to \omega$ almost everywhere pointwise and $\sup_{j} \norm{\omega_j}_{W^{d,\infty}}  < \infty$, then 
		\[
			\lim_{j \to \infty} \int_\sigma \omega_j = \int_{\sigma} \omega.
		\]
		\item \label{enum:wolfe_continuity} If $\sigma \in C_k^{\lip}(\Omega; \R)$, $\tau \in C_{k+1}^{\lip}(\Omega; \R)$ and $\omega \in W^{d, \infty}(\wedge^k T^* \Omega)$, then
		\[
			\abs{\int_{\sigma + \partial \tau} \omega} \leq (\norm{\omega}_{L^\infty} + \norm{d\omega}_{L^\infty})(\nu_{\sigma}(\Omega) + \nu_\tau(\Omega)).
		\]
		\item \label{enum:wolfe_restriction} If $\Omega' \subset \R^n$ is another domain, $\omega \in W^{d, \infty}(\wedge^k T^* \Omega)$ and $\omega' \in W^{d, \infty}(\wedge^k T^* \Omega')$ are such that $\omega\vert_{\Omega \cap \Omega'} = \omega'\vert_{\Omega \cap \Omega'}$ and $\sigma \in C^{\lip}_k(\Omega \cap \Omega'; \R)$, then
		\[
			\int_{\sigma} \omega = \int_{\sigma} \omega'.
		\] 
	\end{enumerate}
\end{prop}

Suppose that $f \colon \Omega \to \Upsilon$ is a Lipschitz map between two Euclidean domains. The push-forward $f_* \sigma \in C_k^{\lip}(\Upsilon; \R)$ for every $\sigma \in C_k^{\lip}(\Omega; \R)$ is given by $f_* \sigma = \sigma \circ f$ for any Lipschitz $k$-simplex $\sigma$.  The definition is extended to chains linearly.

\begin{lemma}\label{lem:bilip_change_of_vars}
	Let $\Omega, \Upsilon \subset \R^n$ be connected open domains, let $f \colon \Omega \to \Upsilon$ be bilipschitz and let $p \in [1, \infty]$. If $\omega \in W^{d, p}(\wedge^k T^*\Upsilon)$, then 
	\[
		\int_\sigma f^*\omega = \int_{f_* \sigma} \omega.
	\]
	for $p$-a.e.\ $\sigma \in C^{\lip}_k(\Omega; \R)$ and for all $\sigma \in C^{\lip}_k(\Omega; \R)$ when $p = \infty$.
\end{lemma}

Finally, in our applications, we will also have to consider integration of $W^{d, \infty}$-forms over Lipschitz $k$-chains that meet the boundary of a Lipschitz domain $D$. The following property of the integral will play a significant role in making that boundary part of the theory work.

\begin{lemma}\label{lem:wolfe_int_one_half_dependence}
	Let $\Omega \subset \R^n$ be open and let $\omega \in W^{d,\infty}(\wedge^k T^* \Omega)$.  If $\omega \equiv 0$ on $\Omega \cap \H_+^n$, then the integral of $\omega$ over every $\sigma \in C_k^{\lip}(\Omega \cap \overline{\H^n_+}; \R)$ vanishes.
\end{lemma}
\begin{proof}
	The claim is clear for all $\sigma \in C_k^{\lip}(\Omega \cap \H^n_+; \R)$ by Proposition \ref{prop:wolfe_int_properties} part \eqref{enum:wolfe_restriction}, since $\omega$ vanishes in a neighborhood of the total image of $\sigma$. The part that requires attention is proving the result for a Lipschitz $k$-simplex $\sigma \colon \Delta_k \to \Omega \cap \overline{\H^n_+}$ that meets $\partial \H^n_+$. Let $\sigma$ be such a $k$-simplex.
	
	The proof is via a standard flat approximation argument. For every $\eps > 0$, we define a map $H_\eps \colon \Delta_k \times [0, \eps] \to \Omega \cap \overline{\H^n_+}$ by $H_\eps(x, t) = \sigma(x) + (0, 0, \dots, 0, t)$; for small enough $\eps > 0$ the image is indeed in $\Omega \cap \overline{\H^n_+}$. Let $\sigma_\eps(x) = H_\eps(x, \eps)$. The map $H_\eps$ can be understood as a Lipschitz $(k+1)$-chain; its total measure is $\nu_{H_\eps}(\Omega) = \eps \nu_{\sigma}(\Omega)$ and its boundary consists of $\sigma - \sigma_\eps$ and a part with total measure $\nu_{\partial H_\eps - (\sigma - \sigma_\eps)}(\Omega) = \eps \nu_{\partial \sigma}(\Omega)$. By Proposition \ref{prop:wolfe_int_properties} part \eqref{enum:wolfe_continuity}, we get
	\[
		\abs{\int_{\sigma - \sigma_\eps} \omega} \leq (\norm{\omega}_{L^\infty} + \norm{d\omega}_{L^\infty})(\eps \nu_{\sigma}(\Omega) + \eps \nu_{\partial\sigma}(\Omega))
		\xrightarrow[\eps \to 0]{} 0.
	\]
	On the other hand, $\omega$ vanishes in a neighborhood of $\sigma_\eps$, so
	\[
		\int_{\sigma - \sigma_\eps} \omega = \int_{\sigma} \omega,
	\]
	completing the proof.
\end{proof}

\subsection{Integration of $W^{d, \infty}$-forms in Lipschitz manifolds}

We now have the necessary preliminaries to discuss integration over Lipschitz $k$-chains on Lipschitz manifolds. Suppose that $\cR$ is a Lipschitz Riemannian manifold without boundary with metric $g$ and let $A \subset \cR$ be any subset of $\cR$. Since we have given a definition for Lipschitz maps $\sigma \colon \Delta_k \to \cR$ in Section \ref{subsect:lip_mfld_wo_bdry}, we obtain the spaces $C^{\lip}_k(A; \K)$ of Lipschitz chains on $A$ with $\K \in \{\R, \Z\}$. 

Suppose that $U \subset \cR$ is an open subset of $\cR$ and that $\omega \in W^{d, \infty}(U)$. If $\sigma \in C^{\lip}_k(U; \K)$, then we may subdivide $\sigma$ into a $\sigma' = k_1 \sigma_1 + \dots + k_i \sigma_i$ such that every $\sigma_j$ is a Lipschitz $k$-simplex contained in a single chart neighborhood $U_{\phi_j} \subset U$. Then 
\[
	\int_{\sigma} \omega = \sum_{j=1}^i k_j \int_{\phi_j^{-1} \circ \sigma_j} \phi_j^* \, \omega, 
\]
where we use Wolfe's canonical integral on the right hand side. By Lemma \ref{lem:bilip_change_of_vars}, the right hand side integrals are independent of the choice of chart $\phi_j$. By considering a common refinement of two subdivisions, we see that the integral depends only on $\sigma$.

Next, we consider integration on Lipschitz submanifolds. Let $M$ be a closed Lipschitz $n$-submanifold in $\cR$, let $U \subset M$ be open in $M$ and let $\omega \in W^{d,\infty}(\wedge^k T^* U)$. Our previous definition lets us integrate $\omega$ over all $\sigma \in C_k^{\lip}(U\cap\intr M; \R)$, but we wish to extend this to all $\sigma \in C_k^{\lip}(U; \R)$. By subdivision, it is again enough to consider integration over a single Lipschitz $k$-simplex $\sigma \colon \Delta_k \to U$ contained in a chart neighborhood $U_\phi$, where the remaining non-trivial case is that $\phi \colon B^n(0, 1) \to U_\phi$ maps $B^n(0, 1) \cap \overline{\H^n_+}$ into $M$ and $B^n(0, 1) \setminus \overline{\H^n_+}$ outside $M$.

The form $\phi^* \omega$ is in $W^{d,\infty}(\wedge^k T^* (B^n(0, 1) \cap \H^n_+))$. Hence, we can use the reflection Lemma \ref{lem:reflection_in_Rn} to extend $\phi^* \omega$ to a form $\alpha \in W^{d, \infty}(\wedge^k T^* B^n(0, 1))$. The integral of $\omega$ over $\sigma$ is the integral of $\alpha$ over $\phi^{-1} \circ \sigma$. In order to see that this is independent of our choices, we let $\psi$ be another such chart and let $\alpha'$ be the extension of $\psi^* \omega$ by reflection. Then $\alpha - (\psi^{-1} \circ \phi)^* \alpha'$ vanishes in the subset of $\H^n_+$ where it is defined. Lemma \ref{lem:wolfe_int_one_half_dependence} shows that the integral of $\alpha - (\psi^{-1} \circ \phi)^* \alpha'$ over $\phi^{-1} \circ \sigma$ vanishes and the same integral is obtained by using $\psi$ and $\alpha'$ instead of $\phi$ and $\alpha$. So the integral of $\omega \in W^{d,\infty}(\wedge^k T^* U)$ is well-defined over every $\sigma \in C_k^{\lip}(U; \R)$, when $U$ is an open subset of a closed Lipschitz $n$-submanifold $M$.

We also note that the definition extends to forms $\omega \in W^{d,\infty}_\loc(\wedge^k T^* U)$. This is since there exists an increasing sequence of sets $U_i \subset U$ such that $U_i$ are open in $M$, $\omega \in W^{d,\infty}(\wedge^k T^* U_i)$ and $U = \bigcup_i U_i$. Every Lipschitz $k$-simplex $\sigma \colon \Delta_k \to U$ is then contained in one of the sets $U_i$ by compactness.

By using bilipschitz charts, subdivision of Lipschitz $k$-simplices and Proposition \ref{prop:wolfe_int_properties} part \eqref{enum:wolfe_stokes}, we immediately have the following version of Stokes' theorem for the integral of $W^{d,\infty}_\loc$-forms.

\begin{prop}\label{prop:Lipschitz_stokes}
	Let $M$ be a closed Lipschitz $n$-submanifold in an oriented Lipschitz Riemannian $n$-manifold $\cR$, let $U \subset M$ be open in $M$, $\sigma \in C_k^{\lip}(U; \R)$ be a singular Lipschitz $k$-chain and $\omega \in W^{d,\infty}_\loc(\wedge^{k-1} T^* U)$. Then 
	\[
		\int_{\partial \sigma} \omega = \int_{\sigma} d\omega.
	\]
\end{prop}

We also require the following result for the canonical integral regarding boundary simplices.

\begin{lemma}\label{lem:canonical_int_vanishing_boundary_values}
	Suppose that $M$ is a closed Lipschitz $n$-submanifold in an oriented Lipschitz Riemannian $n$-manifold $\cR$ without boundary and that $U \subset M$ is an open subset of $M$. Let $D \neq \emptyset$ be a closed Lipschitz $(n-1)$-submanifold in $\partial M$ and let $\omega \in W^{d,\infty}_{T(D), \loc}(\wedge^k T^* U)$. Then
	\[
		\int_\sigma \omega = 0
	\]
	for every $\sigma \in C^{\lip}_k(D \cap U; \R)$.
\end{lemma}
\begin{proof}
	Let $B = B^n(0, 1)$. By subdivision, we may assume that $\sigma$ is a Lipschitz $k$-simplex $\sigma \colon \Delta_k \to D$ that is contained in the image of a boundary chart $\phi \colon B \to V$ with $\phi(B \cap \H^n_+) \subset U$ and $\phi(B \cap \H^n_-) \subset \cR \setminus M$. We wish to show that the integral of $\phi^* \omega$ over $\sigma' = \phi^{-1} \circ \sigma$ vanishes. The case where $V \cap \partial M \subset D$ is simple, as we may just extend $\phi^* \omega$ by zero into $B \cap \H^n_-$ and use Lemma \ref{lem:wolfe_int_one_half_dependence}.
	
	The other case to consider is when $V$ meets $\partial D$, in which case we can set $\phi$ up so that $V \cap D = \phi(\{x \in B : x_n = 0, x_1 > 0\})$. Let $\tau_\eps \colon \R^n \to \R^n$ be the translation map by $(\eps, 0, \dots, 0)$. For small enough $\eps > 0$, $(\tau_\eps)_* \sigma'$ remains in $B$ and the integral of $\phi^* \omega$ over $(\tau_\eps)_* \sigma'$ again vanishes by a use of Lemma \ref{lem:wolfe_int_one_half_dependence}. Similarly as in the proof of Lemma \ref{lem:wolfe_int_one_half_dependence}, we can define $H_\eps \colon \Delta_k \times [0, \eps] \to B$ by $H_\eps(x, t) = (\tau_t)_* \sigma'(x)$ and we have by Proposition \ref{prop:wolfe_int_properties} part \eqref{enum:wolfe_continuity} that
	\begin{align*}
		\abs{\int_{\sigma' - (\tau_\eps)_* \sigma'} \phi^* \omega}
		&\leq (\norm{\phi^*\omega}_{L^\infty} + \norm{d\phi^*\omega}_{L^\infty})(\nu_{H_\eps}(B) + \nu_{\partial H_\eps - (\sigma' - (\tau_\eps)_* \sigma')}(B))\\
		&\leq (\norm{\phi^*\omega}_{L^\infty} + \norm{d\phi^*\omega}_{L^\infty})(\eps \nu_{\sigma'}(B) + \eps \nu_{\partial\sigma'}(B))
		\xrightarrow[\eps \to 0]{} 0.
	\end{align*}
	We conclude that the integral of $\phi^* \omega$ over $\sigma'$ vanishes.
\end{proof}

\subsection{Integration of $L^p$-forms in Lipschitz manifolds}\label{sec:integrationlpforms}

Next, we consider integration of $\omega \in L^p_\loc(\wedge^k T^* \cR)$ over Lipschitz $k$-chains on a Lipschitz Riemannian manifold $\cR$. For this, we need to first define $p$-exceptional families of surfaces. The main technical complication is that you cannot use \eqref{eq:Jacobian} to define $\abs{J_\sigma}$ for every Lipschitz $\sigma \colon \Delta_k \to \cR$ due to $g$ only being defined almost everywhere.

We approach this as follows. Let $\sigma \colon \Delta_k \to \cR$ be Lipschitz. We can choose a subdivision of $\Delta_k$ into finitely many $\Delta_{k,i} \subset \Delta_k$ such that $\sigma(\Delta_{k,i})$ is contained in the image $U_i$ of a bilipschitz chart $\phi_i \colon \Omega_i \to U_i$ for every $i$. Now, for every $i$, the map $\sigma_i = \phi_i^{-1} \circ \sigma \vert_{\Delta_{k,i}}$ is Lipschitz and we can  define a measure $\nu_{\sigma_i}$ on $\Omega_i$ with respect to the Euclidean metric. We then define
\begin{equation}\label{eq:mu_sigma}
	\mu_\sigma = \sum_i (\phi_i)_* \nu_{\sigma_i}.  
\end{equation}

Once we have fixed a measure $\mu_\sigma$ for every Lipschitz $k$-simplex, we can similarly define measures $\mu_\sigma$ for every formal linear combination $\sigma  = c_1 \sigma_1 + \dots + c_k \sigma_k \in C_k^{\lip}(\cR; \R)$ of Lipschitz $k$-simplices $\sigma_i$. We note that the measures $\mu_\sigma$ depend on the choices of subdivisions $\Delta_{k,i}$ and charts $\phi_i$ for each $\sigma$. However, we still have the following.

\begin{lemma}\label{lem:well_def_exc_sets}
	Suppose $\cR$ is a Lipschitz Riemannian manifold and that $\sigma \colon \Delta_k \to \cR$ is a Lipschitz $k$-simplex. Suppose that $\mu_\sigma$ and $\mu_\sigma'$ are both defined using \eqref{eq:mu_sigma}, with different subdivisions and bilipschitz charts. There exists $C \in (0, \infty)$ such that $C^{-1} \mu_\sigma' \leq \mu_\sigma \leq C \mu_\sigma'$. Moreover, if $\cR$ is a smooth Riemannian manifold and if we define $\nu_\sigma$ using \eqref{eq:nu_sigma_def} and \eqref{eq:Jacobian} where $D^T\sigma(x)$ is the adjoint satisfying $g(D\sigma(x) v, w) = g_{\text{eucl}}(v, D^T \sigma(x) w)$, then there exists a constant $C \in (0, \infty)$ such that $C^{-1} \nu_\sigma \leq \mu_\sigma \leq C \nu_\sigma$.
\end{lemma}

\begin{proof}
	By a finite subdivision and use of bilipschitz charts, the claims reduce to the fact that if $\sigma' \colon \Delta_k \to \Omega \subset \R^n$ is Lipschitz and $f \colon (\Omega, g_{\text{eucl}}) \to (\Omega', g) \subset (\R^n, g)$ is bilipschitz, where the Riemannian metric $g$ in the target is smooth, then $C^{-1}\abs{J_{f \circ \sigma'}} \leq \abs{J_{\sigma'}} \leq C \abs{J_{f \circ \sigma'}}$ a.e.\ in $\Delta_k$. For a simple proof of this fact, let $x$ be a point in the interior of $\Delta_k$ such that $f \circ \sigma'$ and $\sigma'$ are differentiable at $x$ and use the singular value decomposition of matrices to select an orthonormal basis $v_1, \dots, v_k$ of $\R^k$ such that the image vectors $w_i = D(f \circ \sigma')(x) v_i$ are $g$-orthogonal. Now, the bilipschitz condition of $f$ and Hadamard's inequality yield that
	\begin{align*}
		\abs{J_{f \circ \sigma'}} &= \prod_{i=1}^k \abs{w_i}_g\\
		&= \prod_{i=1}^k \lim_{t \to 0} \frac{d_g(f \circ \sigma'(x + tv_i), f \circ \sigma'(x))}{t}\\
		&\geq \prod_{i=1}^k C^{-1} \lim_{t \to 0} \frac{\abs{\sigma'(x + tv_i)- \sigma'(x)}}{t}\\
		&= C^{-k} \prod_{i=1}^k \abs{D\sigma'(x) v_i} \\
		&\geq C^{-k} \abs{J_{\sigma'}(x)}.
	\end{align*}
	The other direction of the estimate is analogously obtained by using the singular value decomposition theorem on $D\sigma'(x)$ instead.
\end{proof}

In particular, Lemma \ref{lem:well_def_exc_sets} states that different choices in the construction of the measures $\mu_\sigma$ lead to families $\cM = \{\mu_\sigma : \sigma \in C_k^{\lip}(\cR; \R)\}$ and $\cM' = \{\mu_\sigma' : \sigma \in C_k^{\lip}(\cR; \R)\}$ such that we have weak minorizations $\cM' \lesssim \cM \lesssim \cM'$. Therefore, these families have the same $p$-exceptional sets by Corollary \ref{cor:weak_minorization_p-exc}. These exceptional sets are the same as the exceptional sets of $\{\nu_\sigma : \sigma \in C_k^{\lip}(\cR; \R)\}$ in the smooth setting. Hence, we can define that a subset $\Sigma \subset C_k^{\lip}(\cR; \R)$ is \emph{$p$-exceptional} if the family of measures $\{\mu_\sigma : \sigma \in \Sigma\}$ is $p$-exceptional.

Before we proceed to define integration for $\omega \in L^p_\loc(\wedge^k T^* \cR)$, we prove one more lemma that will be repeatedly useful in following considerations. Let $\sigma \in C_k^{\lip}(\cR; \R)$, let $\phi \colon \Omega \to U$ be a chart and let $\sigma' \in C_k^{\lip}(\Omega; \R)$. We say that \emph{$\sigma$ covers $\sigma'$ via $\phi$} if there exists a constant $c > 0$ such that $\mu_\sigma \geq c \cdot \phi_* \nu_{\sigma'}$. By Lemma \ref{lem:well_def_exc_sets}, this is independent on the choice of $\mu_\sigma$.

\begin{lemma}\label{lem:covering_simplex}
	Suppose $\cR$ is a Lipschitz Riemannian manifold and $\phi \colon \Omega \to U$ is a chart.
	\begin{enumerate}
		\item\label{enum:cover_subdivision} Let $\sigma \in C_k^{\lip}(\cR; \R)$ and $\sigma' \colon \Delta_k \to \Omega$ be a Lipschitz $k$-simplex in $\Omega$. If $a (\phi \circ \sigma')$ (or any bilipschitz reparameterization) occurs in a subdivision of $\sigma$ with $a \in \R \setminus \{0\}$, then $\sigma$ covers $\sigma'$ via $\phi$.
		\item\label{enum:cover_families} Let $\Sigma' \subset C_k^{\lip}(\Omega; \R)$ and 
		\[
			\Sigma = \left\{ \sigma \in C_k^{\lip}(\cR; \R) : \sigma \text{ covers } \sigma' \text{ via } \phi \text{ for some } \sigma' \in \Sigma'\right\}.
		\]
		If $\Sigma'$ is $p$-exceptional, then $\Sigma$ is $p$-exceptional.
	\end{enumerate}
\end{lemma}
\begin{proof}
	If $a (\phi \circ \sigma')$ occurs in a subdivision of $\sigma$, then by further subdividing one can make it so that $\abs{a} \phi_* \nu_{\sigma'}$ is one of the measures added together to form $\mu_\sigma$.
	This gives the first property.
	
	For the second, since $\phi^* \vol_g$ is a.e.\ uniformly comparable with the Lebesgue volume, it follows that the family of measures $\{\phi_* \nu_{\sigma'} : \sigma' \in \Sigma'\}$ is $p$-exceptional in $(\cR, \vol_g)$. Therefore, the family $\{\mu_\sigma : \sigma \in \Sigma\}$ is also $p$-exceptional by Corollary \ref{cor:weak_minorization_p-exc}.
\end{proof}

Suppose that $\omega \in L^p_\loc(\wedge^k T^* \cR)$, with the aim of defining the integral of $\omega$ over $p$-a.e.\ $\sigma \in C_k^{\lip}(\cR; \R)$. We select a countable collection of bounded bilipschitz charts $\phi_i \colon \Omega_i \to U_i$ such that the sets $U_i$ cover $\cR$ and fix a Borel representative of $\phi_i^* \omega \in L^p(\wedge^k T^* \Omega_i)$ for each $i$. We choose the measures $\mu_\sigma$ for all $\sigma \in C_k^{\lip}(\cR; \R)$ using only the charts $\phi_i$. 

For every chart $\phi_i$, let $\Sigma_i' \subset C_k^{\lip}(\Omega_i; \R)$ be the collection of all Lipschitz $k$-chains $\sigma$ such that the integral of $\phi_i^* \omega$ over $\sigma$ is not defined. Similarly, for any two charts $\phi_i$ and $\phi_j$ such that $U_i \cap U_j \neq \emptyset$, we denote $\Omega_{ij} = \phi_i^{-1}(U_i \cap U_j)$ and let $\Sigma_{ij}'$ be the family of $\sigma \in C_k^{\lip}(\Omega_{ij}; \R)$ such that the integral of $\phi_i^*\omega$ over $\sigma$ and the integral of $\phi_j^*\omega$ over $(\phi_j^{-1} \circ \phi_i)_* \sigma$ are both well-defined, but different. The families $\Sigma'_i$ and $\Sigma'_{ij}$ essentially encompass all the simplices where integration behaves poorly and are both $p$-exceptional in $\Omega_i$, since $\phi_i^* \omega \in L^p(\wedge^k T^* \Omega_i)$.

Let $\Sigma$ be the family of all $\sigma \in C_k^{\lip}(\cR; \R)$ such that $\sigma$ covers a $\sigma' \in \Sigma'_i \cup \bigcup_j \Sigma'_{ij}$ via $\phi_i$ for some $i$. Since there are only a countable amount of charts, $\Sigma$ is $p$-exceptional by Lemma \ref{lem:covering_simplex}. Moreover, for all $k$-chains $\sigma \in C_k^{\lip}(\cR; \R) \setminus \Sigma$, Lemma \ref{lem:covering_simplex} also yields that no subdivision contains a simplex where the integral behaves poorly. So integration of $\omega$ over $\sigma$ is defined as usual via subdivision and charts and the result is independent of the exact subdivision used.

Overall, the set $\Sigma$ and the values of the integral of $\omega$ over $\sigma \in C_k^{\lip}(\cR; \R) \setminus \Sigma$ depend on the choices of charts $\phi_i$ and Borel representatives $\phi_i^* \omega$. However, it is reasonably straightforward to see that a different choice of $\phi_i$ and $\phi_i^* \omega$ will only change the integral of $\omega$ in a $p$-exceptional family. 

The definition extends to $\omega \in L^p(\wedge^k T^* U)$ with $U \subset \cR$ by extending $\omega$ as zero outside $U$. If instead $U \subset M$ is an open subset of a closed Lipschitz $n$-submanifold $M$ in $\cR$, then by a similar extension by zero to $U' \subset \cR$ with $U' \cap M = U$, the definition also extends to $\omega \in L^p(\wedge^k T^* U)$ and $\omega \in L^p_\loc(\wedge^k T^* U)$.

We next prove Lipschitz manifold versions of two Euclidean results. The first is a version of the differential forms Fuglede Lemma \ref{lem:fuglede_for_forms_eucl}.

\begin{lemma}\label{lem:fuglede_for_forms}
	Let $\cR$ be a Lipschitz Riemannian manifold without boundary and let $p \in [1, \infty)$. If $\omega \in L^p(\wedge^k T^* \cR)$ and $\omega_i \to \omega$ in $L^p(\wedge^k T^* \cR)$, then there exists a subsequence $\omega_{i_j}$ of $\omega_i$ such that 
	\[
		\int_\sigma \omega_{i_j} \to \int_\sigma \omega
	\]
	for $p$-a.e.\ $\sigma \in C^{\lip}_k(\cR; \R)$.
\end{lemma}
\begin{proof}
	Let $\{\phi_l : l = 1, \dots\}$ be a countable cover of $\cR$ by bilipschitz charts. By the Euclidean result from Lemma \ref{lem:fuglede_for_forms_eucl} and a standard diagonal subsequence argument, we may select $i_j$ such that
	\[
		\int_{\sigma'} \phi_l^* \omega_{i_j} \to \int_{\sigma'} \phi_l^*\omega
	\]
	for every $l$ and $p$-a.e.\ $\sigma' \in C_k^{\lip}(\Omega_l; \R)$. We denote the set of $\sigma' \in C_k^{\lip}(\Omega_l; \R)$ for which this convergence does not hold by $\Sigma'_l$. The claim then follows by Lemma \ref{lem:covering_simplex}, which shows the $p$-exceptionality of the family of all $\sigma \in C_k^{\lip}(\cR; \R)$ which cover an element $\sigma' \in \Sigma'_l$ via $\phi_l$ for some $l \in \bN$.
\end{proof}

Second, we give a Lipschitz version of Proposition \ref{prop:wolfe_int_properties} part \eqref{enum:wolfe_borelforms}, which connects our two integrals.

\begin{lemma}\label{lem:two_integrals}
	Let $\cR$ be a Lipschitz Riemannian manifold without boundary, $M$ be a closed Lipschitz $n$-submanifold in $\cR$, $U \subset M$ be open in $M$ and  $\omega \in W^{d,\infty}(\wedge^k T^* U)$. For every $p \in [1, \infty)$, outside a $p$-exceptional subset of $C^{\lip}_k(U; \R)$, the integral of $\omega$ as an $L^p$-form over $p$-a.e.\ $\sigma \in C^{\lip}_k(U; \R)$ coincides with the integral of $\omega$ as a $W^{d,\infty}$-form over all $\sigma \in C^{\lip}_k(U; \R)$.
\end{lemma}
\begin{proof}
	Suppose that $\phi$ is a chart such that $U_{\phi} \subset U \cap \intr (M)$. If we fix a Borel representative of $\phi^* \omega$ and let $\Sigma'$ be the set of $\sigma \in C^{\lip}_k(\Omega_i; \R)$ for which $\phi_i^* \omega$ does not satisfy the result, then $\Sigma'$ is $p$-exceptional by Proposition \ref{prop:wolfe_int_properties} part \eqref{enum:wolfe_borelforms}. 
	
	On the other hand, suppose that $\phi \colon \Omega_\phi \to U_\phi$ is a chart that takes $\Omega_\phi \cap \overline{\H^n_+}$ into $U$ and $\Omega_\phi \cap \partial \H^n_+$ into $U \cap \partial M$. Then the $W^{d,\infty}$-integral in $U_{\phi}$ is defined using the reflection extension of $\phi^* \omega$ outlined in Lemma \ref{lem:reflection_in_Rn} and the $L^p$-integral is defined using the extension of $\phi^* \omega$ by zero. These two extensions have Borel representatives $\alpha$ and $\beta$ that coincide in $\Omega_\phi \cap \overline{\H^n_+}$. In this case, $\alpha$ and $\beta$ have the exact same $L^p$-integral over every $\sigma \in C^{\lip}_k(\Omega_\phi \cap \overline{\H^n_+}; \R)$. If $\Sigma'$ is the set of $\sigma \in C^{\lip}_k(\Omega_\phi \cap \overline{\H^n_+}; \R)$ for which $\phi_i^* \omega$ does not satisfy the result, then the $L^p$-integral of $\alpha$ over every $\sigma \in \Sigma'$ differs from the corresponding $W^{d,\infty}$-integral. Similarly as before, $\Sigma'$ is $p$-exceptional by Proposition \ref{prop:wolfe_int_properties} part \eqref{enum:wolfe_borelforms}.
	
	The claim again follows by taking a countable cover of $U$ with such charts and by using Lemma \ref{lem:covering_simplex} on the corresponding sets $\Sigma'$ as in the previous lemma.
\end{proof}

To end this section, we comment on the measures $\nu_\sigma$. As mentioned earlier, if the Lipschitz Riemannian metric is only measurable, then we cannot define $\nu_\sigma$ canonically for every $\sigma \in C^{\lip}_k(M; \R)$. However, outside a $p$-exceptional family of Lipschitz $k$-chains with $p \in [1,\infty)$, the measure $\nu_\sigma$ can be defined. 

Indeed, in the following Lemma, we prove that one can define $\nu_\sigma$ outside a $p$-exceptional family of chains in a domain of $\R^n$ equipped with a bounded measurable Riemannian metric $g'$. Once this is done, if $(M, g)$ is a Lipschitz Riemannian manifold and $\phi \colon \Omega \to U$ is a chart, we can then define $\nu_\sigma$ for Lipschitz $k$-simplexes $\sigma \colon \Delta_k \to U$ by defining $\nu_{\phi^{-1} \circ \sigma}$ with respect to $\phi^* g$ and applying push-forward of measures with $\phi$. The definition for $p$-a.e.\ $\sigma \in C^{\lip}_k(M; \R)$ follows by a similar countable covering and subdivision argument as was used to define integration of $k$-forms over $p$-a.e.\ Lipschitz $k$-chain, including the use of Lemma \ref{lem:covering_simplex}. The measures $\nu_\sigma$ again depend on the choices made in the construction, but a different set of choices changes them only in a $p$-exceptional family. 

\begin{lemma}\label{lem:measurenotdefined}
	Let $U \subset \bR^n$ be an open set and let $g$ be a Lipschitz Riemannian metric on $U$.  There exists a collection of Lipschitz $k$-simplices $\cS$ so that $\nu_\sigma$ is defined by \eqref{eq:nu_sigma_def} and \eqref{eq:Jacobian} with respect to $g$ and whose complement in the set of all Lipschitz $k$-simplices is $p$-exceptional.
\end{lemma}
\begin{proof}
	Let $g_n$ be a sequence of smooth metrics that converge to $g$ pointwise almost everywhere and in $L^p$.  Let $\mu_\sigma$ be the surface measure of $\sigma$ with respect to the Euclidean metric on $U$.  By Lemma \ref{lem:Fuglede_properties}, outside a $p$-exceptional set of Lipschitz $k$-simplices,
	\begin{align*}
		\lim_{n \to \infty} \int_\sigma \|g_n -g \|_{\text{op}}d\mu_\sigma =0.
	\end{align*}
	Therefore, on $\sigma$, $g$ is a Lipschitz Riemannian metric outside a set of $0$ measure with respect to $\mu_\sigma$.  This allows us to define $\nu_\sigma$ with respect to $g$.
\end{proof}

\section{Sobolev cohomology theories}\label{sec:sobolevhomology}

We next discuss the relevant parts of Sobolev cohomology theory on Lipschitz manifolds.

\subsection{Local Sobolev cohomology}
Let $M$ be a closed Lipschitz $n$submanifold in an oriented Lipschitz Riemannian $n$-manifold $\cR$ without boundary, let $D$ be a closed Lipschitz $(n-1)$-submanifold in $\partial M$ and let $p \in [1, \infty]$. We note that the map $U \mapsto W^{d,p}_{T(D), \loc}(\wedge^k T^* U)$ is a sheaf on $M$. 

In order to define a cohomology theory in the Sobolev setting the only analytic tool needed is a local exactness result. This is provided by a Sobolev variant of the Poincar\'e lemma. A version in the Euclidean setting with $D = \emptyset$ can be found in \cite[Section 4]{Iwaniec-Lutoborski}. We give the version of the result that we require, relying heavily on the exposition in \cite[Section 4]{Iwaniec-Lutoborski} for the proof.

\begin{lemma}\label{lem:poincare_with_boundary_vanishing}
	Let $M$ be a closed Lipschitz $n$-submanifold in an oriented Lipschitz Riemannian $n$-manifold $\cR$ without boundary, let $p \in [1, \infty]$ and let $D$ be a closed Lipschitz $(n-1)$-submanifold in $\partial M$.
	Let $x \in M$, $k \geq 1$, $U$ a neighborhood of $x$ in $M$ and suppose that $\omega \in W^{d,p}_{T(D), \loc}(\wedge^k T^* U)$ is such that $d\omega = 0$.  There exists a neighborhood $V \subset U$ of $x$ and a form $\tau \in W^{d,p}_{T(D), \loc}(\wedge^{k-1} T^* V)$ such that $\omega \vert_V = d\tau$.
\end{lemma}
\begin{proof}
	We first recall an explicit construction for $\tau$ in the Euclidean setting. Let $B = B^n(0, 1)$. Fix $\eta \in C^\infty_0(B)$ such that $\eta \geq 0$ and $\int_B \eta = 1$ and define a linear operator $T \colon C^\infty(\wedge^k T^* B) \to C^\infty(\wedge^{k-1} T^* B)$ by
	\begin{equation}\label{eq:cartan_homotopy_op_def}
		(T\omega)_x(v_1, \dots, v_{k-1}) = \int_B \eta(y) \int_0^1 t^{k-1} \omega_{tx + (1-t)y}(x-y, v_1, \dots, v_{k-1}) \dd t \dd y.
	\end{equation}
	Then $\omega = d T(\omega) + T(d\omega)$ and $\norm{T\omega}_{L^p} \leq C_{\eta} \norm{\omega}_{L^p}$ for all $p \in [1, \infty]$ and all $\omega \in C^\infty(\wedge^k T^* B)$ (see \cite[(4.6) and (4.15)]{Iwaniec-Lutoborski}). If $p \in [1,\infty)$, then $C^\infty(\wedge^k T^* B)$ is dense in $L^p(\wedge^k T^* B)$ and hence $T$ extends to a bounded linear map $T \colon L^p(\wedge^k T^* B) \to L^p(\wedge^{k-1} T^* B)$. 
	Moreover, a similar extension remains valid for $p = \infty$ even though $C^\infty(\wedge^k T^* B)$ is not dense in $L^\infty(\wedge^k T^* B)$.
	Indeed, we may smoothly approximate $\omega \in L^\infty(\wedge^k T^* B)$ with $\omega_j \in C^\infty(\wedge^k T^* B)$ such that $\norm{\omega_j}_{L^\infty} \leq \norm{\omega}_{L^\infty}$ and $\norm{\omega_j - \omega}_{L^1} \to 0$.
	In which case, $T(\omega_j)$ converges pointwise to a unique $T(\omega) \in L^1(\wedge^k T^* B)$ and the uniform bound $\norm{T(\omega_j)}_{L^\infty} \leq C_{\eta} \norm{\omega}_{L^\infty}$ yields that $\norm{T(\omega)}_{L^\infty} \leq C_{\eta} \norm{\omega}_{L^\infty}$. 
	
	If $p \in [1, \infty)$, it follows by smooth approximation that $T(\omega) \in W^{d,p}(\wedge^{k-1} T^* B)$ and $\omega = d T(\omega) + T(d\omega)$ if $\omega \in W^{d,p}(\wedge^{k} T^* B)$ (see \cite[Lemma 4.2]{Iwaniec-Lutoborski}). The same is also true for $\omega \in W^{d,\infty}(\wedge^k T^* B)$, since $W^{d,\infty}(\wedge^k T^* B) \subset W^{d,1}(\wedge^k T^* B)$ implies that $d T(\omega) = \omega - T(d\omega)$ weakly and hence $\norm{d T(\omega)}_{L^\infty} \leq \norm{\omega}_{L^\infty} + \norm{T(d\omega)}_{L^\infty} < \infty$. In the Euclidean setting, if $\omega \in W^{d,p}(\wedge^k T^* B)$ weakly, then a choice of $\tau = T(\omega)$ yields $\tau \in W^{d,p}(\wedge^{k-1} T^* B)$ and $d\tau = \omega$ weakly.
	
	We return to the Lipschitz manifold setting. Suppose first that $x \in \intr(M)$.  We can choose a bilipschitz chart $\phi \colon B \to V$ such that $x \in V \subset U \cap \intr(M)$. The Euclidean result implies that $\phi^* \omega \in W^{d,p}(\wedge^k T^* B)$ and the desired $\tau \in W^{d,p}(\wedge^{k-1} T^* V)$ is given by $\phi^* \tau = T(\phi^* \omega)$.
	
	The next case we consider is when $x \in \partial M \setminus D$. There exists a bilipschitz chart $\phi \colon B \to V$ such that $x \in V$, $V \cap M \subset U$, $V \cap D = \emptyset$, $\phi(B \cap \overline{\H^n_+}) = V \cap M$ and $\phi (B \setminus \overline{\H^n_+}) = V \setminus M$. In this case, the reflection Lemma \ref{lem:reflection_in_Rn} extends $\phi^* \omega$ to $B$. A valid choice of $\tau$ is given by $\phi^* \tau = T(\phi^* \omega)$ on $B\cap\H^n_+$ and $\phi^* \tau = 0$ on $B\cap\H^n_-$.
	
	Next, we consider the case $x \in \intr D$. There exists a bilipschitz chart $\phi \colon B \to V$ such that $x \in V$, $V \cap M \subset U$, $V \cap \partial M \subset \intr D$, $\phi(B \cap \overline{\H^n_+}) = V \cap M$ and $\phi (B \setminus \overline{\H^n_+}) = V \setminus M$. If we extend $\phi^* \omega$ by zero to $B$, the extension is in $W^{d,p}(\wedge^k T^* B)$ with $\spt (\phi^* \omega) \subset \overline{\H^n_+}$. By selecting the $\eta$ in the definition of $T$ to satisfy $\spt \eta \subset B \cap \H^n_-$ and since $B \cap \H^n_-$ is convex, it follows from \eqref{eq:cartan_homotopy_op_def} that $T(\phi^* \omega) = 0$ on $B \cap \H^n_-$ (see also the comment before \cite[(4.14)]{Iwaniec-Lutoborski}). Thus, $\phi^* \tau = T(\phi^* \omega)$ remains a valid choice that vanishes weakly on $D$.
	
	We have one more remaining case, which is $x \in \partial D$. For this case, we select the bilipschitz chart $\phi \colon B \to V$, where $\phi(B \cap \{x \in \R^n : x_n = 0, x_1 < 0\}) = V \cap (\partial M \setminus D)$ and $\phi(B \cap \{x \in \R^n : x_n = x_1, x_1 > 0\}) = V \cap \intr D$. By extending $\phi^* \omega$ by zero to $B \cap \H^n_+$ and then further extending to $B$ by Lemma \ref{lem:reflection_in_Rn}, we obtain an extension $\phi^* \omega \in W^{d,p}(\wedge^k T^* B)$ that vanishes in $H = B \cap \{x \in \R^n : \abs{x_n} < x_1\}$. Choosing $\eta$ such that $\spt \eta \subset H$, $T(\phi^* \omega)$ vanishes on $H$ due to convexity of $H$. The desired $\tau$ is given by $\phi^*\tau = T(\phi^* \omega)$ in $B \cap \H^n_+$ and $\phi^* \tau = 0$ in $B \cap \H^n_-$.
\end{proof}

Let $H_{p}^k(M, D)$ denote the degree $k$ cohomology space of the complex $W^{d,p}_{T(D), \loc}(\wedge^* T^* M)$. We call $H_{p}^k(M, D)$ the degree $k$ \emph{relative $L^p_\loc$-cohomology of the pair $(M, D)$}. Due to the Poincar\'e lemma shown in Lemma \ref{lem:poincare_with_boundary_vanishing}, the methods of sheaf cohomology now apply to the spaces $H_{p}^k(M, D)$.  We begin with the following corollary of the Poincar\'e lemma. The proof assumes familiarity with sheaf theory and sheaf cohomology; for standard reference works on the topic, we refer the reader to the comprehensive treatment in \cite{Bredon_book}, as well as the more focused introductions in \cite[Chapter 5]{Warner_book} and \cite[Chapter II]{Wells_ComplexMfldAnalysis}.

\begin{cor}\label{cor:Lp_Linfty_cohom_eq}
	Let $M$ be a closed Lipschitz $n$-submanifold in an oriented Lipschitz Riemannian $n$-manifold $\cR$ without boundary, let $p \in [1, \infty)$ and let $D$ be a closed Lipschitz $(n-1)$-submanifold in $\partial M$. The relative $L^p_\loc$-cohomology spaces $H_{p}^k(M, D)$ and $H_{\infty}^k(M, D)$ are isomorphic as vector spaces, with the isomorphism induced by the inclusion maps $W^{d,\infty}_{T(D), \loc}(\wedge^k T^* M) \hookrightarrow W^{d,p}_{T(D), \loc}(\wedge^* T^k M)$.
\end{cor}

\begin{proof}
	Based on the Poincar\'e Lemma \ref{lem:poincare_with_boundary_vanishing}, we obtain that the graded differential sheaves $\cW_{p, D}^* \colon U \mapsto W^{d,p}_{T(D), \loc}(\wedge^* T^* U)$ and $\cW_{\infty, D}^* \colon U \mapsto W^{d,\infty}_{T(D), \loc}(\wedge^* T^* U)$ on $M$ each form a fine torsionless resolution of the sheaf $\cR_D$ of locally constant functions on $M$ which vanish on $D$. Since the inclusion maps $W^{d,\infty}_{T(D), \loc}(\wedge^k T^* M) \hookrightarrow W^{d,p}_{T(D), \loc}(\wedge^* T^k M)$ form a sheaf homomorphism of resolutions, they must induce an isomorphism in cohomology by \cite[II.4.2]{Bredon_book}.
\end{proof}

\subsection{Lipschitz homology and the de Rham theorem}

On smooth manifolds $M$, the de Rham theorem shows that the cohomology of the complex $C^\infty(\wedge^* T^* M)$ is dual to singular homology and the duality map is explicitly given by integration over smooth simplices. Since we can integrate $W^{d,\infty}_{T(D), \loc}(\wedge^* T^* M)$-forms over Lipschitz simplices, a similar explicit duality relation between Lipschitz homology and $L^\infty_\loc$-cohomology exists. 

\begin{prop}\label{prop:de_Rham_thm_in_Lip_setting}
	Let $M$ be a closed Lipschitz $n$-submanifold in an oriented Lipschitz Riemannian $n$-manifold $\cR$ without boundary, let $p \in [1, \infty]$ and let $D$ be a closed Lipschitz $(n-1)$-submanifold in $\partial M$. There is a canonical isomorphism
	\[
		H^k_\infty(M, D) \cong \hom (H_k^{\lip}(M, D; \Z), \R) \cong \hom (H_k^{\lip}(M, D; \R), \R),
	\]
	which, for $[\omega] \in H^k_\infty(M, D)$ and $[\sigma] \in H_k^{\lip}(M, D; \K)$ with $\K \in \{\Z, \R\}$, is given by
	\begin{equation}\label{eq:de_rham_isomorphism}
		[\omega]([\sigma]) = \int_\sigma \omega.
	\end{equation}
	Moreover, this isomorphism maps the wedge product to the cup product on the level of cohomology classes: if $E = \partial M \setminus D$, then $[\omega] \cup [\omega'] = [\omega \wedge \omega']$ for all $[\omega] \in H^k_\infty(M, D), [\omega'] \in H^l_\infty(M, E)$, where $[\omega \wedge \omega'] \in H^{k+l}_\infty(M, \partial M)$.
\end{prop}
The argument is essentially the standard sheaf-theoretic proof of de Rham -type theorems. For the convenience of the reader we provide a proof in the appendix. It is also notable that \eqref{eq:de_rham_isomorphism} also defines a chain map from $W^{d, \infty}_{T(D), \loc}(\wedge^* T^* M)$ to $\hom(C^{\lip}_*(M; \Z), \R)$, but this map does not preserve products: the equivalence of $\cup$ and $\wedge$ applies only for cohomology classes. 

\section{$p$-exceptional homology classes}\label{sec:pexceptional}

In this section we show that homology classes are not $p$-exceptional

\begin{prop}\label{prop:no_exc_hom_classes}
	Let $M$ be a closed Lipschitz $n$-submanifold in an oriented Lipschitz Riemannian $n$-manifold $\cR$ without boundary, let $p \in (1, \infty)$ and $k \in \{0, \dots, \infty\}$ and let $D$ be a closed Lipschitz $(n-1)$-submanifold in $\partial M$. If $[\sigma] \in H_k^{\lip}(M, D; \K)$ for $\K \in \{\Z, \R\}$, then $[\sigma]$ is not $p$-exceptional.
\end{prop}

We present the argument in increasingly general cases, building up to Proposition \ref{prop:no_exc_hom_classes}. The case involving relative homology presents significant technical difficulties, which make the originally simple argument far more difficult to follow. We start with the case where $M$ has no boundary, where the proof is straightforward.

\begin{lemma}\label{lem:no_exceptional_hom_classes_boundaryless}
	Let $M$ be an oriented Lipschitz Riemannian manifold without boundary and let $p \in (1, \infty)$, $k \in \{0, \dots, \infty\}$. If $[\sigma] \in H_k^{\lip}(M; \K)$, for $\K \in \{\Z, \R\}$, then $[\sigma]$ is not $p$-exceptional.
\end{lemma}
\begin{proof}[Idea of proof]
    Suppose that $c \in H_k^{\lip}(M; \K)$ is not $p$-exceptional. Then there is a function $\rho \in L^p(M)$ such that the integral of $\rho$ over the measure of every $\sigma \in c$ is infinite. The idea is to select an arbitrary chain $\sigma \in c$ and to construct a function $\rho'$ by smoothing $\rho$ in a neighborhood of the total image of $\sigma$. By the averaging properties of the convolution, the property that the integral of $\rho'$ over the measure of $\sigma$ is infinite remains true. However, the resulting $\rho'$ will be continuous and finite-valued on the compact total image of $\sigma$ and $\rho'$ will therefore have a maximal value on it. This is a contradiction.
\end{proof}
\begin{proof}[Proof of Lemma \ref{lem:no_exceptional_hom_classes_boundaryless}]
	Suppose towards contradiction that $c \in H_k^{\lip}(M; \K)$ is $p$-exceptional. By Lemma \ref{lem:Fuglede_exceptional_lemma} there exists a Borel function $\rho \colon M \to [0, \infty]$ such that $\norm{\rho}_{L^p} < \infty$, but $\rho$ has infinite integral over every $\mu_{\sigma}$, $\sigma \in c$. We fix a representative $\sigma \in c$ and note that the total image $K$ of $\sigma$ is compact. 
	
	We select bilipschitz charts $\phi_i \colon B^n(0, 2) \to M$, $i \in \{1, \dots, m\}$, such that $U_i = \phi_i(B^n(0, 1))$ cover $K$. We proceed by induction. Suppose that $\rho_{i-1} \in L^p(M)$ is such that $\rho_{i-1}$ is continuous and finite-valued on $U_1 \cup \dots \cup U_{i-1}$ and that $\rho_{i-1}$ has infinite integral over every $\mu_{\sigma}$, $\sigma \in c$. We construct $\rho_i$ that is continuous and finite-valued on $U_1 \cup \dots \cup U_{i}$ and has otherwise the same properties. Note that we take $\rho_0 = \rho$. Once these properties are shown for all $\rho_i$, we can pick $\rho' = \rho_m$ and the proof is complete.
	
	We fix $\eps > 0$ and define a function $\eta \colon B^n(0, 2) \to [0,\eps]$ by
	\[
		\eta(x) = \max\left( 0, (1 - \abs{x}) \eps \right)
	\] 
	For every vector $v \in B^n(0, 1)$, we define a homeomorphism $G_v : B^n(0, 2) \to B^n(0, 2)$ by
	\[
		G_v(x) = x + \eta(x) v.
	\]
	With a suitable choice of $\eps > 0$ dependent only on $n$, this is bilipschitz. By setting $F_{i, v} = \phi_i \circ G_v \circ \phi_i^{-1}$ in $U_i$ and extending to the rest of $M$ as identity, we hence obtain a bilipschitz automorphism $F_{i, v} \colon M \to M$. We next define $\rho_{i} \colon M \to [0, \infty]$ in the following way:
	\[
		\rho_i(x) = \int_{B^n(0, 1)} \rho_{i-1} \circ F_{i, v}(x) \dd v.
	\]
	
	We first show that the integral of $\rho_i$ with respect to $\mu_{\sigma}$ is still infinite for every $\sigma \in c$. For this, note that $(F_{i, v})_* \sigma$ is Lipschitz homotopic to $\sigma$ via $H(z, t) = ((F_{i, tv})_* \sigma)(z)$ and thus $(F_{i, v})_* \sigma \in c$. Hence,
	\begin{align}\label{eq:infinite_average}
	    \begin{split}
    		\int_M \rho_i(x) \dd \mu_{\sigma}(x) &= \int_M \int_{B^n(0, 1)} \rho_{i-1} \circ F_{i, v}(x) \dd v \dd \mu_{\sigma}(x)\\
    		&= \int_{B^n(0, 1)} \int_M \rho_{i-1}(z) \dd (F_{i, v})_* \mu_{\sigma}(z) \dd v\\
    		&\geq C^{-1} \int_{B^n(0, 1)} \int_M \rho_{i-1}(z) \dd \mu_{(F_{i, v})_*\sigma}(z) \dd v\\
    		&=  C^{-1} \int_{B^n(0, 1)} \infty \dd v = \infty.
		\end{split}
	\end{align}
	Note that the inequality here is a change of variables that uses the fact that $F_{i, v}$ is bilipschitz.
	
	Next, we show that $\rho_i$ is still in $L^p(M)$. It is enough to check that $\rho_i \circ \phi_i \in L^p(B^n(0, 1))$. However, if $x \in B^n(0, 1)$,
	\[
	\rho_i \circ \phi_i(x) = \int_{B^n(0, 1)} \rho_{i-1} \circ \phi_i(x + \eta(x) v) \dd v = \frac{1}{\eta^n(x)} \int_{B^n(x, \eta(x))} \rho_{i-1} \circ \phi_i(w) \dd w.
	\]
	Notably, $\rho_i \circ \phi_i(x)$ is bounded from above up to a constant by the maximal function $\maxfun(\rho_{i-1} \circ \phi_i)(x)$. By the Hardy-Littlewood maximal inequality for $p > 1$, we have $\rho_i \circ \phi_{i} \in L^p(B^n(0, 1))$ and hence $\rho_i \in L^p(M)$.
	
	It remains to prove continuity and finiteness of $\rho_i$ on $U_1 \cup \dots \cup U_{i}$. Finiteness on $U_i$ can be immediately seen from the previously stated expression for $\rho_i \circ \phi_{i}$, since $\eta(x) >  0$ and $\rho_{i-1} \circ \phi_i$ is in $L^1(B^n(x, \eta(x)))$. Suppose then that $x_k, x \in B^n(0, 1)$ for $k \in \Z_{> 0}$ and that $x_k \to x$ as $k \to \infty$. Then, denoting $B_x = B^n(x, \eta(x))$ and $B_{x, k} = B^n(x_k, \eta(x_k))$, we have
	\begin{align*}
		&\abs{\rho_i \circ \phi_i(x) - \rho_i \circ \phi_i(x_k)}\\
		&\qquad\leq \frac{1}{\eta^n(x)} \int_{(B_x \cup B_{x,k}) \setminus (B_x \cap B_{x,k})} \rho_{i-1} \circ \phi_i + \abs{\frac{1}{\eta^n(x_k)} - \frac{1}{\eta^n(x)}} \int_{B^n(0, R)} \rho_{i-1} \circ \phi_i,
	\end{align*}
	where both right hand side terms tend to 0 since $\eta$ is continuous and $\rho_{i-1} \circ \phi_i$ is $L^1(B^n(0, 1))$. We conclude that $\rho_i$ is continuous at $\phi_i(x)$.
	
	Moreover, since $\rho_i = \rho_{i-1}$ outside $U_i$, we have that $\rho_i$ is finite-valued in $U_{1} \cup \dots \cup U_{i-1}$ and continuous in $(U_{1} \cup \dots \cup U_{i-1}) \setminus \overline{U_i}$. Finally, suppose that $x \in \partial B^n(0, 1)$ and $\phi_i(x) \in U_1 \cup \dots \cup U_{i-1}$. Then $\rho_{i-1} \circ \phi_i$ is continuous at $x$ and hence if $x_k \in B^n(0, 1)$ are such that $x_k \to x$, we get
	\begin{align*}
		\abs{\rho_i \circ \phi_i(x) - \rho_i \circ \phi_i(x_k)} &\leq \int_{B^n(0, 1)} \abs{\rho_{i-1} \circ \phi_i(x_k + \eta(x_k) v) - \rho_{i-1} \circ \phi_i(x)}\\
		& \leq \vol_n(B^n(0, 1)) \max_{w \in B^n(x, \abs{x-x_k} + \eta(x_k))} \abs{\rho_{i-1} \circ \phi_i(w) - \rho_{i-1} \circ \phi_i(x)}\\
		&\xrightarrow{k \to \infty} 0.
	\end{align*}
	From this, it follows that $\rho_i$ is continuous at $\phi_i(x) \in (U_{1} \cup \dots \cup U_{i-1}) \cap \partial U_i$, completing the proof that $\rho_i$ is continuous in $U_1 \cup \dots \cup U_i$. 
\end{proof}

The non-relative case where $M$ has boundary is a reduction to the previous proof.

\begin{lemma}\label{lem:no_exceptional_hom_classes_bdry_non-relative}
	Let $M$ be a closed Lipschitz $n$-submanifold in an oriented Lipschitz Riemannian $n$-manifold $\cR$ without boundary and let $p \in (1, \infty)$, $k \in \{0, \dots, \infty\}$. Then no homology class $[\sigma] \in H_k^{\lip}(M; \K)$ is $p$-exceptional, where $\K \in \{\Z, \R\}$.
\end{lemma}
\begin{proof}
	There exists a locally Lipschitz \emph{collar neighborhood} of $\partial M$ in $M$. That is, there exists a locally Lipschitz homeomorphism $F_{\partial M} \colon (\partial M) \times [0, 2\eps] \to U \subset M$ such that $F_{\partial M}(x, 0) = x$. We refer to  \cite[Prop. 3.42]{Hatcher_AlgTopo} for the general idea of such constructions. For every $\delta \in [0, 2\eps]$, we define the projection map $P^\delta \colon M \to M$ where $P^\delta = \id_M$ outside $F_{\partial M}(\partial M \times [0, \delta])$ and $P^\delta(F_{\partial M}(x, t)) = F_{\partial M}(x, \delta)$ for $t \in [0, \delta]$. Now, if $\sigma \in c \in H_k^{\lip}(M; \K)$, then $\sigma' = P^{\eps}_* \sigma$ is Lipschitz homotopic to $\sigma$ by $H(x, t) = P^{t\eps}_*(\sigma(x))$. Hence, $\sigma' \in c$ and since $\sigma'$ does not meet the boundary of $M$, we can repeat the proof of Lemma \ref{lem:no_exceptional_hom_classes_boundaryless} on it with no significant changes.
\end{proof}

The final most general case is the full relative case of Proposition \ref{prop:no_exc_hom_classes}. Here, the extra technical hurdles caused by relative homology take far more effort to overcome.

\begin{proof}[Idea of the proof of Proposition \ref{prop:no_exc_hom_classes}]
    We use a locally Lipschitz collar neighborhood to find a smaller submanifold $M_\eps \approx M$ inside $M$ such that $M \setminus \intr M_\eps$ is essentially $\partial M \times [0, \eps]$. If $H_k^{\lip}(M, D; \K)$ has a $p$-exceptional class $c$, then $H_k^{\lip}(M_\eps, D_\eps; \K)$ also has a $p$-exceptional class $c_\eps$, where $D_\eps$ is $D \times \{\eps\}$. We pick $\rho \in L^p(M_\eps)$ such that the integral of $\rho$ over the measures of every chain $\sigma_\eps \in c_\eps$ is infinite and extend $\rho$ by zero to all of $M$. Moreover, we fix an element $\sigma_\eps \in c_\eps$ and turn it into an element $\sigma \in c$ by adding a chain that is essentially ``$\partial \sigma_\eps \times [0, \eps]$''. 
    
    The strategy is again to construct a $\rho'$ by smoothing out $\rho$ in a neighborhood of $\sigma_\eps$. Notably, our convolutions will now only translate the underlying space inwards in any neighborhoods meeting $\partial M_\eps$. Since $\rho$ was zero in $M \setminus M_\eps$, $\rho'$ will be continuous and finite in a neighborhood of the total image of $\sigma$ and it then remains to show that the integral of $\rho'$ over the measure of $\sigma$ remains infinite. 
    
    The idea for this is that after any of the translations used to perform the convolution, the part of $\sigma$ that gets translated inside $M_\eps$ will be an element of $c_\eps$ and hence $\rho$ will have infinite integral over it. A major technical issue, however, is that the intersection of a Lipschitz chain and a Lipschitz subdomain is not always a Lipschitz chain: a simple example is the graph of a Lipschitz function $f \colon \R \to \R$ crossing the $x$-axis at infinitely many points, in which case the part in the upper half-space is not a singular chain due to having infinitely many components. Regardless, using the fact that our chosen chain $\sigma$ is "straight" in the region $M \setminus \intr M_\eps \approx \partial \sigma_\eps \times [0, \eps]$ and the fact that any boundary translations in our convolution are inwards into $M_\eps$, we are able to verify that the required intersection is Lipschitz if the translations we use in the convolution are small enough.
\end{proof}
\begin{proof}[Full proof of Proposition \ref{prop:no_exc_hom_classes}]
	We define $F_{\partial M} \colon (\partial M) \times [0, 2\eps] \to U \subset M$ and $P^\delta$ for $\delta \in [0, 2\eps]$ as in the proof of Lemma \ref{lem:no_exceptional_hom_classes_bdry_non-relative}. For $\delta \in [0, 2\eps]$, we denote $M_{\delta} = P^\delta(M)$, $\partial M_{\delta} = P^\delta(\partial M)$, $D_\delta = P^\delta(D)$ and $\partial D_\delta = P^\delta(\partial D)$. We also define an alternate map $Q^\delta \colon M \to M_\delta$ for $\delta \in [0, 2\eps)$, where $Q^\delta = \id_M$ outside $U$ and $Q^\delta(F_{\partial M}(x, t)) = F_{\partial M}(x, \delta + (1-\delta/(2\eps))t)$. In particular, $Q^\delta$ is a locally bilipschitz homeomorphism $M \to M_\delta$.
	
	Suppose towards contradiction that $c \in H_k^{\lip}(M, D; \K)$ is $p$-exceptional. We denote $c_\eps = Q^\eps_* c \in H^k_{\lip}(M_\eps, D_\eps; \K)$. Since $Q^\eps$ is locally bilipschitz, it follows that $c_\eps$ is also $p$-exceptional. We consequently find a $\rho \in L^p(M_\eps)$ such that for every $\sigma \in c_\eps$, the integral of $\rho$ over $\mu_\sigma$ is infinite. We extend $\rho$ to $M$ by setting $\rho = 0$ on $M \setminus M_\eps$.
	
	We fix a Lipschitz $k$-chain $\sigma_\eps \in c_\eps$, in which case $\partial \sigma_\eps \in C_{k-1}^{\lip}(D_\eps; \K)$. Let $K$ be the total image of $\sigma_\eps$. We may assume that $K \cap (\partial M_\eps \setminus \intr D_\eps) = \emptyset$. Indeed, this can be achieved by similar methods as in the proof of Lemma \ref{lem:no_exceptional_hom_classes_bdry_non-relative}, by using a locally Lipschitz map similar to $P^\delta$ that takes $\partial M_\eps \setminus D_\eps$ inside $M_\eps$ and $D_\eps$ inside $\intr D_\eps$. 
	
	The boundary $\partial \sigma_\eps$ is in $C^{\lip}_{k-1}(\partial M_\eps, \K)$. We use $\partial \sigma_\eps$ to define a chain $\sigma_c \in C^{\lip}_k(U; \K)$ as follows: for every simplex $\gamma \colon \Delta_{k-1} \to \partial M_\eps$ in the sum representation of $\partial \sigma_\eps$, we define $\gamma_c \colon \Delta_{k-1} \times [0, \eps] \to U$ such that $F_{\partial M}^{-1} \circ \gamma_c(x, t) = ((F_{\partial M}^{-1} \circ \gamma(x))_{\partial M}, t)$. Every $\gamma_c$ can be understood as a Lipschitz $k$-chain following a suitable subdivision of the domain. Due to this, we treat $\gamma_c$ like singular Lipschitz simplices during the rest of the proof. We define $\sigma_c$ to be the sum of these $\gamma_c$ with the corresponding coefficients from $\partial \sigma_\eps$. We then fix $\sigma = \sigma_\eps + \sigma_c$. Since $\sigma_c$ is essentially a straight Lipschitz homotopy from $\partial \sigma_\eps$ to $(Q^\eps)^{-1}_* \partial \sigma_\eps$, it follows that $\partial \sigma \in \intr D$.
	
	We next pick charts $\phi_i \colon (-2, 2)^n \to M$, $i \in \{1, \dots, m\}$, such that the sets $U_i = \phi_i((-1, 1)^n)$ cover $K$. Because of our assumption that $K \cap (\partial M_\eps \setminus \intr D_\eps) = \emptyset$, we can restrict our attention to charts of two types. The first type consists of charts inside the interior of $M_\eps$, i.e.\ $U_i \subset M_\eps$. The second type consists of rectangular boundary charts around points of $D_\eps$. We may choose these to be such that $U_i = F_{\partial M} (V_i \times (0, 2\eps))$ and $\phi_i = F_{\partial M} \circ (\phi_i' \times ((\eps/2) \id_\R + \eps))$ where $V_i \subset \intr D$ and $\phi_i' \colon (-2,2)^{n-1} \to V_i$ is a bilipschitz chart.
	
	For every $i \in \{1, \dots, m\}$ we let $\eps_i > 0$ and define maps $\eta_i \colon (-2, 2)^n \to [0, \eps_i]$ by $\eta_i(x) = \max(0, 1-\abs{x}_{\ell^\infty})\eps_i$. That is, $\eta_i$ are similar as in Lemma \ref{lem:no_exceptional_hom_classes_boundaryless}, except adjusted to use $\ell^\infty$-norms due to our domain being a square. Further following the proof of Lemma \ref{lem:no_exceptional_hom_classes_boundaryless}, we also define for every $v \in B^n(0, 1)$ a map $G_{i, v} \colon (-2, 2)^n \to (-2, 2)^n$ by $G_{i, v}(x) = x + \eta_i(x) v$, assume $\eps_i$ are small enough that $G_{i, v}$ are bilipschitz and set $F_{i,v} = \phi_i \circ G_{i, v} \circ \phi_i^{-1}$ in $U_i$ and $F_{i, v} = \id_M$ elsewhere. We again set $\rho_0 = \rho$ and define inductively that
	\[
		\rho_i(x) = \int_{B^n(0, 1) \cap \H^n_{+}} \rho_{i-1} \circ F_{i, v}(x) \dd v.
	\]
	Note the difference to Lemma \ref{lem:no_exceptional_hom_classes_boundaryless} that the averaging occurs over the upper half ball.
	
	By repeating the corresponding parts of the proof of Lemma \ref{lem:no_exceptional_hom_classes_boundaryless} with trivial changes, we are able to conclude that $\rho_m$ is in $L^p(M)$ and that $\rho_m$ is continuous and finite-valued on $U_1 \cup \dots \cup U_m$. Moreover, since the $\rho$ we started with is identically zero in $M \setminus M_\eps$ and hence continuous there, it in fact follows that $\rho_m$ is continuous and finite-valued in $U_1 \cup \dots \cup U_m \cup (M \setminus M_\eps)$. 
    Since the total image of $\sigma$ lies in $U_1 \cup \dots \cup U_m \cup (M \setminus M_\eps)$, it only remains to show that the integral of $\rho_m$ over $\mu_\sigma$ is infinite, in which case a contradiction follows. 
	
	For this, fix $v_1, \dots, v_m \in B^n(0, 1) \cap \H^n_+$ and denote $\sigma_{v_1, \dots, v_m} = (F_{m, v_m})_* \dots (F_{2,v_2})_* (F_{1, v_1})_* \sigma$. Due to our use of $v_i$ in the upper half space and to how our boundary charts are set up, the total image of $(F_{m, v_m})_* \dots (F_{2,v_2})_* (F_{1, v_1})_* \sigma_\eps$ lies inside $\intr M_\eps$. Therefore any intersections of $\sigma_{v_1, \dots, v_m}$ with $\partial M_\eps$ will occur in the $(F_{m, v_m})_* \dots (F_{2,v_2})_* (F_{1, v_1})_* \sigma_c$-part. The Lipschitz constants of $G_{i, sv_i}$ have an upper bound independent of the choice of $s \in [0,1]$ and $v_i \in B^n(0, 1)$. Hence, we may fix a constant $L \in (1, \infty)$ independent of $s$ and $v_i$ (but dependent on the charts $\phi_i$) such that the maps $\phi_i^{-1} \circ \Phi_{i'} \circ \phi_{i''}$ are all $L$-Lipschitz in their domains of definition.
	
	For every $i \in \{0, \dots, m\}$, we define $\Phi_i \colon \partial_M \times [0, \eps] \times [0, 1] \to M$ by
	\begin{align*}
	    	\Phi_i(x, t, s) = F_{i, sv_i} \circ F_{i-1, sv_{i-1}} \circ \dots \circ F_{1, sv_1} \circ F_{\partial_M} (x, t).
	\end{align*}
	We fix a chart $\phi = \phi_j' \times \id_{[0, \eps]} \times \id_{[0, 1]}$ where $\phi_j' \colon (-2, 2)^{n-1} \to V_j$ is from the definition of one of the charts $\phi_i$ that meet the boundary of $M_\eps$. For every $i \in \{0, m\}$ and $(x, s) \in V_j \times [0, 1]$, we denote by $t_i(x, s)$ the largest $t \in [0, \eps]$ such that $\Phi_i(x, t, s) \in \partial M_\eps$; note that $t_0(x,s) \equiv \eps$.
	We claim that for every $i \in \{1, m\}$, if we select $\eps_i$ small enough, then the function $t_i$ is Lipschitz in $V_j$. First we will show that for all $x, x' \in V_j$, $s_1, s_2 \in [0, 1]$, $t_1 \in [0, t_{i-1}(x, s_1)]$ and $t_2 \in [0, t_{i-1}(x, s_2)]$ with $t_1 < t_2$, we have
	\begin{align}\label{eq:t_increasing}
	\begin{split}
		2^{-i} (t_1 - t_2) - (1-2^{-i}) &\left( \abs{(\phi_j')^{-1}(x)-(\phi_j')^{-1}(x')} + \abs{s_1 - s_2}\right)\\
		&\leq (F_{\partial M}^{-1}(\Phi_i(x, t_1, s_1)))_{[0, 2\eps]} - (F_{\partial M}^{-1}(\Phi_i(x', t_2, s_2)))_{[0, 2\eps]}\\
		&\leq (2-2^{-i}) (t_1 - t_2) + (1-2^{-i}) \left( \abs{(\phi_j')^{-1}(x)-(\phi_j')^{-1}(x')} + \abs{s_1 - s_2}\right).
	\end{split}
	\end{align}
	In fact, $t_i(x, s_1)$ is the only $t \in [0, \eps]$ such that $\Phi_i(x, t, s_1) \in \partial M_\eps$.

	We prove \eqref{eq:t_increasing} by induction. 
	Suppose that the claim holds for $i-1$, where the case of $i-1 = 0$ is clear since $(F_{\partial M}^{-1}(\Phi_0(x, t, s)))_{[0, 2\eps]} = t$.
	If $\Phi_{i-1}(x, t, s) \in M \setminus M_{\eps}$, then $\Phi_i(x, t, s) \in M \setminus M_{2\eps}$ and so \eqref{eq:t_increasing} is well defined. Next, supposing that $(x, s, t) \in V_j \times [0, \eps] \times [0, 1]$, if we have $\Phi_{i-1}(x, t, s) \in U_i$, we can compute that
	\begin{align*}
		(F_{\partial M}^{-1}(\Phi_i(x, t,s)))_{[0, 2\eps]}
		&= (F_{\partial M}^{-1} \circ \phi_i \circ G_{i, sv_i} \circ \phi_i^{-1} \circ (\Phi_{i-1}(x, t,s)))_{[0, 2\eps]}\\
		&= (\eps/2) (G_{i, sv_i} \circ \phi_i^{-1} (\Phi_{i-1}(x, t,s)))_n + \eps\\
		&= (\eps/2) (\phi_i^{-1} (\Phi_{i-1}(x, t,s)))_n + \eps + (\eps/2) \eps_i (sv_i)_n (1 - \lvert\phi_i^{-1} (\Phi_{i-1}(x, t,s))\rvert_{\ell^\infty})\\
		&= (F_{\partial M}^{-1}(\Phi_{i-1}(x, t,s)))_{[0, \eps]} + (\eps/2) \eps_i s (v_i)_n (1 - \lvert\phi_i^{-1} (\Phi_{i-1}(x, t,s))\rvert_{\ell^\infty}) .
	\end{align*}
	On the other hand, if $\Phi_{i-1}(x, t,s) \notin U_i$, then $(F_{\partial M}^{-1}(\Phi_{i}(x, t,s)))_{[0, \eps]} = (F_{\partial M}^{-1}(\Phi_{i-1}(x, t,s)))_{[0, \eps]}$. It follows that if $\Phi_{i-1}(x, t_1,s_1)$ and $\Phi_{i-1}(x', t_2,s_2)$ are both in $U_i$, we may estimate using $s_1, s_2 \leq 1$ and $\lvert\phi_i^{-1} (\Phi_{i-1}(x, t_1, s_1))\rvert_{\ell^\infty} \leq 1$ that
	\begin{equation}\label{eq:double_diff_case_1}\begin{aligned}
		&\Big\lvert\bigl((F_{\partial M}^{-1}(\Phi_{i}(x, t_1, s_1)))_{[0, 2\eps]} - (F_{\partial M}^{-1}(\Phi_{i}(x', t_2, s_2)))_{[0, 2\eps]}\bigr)\\
		&\qquad\qquad - \bigl((F_{\partial M}^{-1}(\Phi_{i-1}(x, t_1,s_1)))_{[0, 2\eps]} - (F_{\partial M}^{-1}(\Phi_{i-1}(x', t_2,s_2)))_{[0, 2\eps]}\bigr)\Big\rvert\\
		&\qquad \leq (\eps/2) \eps_i \left( \abs{s_1 - s_2} + \abs{s_1 \lvert\phi_i^{-1} (\Phi_{i-1}(x, t_1,s_1))\rvert_{\ell^\infty} - s_2\lvert\phi_i^{-1} (\Phi_{i-1}(x', t_2,s_2))\rvert_{\ell^\infty}} \right)\\
		&\qquad \leq (\eps/2) \eps_i \left( 2\abs{s_1 - s_2} + \abs{ \lvert\phi_i^{-1} (\Phi_{i-1}(x, t_1,s_1))\rvert_{\ell^\infty} - \lvert\phi_i^{-1} (\Phi_{i-1}(x', t_2,s_2))\rvert_{\ell^\infty}} \right)\\
		&\qquad \leq (\eps/2) \eps_i  \left( 2\abs{s_1 - s_2} + \abs{ \phi_i^{-1} (\Phi_{i-1}(x, t_1,s_1)) - \phi_i^{-1} (\Phi_{i-1}(x', t_2,s_2))} \right) \\
		&\qquad \leq \eps \eps_i (1 + L)\left( (t_1 - t_2) + \abs{(\phi_j')^{-1}(x)-(\phi_j')^{-1}(x')} + \abs{s - s'}\right).
	\end{aligned}\end{equation}
	If on the other hand $\Phi_{i-1}(x, t_1,s_1) \in U_i$ and $\Phi_{i-1}(x', t_2,s_2) \notin U_i$, then there exists a convex combination $(x'', t_3, s_3) = \phi(\lambda\phi^{-1}(x, t_1, s_1) + (1-\lambda)\phi^{-1}(x', t_2, s_2)) \in V_i \times [0, \eps] \times [0, 1]$, $\lambda \in (0, 1]$, such that $\Phi_{i-1}(x'', t_3,s_3) \in (\phi_i (-2, 2)^n) \setminus U_i$. Hence, we obtain that $\lvert\phi_i^{-1} (\Phi_{i-1}(x, t_1,s_1))\rvert_{\ell^\infty} < 1 \leq \lvert\phi_i^{-1} (\Phi_{i-1}(x'', t_3,s_3))\rvert_{\ell^\infty} \leq 2$.
	We instead estimate
	\begin{equation}\label{eq:double_diff_case_2}\begin{aligned}
		&\Big\lvert\bigl((F_{\partial M}^{-1}(\Phi_{i}(x, t_1, s_1)))_{[0, 2\eps]} - (F_{\partial M}^{-1}(\Phi_{i}(x', t_2, s_2)))_{[0, 2\eps]}\bigr)\\
		&\qquad\qquad - \bigl((F_{\partial M}^{-1}(\Phi_{i-1}(x, t_1,s_1)))_{[0, 2\eps]} - (F_{\partial M}^{-1}(\Phi_{i-1}(x', t_2,s_2)))_{[0, 2\eps]}\bigr)\Big\rvert\\
		&\qquad \leq (\eps/2) \eps_i \abs{1 - \lvert\phi_i^{-1} (\Phi_{i-1}(x, t_1,s_1))\rvert_{\ell^\infty}}\\
		&\qquad \leq (\eps/2) \eps_i  \abs{\lvert\phi_i^{-1} (\Phi_{i-1}(x, t_1,s_1))\rvert_{\ell^\infty} - \lvert\phi_i^{-1} (\Phi_{i-1}(x'', t_3,s_3))\rvert_{\ell^\infty}}\\
		&\qquad \leq (\eps/2) \eps_i L \left( (t_1 - t_3) + \abs{(\phi_j')^{-1}(x)-(\phi_j')^{-1}(x'')} + \abs{s_1 - s_3} \right)\\
		&\qquad \leq (\eps/2) \eps_i L \left( (t_1 - t_2) + \abs{(\phi_j')^{-1}(x)-(\phi_j')^{-1}(x')} + \abs{s_1 - s_2} \right).
	\end{aligned}\end{equation}
	The case $\Phi_{i-1}(x, t_1,s_1) \notin U_i, \Phi_{i-1}(x', t_2,s_2) \in U_i$ is identical to the previous case and in the final case $\Phi_{i-1}(x, t_1,s_1) \notin U_i, \Phi_{i-1}(x', t_2,s_2) \notin U_i$, we have
	\begin{equation}\label{eq:double_diff_case_3}\begin{aligned}
			&\Big\lvert\bigl((F_{\partial M}^{-1}(\Phi_{i}(x, t_1, s_1)))_{[0, 2\eps]} - (F_{\partial M}^{-1}(\Phi_{i}(x', t_2, s_2)))_{[0, 2\eps]}\bigr)\\
			&\qquad\qquad - \bigl((F_{\partial M}^{-1}(\Phi_{i-1}(x, t_1,s_1)))_{[0, 2\eps]} - (F_{\partial M}^{-1}(\Phi_{i-1}(x', t_2,s_2)))_{[0, 2\eps]}\bigr)\Big\rvert = 0.
	\end{aligned}\end{equation} 
	Thus, by choosing $\eps_i < \eps^{-1} (1 + L)^{-1} 2^{-i}$ and by combining \eqref{eq:double_diff_case_1}-\eqref{eq:double_diff_case_3} with the case $i-1$ of \eqref{eq:t_increasing}, we get the case $i$ of \eqref{eq:t_increasing}.
	
	By using \eqref{eq:t_increasing} with $x = x'$ and $s_1 = s_2 = s$, it follows that $t \mapsto (F_{\partial M}^{-1}(H_{c,i}(x, t, s)))_{[0, 2\eps]}$ is increasing on $[0, t_{i-1}(x, s)]$. Moreover, since $F_{i, sv_i}$ translate inwards whenever $U_i$ meets $\partial M_\eps$, we have $F_{\partial M}^{-1}(H_{c,i}(x, t_{i-1}(x, s), s)) \geq \eps$. Hence, for every $(x, s) \in V_i \times [0, 1]$ there in fact exists exactly one $t_i(x, s) \in [0, \eps]$ such that $\Phi_i(x, t_i(x, s), s) \in \intr D_\eps$, that is, $(F_{\partial M}^{-1}(\Phi_i(x, t_i(x, s), s)))_{[0, 2\eps]} = \eps$. Moreover, by applying \eqref{eq:t_increasing} with $t_1 = t_{i}(x, s_1)$ and $t_2 = t_{i}(x', s_2)$, it follows that
	\[
		\abs{t_{i}(x, s_1) - t_{i}(x', s_2)} \leq 2^i (1 - 2^{-i}) \left( \abs{(\phi_j')^{-1}(x) - (\phi_j')^{-1}(x')} + \abs{s_1 - s_2}\right),
	\]
	which implies that $t_i \circ \phi$ is Lipschitz. The induction part is hence complete.
	
	With the map $t_m$, we may define a map $P' \colon \partial M \times [0, \eps] \times [0, 1] \to \partial M \times [0, \eps]$ by
	\[
		P'(x, t, s) = (x, \max(t, t_m(x, s))).
	\]
	With the help of this map, we then define $P \colon M \times [0, 1] \to M$ by setting $P(x, s) = F_{\partial M}(P'(F_{\partial M}^{-1}(x), s)$ on $M \setminus M_\eps$ and $P(x, s) = x$ on $M_\eps$. Finally, we define $\Theta \colon M \times [0, 1] \to M$ by
	\[
		\Theta(x, s) = F_{m, sv_m} \circ \dots \circ F_{1, sv_1} \circ P(x, s).
	\] 
	We note that the image of $\Theta$ lies in $M_\eps$. By our proof that $t_m$ is Lipschitz on any of our chosen boundary neighborhoods, we get that $\Theta$ is Lipschitz in a neighborhood of $K \times [0, 1]$. We denote $\Theta^s = \Theta(\cdot, s)$ for $s \in [0, 1]$.
	
	The map $\Theta$ hence provides a Lipschitz homotopy between $\Theta^1_* \sigma$ and $\Theta^0_* \sigma$. Recalling that $\sigma = \sigma_\eps + \sigma_c$ where $\sigma_c$ is a "straight" extension of $\partial \sigma_\eps$ to $M \setminus M_\eps$, let $\gamma_c \colon \Delta_{k-1} \times [0, \eps] \to M$ be one of the terms in the sum representation of $\sigma_c$. Since $t_m(\cdot, 0) \equiv \eps$, we have $\Theta^0_* \gamma_c(x, t) = \gamma_c(x, \eps) = \gamma(x)$ for all $(x, t) \in \Delta_{k-1} \times [0, \eps]$, where $\gamma$ again is the simplex of $\partial \sigma_\eps$ that $\gamma_c$ is a straight extension of. Since also $\Theta^0_* \sigma_\eps = \sigma_\eps$, we conclude that $\Theta^0_* \sigma$ is in the same Lipschitz homology class $c_\eps$ as $\sigma_\eps$. Moreover, the total image of the boundary $\partial \sigma$ is in $K \cap \partial M$, which every $\Theta_s$ maps into $D_\eps$. Because of this, $\Theta^1_* \sigma - \Theta^0_* \sigma \in \partial C^{\lip}_{k+1}(M_\eps; \K) + C^{\lip}_k(D_\eps; \K)$, implying that $\Theta^1_* \sigma \in c_\eps$.
	
	Finally, we return to the $k$-chain $\sigma_{v_1, \dots, v_m} = (F_{m, v_m})_* \dots (F_{2,v_2})_* (F_{1, v_1})_* \sigma$. Clearly we have $\Theta^1_* \sigma_\eps = (F_{m, v_m})_* \dots (F_{2,v_2})_* (F_{1, v_1})_* \sigma_\eps$. Moreover, if $\gamma_c$ is a component of the sum representation of $\sigma_c$, we have $\Theta^1_* \gamma_c(x, t) = (F_{m, v_m})_* \dots (F_{2,v_2})_* (F_{1, v_1})_* \gamma_c(x, t)$ when $t \geq t_m((Q^\eps)^{-1}_* \gamma(x), 1)$ and $\Theta^1_* \gamma_c(x, t) = (F_{m, v_m})_* \dots (F_{2,v_2})_* (F_{1, v_1})_* \gamma_c(x, t_m((Q^\eps)^{-1}_* \gamma(x), 1))$ otherwise. Notably, for $t \leq t_m((Q^\eps)^{-1}_* \gamma(x), 1)$, $\Theta^1_* \gamma_c(x, t)$ is independent of $t$ and this part will hence yield a zero-dimensional Jacobian if $\Theta^1_* \gamma_c$ is post-composed with a chart. It follows that $\mu_{\sigma_{v_1, \dots, v_m}}$ is bounded from below by a multiple of $\mu_{\Theta^1_* \sigma}$. Since $\Theta^1_* \sigma \in c_\eps$, the integral of $\rho$ over $\mu_{\Theta^1_* \sigma}$ is infinite by the definition of $\rho$ and therefore the integral of $\rho$ over $\mu_{\sigma_{v_1, \dots, v_m}}$ is also infinite. Since this holds for all choices of $v_1, \dots, v_m$, it follows by a computation as in \eqref{eq:infinite_average} that the integral of $\rho_m$ over $\mu_{\sigma}$ is infinite. The proof is hence complete.

\end{proof}

\section{Stokes-type theorems for $L^p$-forms}\label{sec:stokes}

In this section, we prove Stokes-type results for $L^p$-forms in the setting of Lipschitz manifolds. In particular, our aim is to provide a generalization of the theory discussed in \cite[Chapter III]{Fuglede_surface-modulus}, where instead of vector fields in Euclidean space, we consider Sobolev forms on Lipschitz domains in a Lipschitz manifold.

We start by pointing out that, by combining Proposition \ref{prop:Lipschitz_stokes} with the Fuglede lemma for forms given in Lemma \ref{lem:fuglede_for_forms}, we obtain the following $L^p$-version of the Lipschitz Stokes theorem.

\begin{cor}\label{cor:lipschitz_stokes_L^p}
	Let $M$ be a closed Lipschitz $n$-submanifold in an oriented Lipschitz Riemannian $n$-manifold $\cR$. Suppose that $\omega \in W^{d,p}(\wedge^{k-1} T^* M)$. There exists a $p$-exceptional $E_1 \subset C_k^{\lip}(M; \Z)$ and $E_2 \subset C_k^{\lip}(M; \Z)$ for which $\partial E_2 \subset C_{k-1}^{\lip}(M; \Z)$ is $p$-exceptional, such that
	\[
		\int_{\partial \sigma} \omega = \int_{\sigma} d\omega.
	\]
	for all $\sigma \in C_k^{\lip}(M; \R) \setminus (E_1 \cup E_2)$.
\end{cor}
\begin{proof}
	By Lemma \ref{lem:W_infinity_approximation}, we can approximate $\omega$ by $W^{d,\infty}$-forms $\omega_i$. By Proposition \ref{prop:Lipschitz_stokes}, every $\omega_i$ satisfies the desired Stokes' theorem for every $\sigma \in C_k^{\lip}(M; \R)$. Then we extend $\omega$, $\omega_i$, $d\omega$ and $d\omega_i$ by zero to $\cR$.
	By Lemmas \ref{lem:fuglede_for_forms} and \ref{lem:two_integrals} and after moving to a subsequence, $\int_{\sigma} d\omega_j \to \int_{\sigma} d\omega$ outside a $p$-exceptional set of $\sigma \in C_k^{\lip}(M; \Z)$.  Additionally, $\int_{\partial \sigma} \omega_j \to \int_{\partial \sigma} \omega$ outside a $p$-exceptional set of $\partial \sigma \in \partial C_{k-1}^{\lip}(M; \Z)$.
\end{proof}

Having two exceptional sets in Corollary \ref{cor:lipschitz_stokes_L^p} is necessary. For the necessity of $E_2$, we can take a smooth $\omega \in C^\infty_0(\wedge^{n-1} \R^n)$ and redefine it on the boundary of a single $n$-simplex $\sigma$, making the formula fail on every simplex $\sigma'$ with $\partial \sigma' = \partial \sigma$.
The singleton $\{\partial \sigma\}$ is $p$-exceptional, but the family of $n$-simplices $\sigma'$ with $\partial \sigma' = \partial \sigma$ is not. 

For the necessity of $E_1$, we redefine a smooth $d\omega \in C^\infty_0(\wedge^k \R^n)$ with $k < n$ on a single $k$-chain $\sigma$ with $\partial \sigma = 0$. The formula fails for all $k$-chains $\sigma + \sigma'$ where the image of $\sigma'$ is disjoint from the image of $\sigma$.
The family of such $k$-chains is exceptional, but the family of their boundaries is not.

\subsection{The Fuglede-type characterization}
The main result of this section characterizes closedness and exactness for Sobolev forms purely by integration over $k$-chains. Moreover, it also includes a similar characterization of the weak vanishing of $\omega_T$ in such cases.  

\begin{theorem}\label{thm:integral_characterization_thm}
	Suppose that $M$ is a closed Lipschitz $n$-submanifold in an oriented Lipschitz Riemannian $n$-manifold $\cR$ without boundary and $D$ is a closed Lipschitz $(n-1)$-submanifold in $\partial M$. Let $p \in (1, \infty)$ and $\omega \in L^p_\loc(\wedge^k T^* M)$. The following are equivalent.
	\begin{itemize}
		\item $\omega \in W^{d,p}_{T(D), \loc}(\wedge^k T^* M)$ and $d\omega = 0$.
		\item For $p$-a.e.\ $\sigma \in C_k^{\lip}(D; \Z) + \partial C_{k+1}^{\lip}(M; \Z) \subset C_k^{\lip}(M; \Z)$, 
		\[
			\int_{\sigma} \omega = 0.
		\]
	\end{itemize}
	The following are also equivalent.
	\begin{itemize}
		\item $\omega = d\tau$ for some $\tau \in W^{d,p}_{T(D), \loc}(\wedge^{k-1} T^* M)$.
		\item For $p$-a.e.\ $\sigma \in C_{k+1}^{\lip}(M; \Z)$ such that $\partial \sigma \in C_k^{\lip}(D; \Z)$, 
		\[
			\int_{\sigma} \omega = 0.
		\]
	\end{itemize}
\end{theorem}
The theorem immediately implies Theorem \ref{thm:characterization_thmintro} by considering $M = \cc R$ and $\partial M = \emptyset$.
Theorem \ref{thm:integral_characterization_thm} consists of four individual implications. Two of them can be proven locally and we begin with those statements. The remaining two require global information, which in our proof is provided by the use of $L^p$-cohomology. Note that the proofs of almost all of the implications also apply for $p=1$, except for one proof which uses Proposition \ref{prop:no_exc_hom_classes}.

\subsection{The local implications}

We start by showing the following part of Theorem \ref{thm:integral_characterization_thm}.

\begin{lemma}\label{lem:chara_implication_2}
	Suppose that $M$ is a closed Lipschitz $n$-submanifold in an oriented Lipschitz Riemannian $n$-manifold $\cR$ without boundary and $D$ is a closed Lipschitz $(n-1)$-submanifold in $\partial M$. Let $p \in [1, \infty)$ and $\omega \in L^p_\loc(\wedge^k T^* M)$. If 
	\[
		\int_{\sigma} \omega = 0,
	\]
	for $p$-a.e.\ $\sigma \in C_k^{\lip}(D; \Z) + \partial C_{k+1}^{\lip}(M; \Z)$, then $\omega \in W^{d,p}_{T(D), \loc}(\wedge^k T^* M)$ and $d\omega = 0$.
\end{lemma}

We first prove Euclidean versions of the result, starting without boundary conditions.

\begin{lemma}\label{lem:zero_ints_implies_weak_cl_Rn}
	For $k \in \{1, \dots, n\}$, let $\cB^n_k \subset C_{k}^{\lip}(\R^n; \Z)$ denote the collection of all isometrically embedded $k$-balls in $\R^n$. Let $k \geq 0$ and $\omega \in L^{p}_\loc(\wedge^k \R^n)$, where $p \in [1, \infty)$. If
	\[
	\int_{\partial \sigma} \omega = 0,
	\]
	for $p$-a.e.\ $\partial \sigma \in \partial \cB^n_{k+1}$, then $d\omega = 0$ weakly.
\end{lemma}
\begin{proof}
    We first show that if $\sigma \in \cB^n_{k+1}$ and $A \subset \R^n$ is a measurable set such that the family $\cS_A := \{\partial\sigma - y : y \in A\}$ is $p$-exceptional, then $m_n(A) = 0$. By subadditivity of $m_n$ and $\moddens_p$, we may assume that $A$ is bounded and in particular $m_n(A) < \infty$. Since $\spt \nu_{\partial\sigma}$ is also bounded, we may select a bounded $B \subset \R^n$ such that $x - y \in B$ whenever $y \in A$ and $x \in \spt \nu_{\partial\sigma}$. Suppose towards contradiction that $m_n(A) > 0$. Lemma \ref{lem:Fuglede_exceptional_lemma} yields a function $\rho \in L^p(\R^n)$ such that the integral of $\rho$ over every $\partial \sigma - y$ with $y \in A$ is infinite. By Tonelli's theorem and our counterassumption that $m_n(A) > 0$, we have
	\begin{align*}
		\int_{\partial \sigma} \int_A \rho(x-y) \dd y \dd x
		&= \int_A \int_{\partial \sigma} \rho(x-y) \dd x \dd y\\
		&= \int_A \int_{\partial \sigma - y} \rho(x') \dd x' \dd y = \infty.
	\end{align*}
	Hence, the function
	\[
	    x \mapsto \int_A \rho(x-y)dy
	\]
	must be unbounded over $\spt \nu_{\partial\sigma}$. However, since $\rho \in L^p(\bR^n)$ and $B \subset \R^n$ is bounded, we have $\rho \in L^1(B)$. It follows that
	\begin{align*}
	    \int_A \rho(x-y)dy \le \int_{B} \rho < \infty,
	\end{align*}
	for all $x \in \spt \nu_{\partial\sigma}$, resulting in a contradiction. We conclude that indeed $m_n(A) = 0$.

	Suppose then that
	\[
	    \int_{\partial \sigma} \omega = 0
	\]
	for $p$-a.e.\ $\partial \sigma \in \partial \cB^n_{k+1}$. We take the convolution $\omega_j = \eta_j \ast \omega$ for a sequence of mollifying kernels $\eta_j \in C^\infty_0(\R^n)$. Fix $\partial \sigma \in \partial \cB^n_{k+1}$.
	By Fubini's theorem and a change of variables,
	\begin{align*}
		\int_{\partial\sigma} \omega_j 
		&= \int_{\partial\sigma} \int_{\R^n} \omega_{x-y} \eta_j(y) \dd y \dd x\\
		&= \int_{\R^n} \eta_j(y) \int_{\partial\sigma} \omega_{x-y} \dd x \dd y\\
		&= \int_{\R^n} \eta_j(y) \left( \int_{\partial\sigma - y} \omega\right) \dd y.
	\end{align*}
    For $p$-a.e.\ element of $\{\partial\sigma - y : y \in \R^n\}$, we have by assumption that
	\[
	    \int_{\partial\sigma - y} \omega = 0.
	\]
	The $p$-exceptional family where this does not hold is of the form $\cS_A$ for some $A \subset \R^n$.
	We can place $A$ between two measurable sets $A^{-} \subset A \subset A^{+}$ with the same $m_n$-measure. By the statement shown in the beginning of the proof, we have $m_n(A^{-}) = 0$, which also yields $m_n(A^{+}) = 0$. 
	So
	\[
	\int_{\R^n} \eta(y) \left( \int_{\partial\sigma - y} \omega\right) \dd y
	= \int_{\R^n\setminus A^{+}} 0 \dd y = 0.
	\]
	In conclusion, the integral of $\omega_j$ over every $\partial \sigma \in \partial \cB^n_{k+1}$ is zero.
	
	Since $\omega_j$ are smooth, it follows by the standard smooth Stokes' theorem that the integral of $d\omega_j$ over every $\sigma \in \cB^n_{k+1}$ is zero. This is only possible if $d\omega_j$ is identically zero. 
	Let $V \subset \R^n$ be a bounded domain, in which case  $\omega_j \to \omega$ in $L^p(\wedge^k V)$. Since $d\omega_j \to 0$ in $L^p(\wedge^k V)$, we obtain that $d\omega = 0$ weakly in $V$. By the locality properties of the weak differential, it follows that $d\omega = 0$ weakly in $\R^n$.
\end{proof}

Next, we give a Euclidean version with boundary.

\begin{lemma}\label{lem:zero_ints_implies_weak_cl_Rn_bdry}
	Let $\H^n$ denote the upper half-space of $\R^n$. For $k \in \{1, \dots, n\}$, let $\cB^n_k \subset C_{k}^{\lip}(\R^n; \Z)$ denote the collection of all isometrically embedded $k$-balls in $\R^n$. Let $k \geq 0$ and $\omega \in L^{p}_\loc(\wedge^k \H^n)$ where $p \in [1, \infty)$. If
	\[
	\int_{\sigma} \omega = 0,
	\]
	for $p$-a.e.\ $\sigma \in C_{k}^{\lip}(\R^n; \Z)$ of the form $\sigma = \partial \sigma' \cap \overline{\H^n}$ where $\sigma' \in \cB^n_{k+1}$, then $d\omega = 0$ weakly in $\H^n$ and $\omega_T$ vanishes weakly on $\partial \H^n$.
\end{lemma}

\begin{proof}
	We extend $\omega$ into $\R^n$ by setting $\omega = 0$ in $\R^n \setminus \H^n$. We wish to prove that this extended $\omega$ satisfies $d\omega = 0$ in $\R^n$. It  immediately follows from this that $d\omega = 0$ in $\H^n$ and that $\omega_T = 0$ weakly on $\partial \H^n$.
	
	We point out that for every $\partial \sigma' \in \partial \cB^n_{k+1}$ that meets $\H^n$, the intersection $\partial \sigma' \cap \overline{\H^n}$ is clearly a valid element of $C_k^{\lip}(\R^n; \Z)$. (This is our reason for focusing on $\sigma' \in \cB^n_{k+1}$, as this would not be true for general $\sigma' \in C_{k+1}^{\lip}(\R^n; \Z)$ due to, e.g., the possibility of infinitely many crossings.) By the assumption, the integral of $\omega$ over $p$-a.e.\ such $\partial \sigma' \cap \overline{\H^n}$ vanishes. Since the extension of $\omega$ vanishes outside $\H^n$, the integral of $\omega$ over the rest of $\partial \sigma'$ also vanishes. It follows that the integral of $\omega$ over every $\partial \sigma' \in \partial \cB^n_{k+1}$ is zero, outside a collection $\cS$ for which the collection $\cS' = \{\sigma \cap \H^n : \sigma \in \cS\}$ is $p$-exceptional. Since $\cS'$ minorizes $\cS$, it follows that $\cS$ is $p$-exceptional by Lemma \ref{lem:Fuglede_properties} part \eqref{enum:fuglede_minorization}. The claim follows by Lemma \ref{lem:zero_ints_implies_weak_cl_Rn}.
\end{proof}

\begin{proof}[Proof of Lemma \ref{lem:chara_implication_2}]
	Let $\omega \in L^p_\loc(\wedge^k T^* M)$ be such that the integral of $\omega$ vanishes over $p$-a.e.\ $\sigma \in C_k^{\lip}(D; \Z) + \partial C_{k+1}^{\lip}(M; \Z)$. We extend $\omega$ by zero to $N = (\cR \setminus \partial M) \cup \intr D$, in which case $\omega \in L^p_\loc(\wedge^k T^* N)$. Let $B = B^n(0, 1)$ be the unit Euclidean ball. We cover $N$ with countably many bilipschitz charts $\phi_i \colon B \to U_i$, where either $U_i \subset \intr M$ or $U_i \subset \cR \setminus M$ or $\phi_i$ maps the upper and lower half-spaces of $B$ into $\intr(M)$ and $\cR \setminus M$ respectively. In the last case, the boundary $B \cap \partial \H^n_+$ is mapped into $D$.
	
	If $U_i \subset \cR \setminus M$, then $\phi_i^* \omega$ is identically zero a.e.\ on $B$ and $d \phi_i^* \omega = 0$. If $U_i \subset \intr M$, then we note that $p$-exceptional sets of $C_k^{\lip}(B; \R)$ and $C_k^{\lip}(U_i; \R)$ are in one-to-one correspondence by $(\phi_i)_*$ and our assumption yields that the integral of $\phi_i^* \omega$ vanishes over $p$-a.e.\ $\sigma \in C^{\lip}_k(B; \R)$ with $\partial \sigma = 0$. We fix a diffeomorphism $f \colon \R^n \to B$ and we similarly get that the integral of $f^* \phi_i^* \omega$ vanishes over $p$-a.e.\ $\sigma \in C^{\lip}_k(\R^n; \R)$ with $\partial \sigma = 0$. This is because we can write $\R^n$ as a union of increasingly large domains on which $f$ is bilipschitz and every $\sigma \in C^{\lip}_k(\R^n; \R)$ is contained in one of these domains. In particular, Lemma \ref{lem:zero_ints_implies_weak_cl_Rn} implies that $d f^* \phi_i^* \omega = 0$ and consequently $d \phi_i^* \omega = 0$ by use of Lemma \ref{lem:bilip_chain_rule} on subdomains. 
	
	For the remaining charts $\phi_i$ that map $B \cap \partial \H^n_+$ into $\intr D$, we similarly obtain that $d \phi_i^* \omega = 0$ by using Lemma \ref{lem:zero_ints_implies_weak_cl_Rn_bdry}.
	So the form $\phi_i^* \omega$ is weakly closed for a countable cover of $N$ by bilipschitz charts. It follows that $d\omega = 0$ weakly in $N = (\cR \setminus \partial M) \cup \intr D$. Consequently, $d\omega = 0$ weakly in $M$ and $\omega_T$ vanishes weakly on $D$.
\end{proof}

The other part of Theorem \ref{thm:integral_characterization_thm} that does not require cohomological tools is the following.

\begin{lemma}\label{lem:chara_implication_3}
	Suppose that $M$ is a closed Lipschitz $n$-submanifold in an oriented Lipschitz Riemannian $n$-manifold $\cR$ without boundary and $D$ is a closed Lipschitz $(n-1)$-submanifold in $\partial M$. Let $p \in [1, \infty)$ and $\omega \in L^p_\loc(\wedge^k T^* M)$. If $\omega = d\tau$ for some $\tau \in W^{d,p}_{T(D), \loc}(\wedge^{k-1} T^* M)$, then for $p$-a.e.\ $\sigma \in C_k^{\lip}(M; \Z)$ such that $\partial \sigma \in C_{k-1}^{\lip}(D; \Z)$, we have
	\[
		\int_{\sigma} \omega = 0.
	\]
\end{lemma}
\begin{proof}
	We may take a sequence of compact Lipschitz $n$-submanifolds $M_i \subset M$ such that $\bigcup_i M_i = M$, $\bigcup_i \partial M_i = \partial M$ and $D \cap \partial M_i$ is a closed Lipschitz $(n-1)$-submanifold of $\partial M_i$. If $\sigma \in C_k^{\lip}(M; \Z)$, we necessarily have that $\sigma \in C_k^{\lip}(M_i; \Z)$ for some $i$. Since the union of countably many $p$-exceptional sets is $p$-exceptional, we may assume that $\tau \in W^{d,p}_{T(D)}(\wedge^{k-1} T^* M)$ by replacing $M$ with one of the submanifolds $M_i$.
	
	By Lemma \ref{lem:W_infinity_approx_boundary}, we may approximate $\tau$ in the $\norm{\cdot}_{W^{d,p}}$-norm with $\tau_i \in W^{d,\infty}(\wedge^{k-1} T^* M)$ such that $\spt \tau_i \cap D = \emptyset$. Let $\omega_i = d\tau_i$. Suppose that $\sigma \in C_k^{\lip}(M; \Z)$ is such that $\partial \sigma \in C_k^{\lip}(D; \Z)$. Since $\tau_i$ vanishes in a neighborhood of $D$,
	\[
		\int_\sigma \omega_i = \int_{\partial \sigma} \tau_i = 0,
	\]
	by Proposition \ref{prop:Lipschitz_stokes}.
	Since $\omega_i \to \omega$ in the $L^p$-norm, we obtain by Lemma \ref{lem:fuglede_for_forms} that the integral of $\omega$ over $p$-a.e.\ such $\sigma$ vanishes.
\end{proof}

\subsection{The global implications}

The remaining two parts of Theorem \ref{thm:integral_characterization_thm} use results related to Sobolev cohomology. The first one we prove is the following.

\begin{lemma}\label{lem:chara_implication_1}
	Suppose that $M$ is a closed Lipschitz $n$-submanifold in an oriented Lipschitz Riemannian $n$-manifold $\cR$ without boundary and $D$ is a closed Lipschitz $(n-1)$-submanifold in $\partial M$. Let $p \in [1, \infty)$ and $\omega \in W^{d,p}_{T(D), \loc}(\wedge^k T^* M)$ with $d\omega = 0$. Then
	\[
		\int_{\sigma} \omega = 0,
	\]
	for $p$-a.e.\ $\sigma \in C_k^{\lip}(D; \Z) + \partial C_{k+1}^{\lip}(M; \Z)$.
\end{lemma}
\begin{proof}
	Since $d\omega = 0$, we may use the equivalence of cohomologies result of Corollary \ref{cor:Lp_Linfty_cohom_eq} to conclude that $\omega = \omega' + d\tau$, where $\omega' \in W^{d,\infty}_{T(D), \loc}(\wedge^k T^* M)$ with $d\omega' = 0$ and $\tau \in W^{d,p}_{T(D), \loc}(\wedge^{k-1} T^* M)$. By Lemma \ref{lem:chara_implication_3}, we know that the integral of $d\tau$ vanishes over $p$-a.e.\ $\sigma \in C_k^{\lip}(D; \Z) + \partial C_{k+1}^{\lip}(M; \Z)$, since $\partial \sigma \in C_{k-1}^{\lip}(D; \Z)$ for any such $\sigma$. We have hence reduced the claim to the corresponding claim for $\omega'$.
	
	Let $\sigma = \sigma' + \partial \sigma''$, where $\sigma' \in C^{\lip}_k(D; \Z)$. We immediately get via the Stokes Theorem from Proposition \ref{prop:Lipschitz_stokes} that
	\[
		\int_{\partial \sigma''} \omega' = \int_{\sigma''} d\omega' = \int_{\sigma''} 0 = 0.
	\]
	Moreover, by Lemma \ref{lem:canonical_int_vanishing_boundary_values}, we have that
	\[
		\int_{\sigma'} \omega' = 0.
	\]
	Hence, the integral of $\omega'$ over $\sigma$ vanishes and the claim follows.
\end{proof}

The last remaining implication is the following.

\begin{lemma}\label{lem:chara_implication_4}
	Suppose that $M$ is a closed Lipschitz $n$-submanifold in an oriented Lipschitz Riemannian $n$-manifold $\cR$ without boundary and $D$ is a closed Lipschitz $(n-1)$-submanifold in $\partial M$. Let $p \in (1, \infty)$ and $\omega \in L^{p}_\loc(\wedge^k \H^n)$ be such that for $p$-a.e.\ $\sigma \in C_k^{\lip}(M; \Z)$ with $\partial \sigma \in C_{k-1}^{\lip}(D; \Z)$, we have
	\[
		\int_{\sigma} \omega = 0.
	\]
	Then $\omega = d\tau$ weakly for some $\tau \in W^{d,p}_{T(D), \loc}(\wedge^{k-1} T^* M)$.
\end{lemma}
\begin{proof}
	By Lemma \ref{lem:chara_implication_2}, we have $\omega \in W^{d,p}_{T(D), \loc}(\wedge^k T^* M)$ and $d\omega = 0$. By Corollary \ref{cor:Lp_Linfty_cohom_eq}, we have $\omega = \omega' + d\tau$, where $\omega' \in W^{d,\infty}_{T(D), \loc}(\wedge^k T^* M)$ and $\tau \in W^{d,p}_{T(D), \loc}(\wedge^{k-1} T^* M)$. In particular, it is enough to show the claim for $\omega'$.
	
	Fix $[\sigma] \in H^{\lip}_k(M, D; \R)$. By Proposition \ref{prop:de_Rham_thm_in_Lip_setting}, there exists $c_{[\sigma]} \in \R$ such that
	\[
		\int_{\sigma'} \omega' = c_{[\sigma]},
	\]
	for every $\sigma' \in [\sigma]$. By our assumption and Lemma \ref{lem:chara_implication_3}, the integral of $\omega'$ vanishes over $p$-a.e.\ $\sigma' \in [\sigma]$. Since $[\sigma]$ is not $p$-exceptional by Proposition \ref{prop:no_exc_hom_classes}, we have that $c_{[\sigma]} = 0$. Referring again to Proposition \ref{prop:de_Rham_thm_in_Lip_setting}, we must have $[\omega'] = [0]$. This means that $\omega' = d\tau'$ for some $\tau' \in W^{d,\infty}_{T(D), \loc}(\wedge^{k-1} T^* M)$.
\end{proof}

The proof of Theorem \ref{thm:integral_characterization_thm} is complete by combining Lemmas \ref{lem:chara_implication_2}, \ref{lem:chara_implication_3}, \ref{lem:chara_implication_1} and \ref{lem:chara_implication_4}.

\section{Modulus with differential forms}\label{sec:moddifforms}
In this section we define a version of modulus based on differential forms.  This modulus will exhibit duality properties for dual homology classes, which we will prove in the following section. For the rest of this section, let $M$ be an closed Lipschitz $n$-submanifold in an oriented Lipschitz Riemannian $n$-manifold $\cR$ without boundary and let $D \subset \partial M$ be a closed Lipschitz $(n-1)$-submanifold of $\partial M$, with all the cases $\partial M = \emptyset$, $D = \emptyset$ and $\partial D = \emptyset$ permitted.

Before we give the definition for differential form modulus, we define the standard surface modulus $\Mod(\Sigma)$ for families $\Sigma$ of Lipschitz $k$-chains on $M$. The difficulty in the definition of the standard modulus is that the measure of a Lipschitz $k$-chain is not canonically defined, as discussed in Section \ref{sec:integrationlpforms}. However, we have given a valid definition of $p$-exceptional families of $k$-chains in Section \ref{sec:integrationlpforms} and by Lemma \ref{lem:measurenotdefined} and the surrounding discussion, we may define $\nu_\sigma$ for $p$-a.e.\ Lipschitz $k$-chain $\sigma$. 

Hence, for any $\Sigma \subset C^{\lip}_k(M; \R)$, we select a $p$-exceptional family $\Sigma_p \subset \Sigma$ such that, for every $\sigma \in \Sigma \setminus \Sigma_p$ and every Lipschitz $k$-simplex $\sigma_i$ in the sum representation of $\sigma$, we have that $\nu_{\sigma_i}$ exists and is uniformly comparable with $\mu_{\sigma_i}$, where $\mu_{\sigma_i}$ is the measure on $\sigma_i$ that is defined via charts in Section \ref{sec:integrationlpforms}. We then define
\[
    \Mod_p(\Sigma) = \Mod_p(\{\nu_\sigma :  \sigma \in \Sigma \setminus \Sigma_p\}).
\]
We also denote $\rho \in \essadm_p(\Sigma)$ if $\rho \in \adm(\{\nu_\sigma : \sigma \in \Sigma \setminus \Sigma_p\})$ for some $p$-exceptional $\Sigma_p$ as above. This definition of $p$-modulus of surface families is consistent with the Euclidean one by \eqref{eq:modulus_essential_def}. Moreover, for another choice of exceptional set $\Sigma_p$, we observe that $\{\nu_{\sigma} : \sigma \in (\Sigma_p \setminus \Sigma_p') \cup (\Sigma_p' \setminus \Sigma_p)\}$ is $p$-exceptional: Indeed, if $\sigma \in \Sigma_p' \setminus \Sigma_p$, then $\nu_\sigma$ is uniformly comparable with $\mu_\sigma$ and hence $\{\nu_\sigma : \sigma \in \Sigma_p' \setminus \Sigma_p\} \gtrsim \{\mu_\sigma : \sigma \in \Sigma_p'\}$ where the latter is $p$-exceptional. Hence, the value of $\Mod_p(\Sigma)$ is independent on the choice of $\Sigma_p$.

We also give a lemma related to exceptional sets of formal sums that is of use to us later. 
Given two families of formal sums $\Sigma$ and $\Sigma'$, we denote by $\Sigma + \Sigma'$ the family of all formal sums $\sigma + \sigma'$ where $\sigma \in \Sigma$ and $\sigma' \in \Sigma'$.

\begin{lemma}\label{lem:sum_exceptional_set}
	Let $\Sigma$ and $\Sigma'$ be families of Lipschitz $k$-chains on $M$ with $k < n$ and let $p \in [1, \infty)$. Then
	\[
		\moddens_p(\Sigma + \Sigma') \geq 2^{-p}\min(\moddens_p(\Sigma), \moddens_p(\Sigma')).
	\]
	In particular, if $\Sigma + \Sigma'$ is $p$-exceptional, then $\Sigma$ or $\Sigma'$ is $p$-exceptional.
\end{lemma}
\begin{proof}
	We fix $p$-exceptional families $\cE_{\Sigma}$, $\cE_{\Sigma'}$ and $\cE_{\Sigma + \Sigma'}$ such that
	\begin{gather*}
		\Mod_p(\Sigma) = \Mod_p(\{\nu_\sigma : \sigma \in \Sigma\setminus \cE_{\Sigma}\}), \qquad
		\Mod_p(\Sigma') = \Mod_p(\{\nu_{\sigma'} : \sigma' \in \Sigma'\setminus \cE_{\Sigma'}\}),\\
		\Mod_p(\Sigma + \Sigma') = \Mod_p(\{\nu_\sigma : \sigma \in (\Sigma+\Sigma')\setminus \cE_{\Sigma + \Sigma'}\}).
	\end{gather*}
	Note that by shrinking $\cE_{\Sigma + \Sigma'}$, we may assume that $\cE_{\Sigma + \Sigma'} \cap ((\Sigma \setminus \cE_\Sigma) + (\Sigma' \setminus \cE_{\Sigma'})) = \emptyset$. Indeed, if we suppose that $\sigma \in \Sigma \setminus \cE_\Sigma$ and $\sigma' \in \Sigma' \setminus \cE_{\Sigma'}$, then for all singular $k$-simplices $\sigma_i$ in the sum representations of $\sigma$ and $\sigma'$, the surface measures $\nu_{\sigma_i}$ exist and are comparable to $\mu_{\sigma_i}$. Hence, we may remove $\sigma + \sigma'$ from $\cE_{\Sigma + \Sigma'}$ if needed.
	
	Next, let $\sigma = a_1 \sigma_1 + \dots + a_m \sigma_m \in \Sigma \setminus \cE_{\Sigma}$. Let $\Sigma'_\sigma$ denote the collection of all $\sigma' \in \Sigma' \setminus \cE_{\Sigma'}$ such that $\nu_{\sigma + \sigma'} \neq \nu_{\sigma} + \nu_{\sigma'}$.
	Every $\sigma' = a_1' \sigma_1' + \dots + a'_{m'} \sigma'_{m'} \in \Sigma'_\sigma$ shares a positive-measured surface with $\sigma$; i.e.\ $\sigma_i = \sigma'_{i'}$ for some indices $i$ and $i'$, $a_i$ and $a'_{i'}$ have opposite signs and $\nu_{\sigma_i} \neq 0$. 
	
	The family $\Sigma'_i$ of all $\sigma' \in \Sigma' \setminus \cE_{\Sigma'}$ containing $\sigma_i$ in their formal sum representation is in fact $p$-exceptional, as long as $\nu_{\sigma_i} \neq 0$. Indeed, the density $\rho$ which is $\infty$ on the image of $\sigma_i$ and $0$ elsewhere is essentially admissible for $\Sigma'_i$.
	Since $k < n$, the surface $\sigma_i$ has zero $n$-dimensional volume and the integral of $\rho$ over $M$ vanishes, implying $p$-exceptionality of $\Sigma'_i$. Since $\Sigma'_\sigma$ is contained in the finite union of all $\Sigma'_i$ where $\nu_{\sigma_i} \neq 0$, we conclude that $\Sigma'_\sigma$ is $p$-exceptional.
	
	Next, suppose that $\rho$ is admissible for $\{\nu_\sigma : \sigma \in (\Sigma+\Sigma')\setminus \cE_{\Sigma + \Sigma'}\}$ and let $\sigma \in \Sigma\setminus \cE_{\Sigma}$. Using $\cE_{\Sigma + \Sigma'} \cap ((\Sigma \setminus \cE_\Sigma) + (\Sigma' \setminus \cE_{\Sigma'})) = \emptyset$, it follows that for every $\sigma' \in \Sigma' \setminus \cE_{\Sigma'}$ we have
	\[
		\int_M \rho \dd \nu_{\sigma + \sigma'} \geq 1.
	\]
	If $\sigma' \in \Sigma' \setminus (\cE_{\Sigma'} \cup \Sigma'_{\sigma})$, we then have $\nu_{\sigma + \sigma'} = \nu_\sigma + \nu_{\sigma'}$ and therefore one of the following two conditions is true:
	\begin{align*}
		\int_M 2\rho \dd \nu_{\sigma} \geq 1
		\quad \text{or} \quad
		\int_{M} 2\rho \dd \nu_{\sigma'} \geq 1.
	\end{align*}
	If the latter is the case for every $\sigma' \in \Sigma' \setminus (\cE_{\Sigma'} \cup \Sigma'_{\sigma})$, then since $\Sigma'_{\sigma}$ is $p$-exceptional, it follows that $2\rho \in \essadm_p(\{\nu_{\sigma'} : \sigma' \in \Sigma'\setminus \cE_{\Sigma'}\})$ and therefore
	\begin{equation}\label{eq:mod_est_second_term}
		\int_M \rho^p \vol_M \geq 2^{-p} \moddens_p(\Sigma').
	\end{equation}
	Hence, either the integral of $2\rho$ over $\nu_{\sigma}$ is at least 1 or \eqref{eq:mod_est_second_term} holds. By repeating this for all $\sigma \in \Sigma\setminus \cE_{\Sigma}$, we conclude that either $2\rho \in \adm(\{\nu_\sigma : \sigma \in \Sigma\setminus \cE_{\Sigma}\})$ or \eqref{eq:mod_est_second_term} holds. 
	In the former case, we have
	\begin{equation}\label{eq:mod_est_first_term}
		\int_M \rho^p \vol_M \geq 2^{-p} \moddens_p(\Sigma).
	\end{equation}
	So either \eqref{eq:mod_est_second_term} or \eqref{eq:mod_est_first_term} holds for an arbitrary $\rho \in \adm(\{\nu_\sigma : \sigma \in (\Sigma+\Sigma')\setminus \cE_{\Sigma + \Sigma'}\})$, completing the proof.
\end{proof}

We now proceed to discuss modulus with differential forms. Suppose that $\cc S$ is a family of Lipschitz $k$-chains on $M$. If $p \in [1,\infty)$, then we say that a $k$-form $\omega \in L^p(\wedge^k T^*M)$ is \emph{$p$-weakly admissible} for $\cS$ if
\begin{equation}\label{eq:admissibility_forms}
	\int_\sigma \omega \geq 1,
\end{equation}
for $p$-a.e.\ $\sigma \in \cS$. The set of such $\omega \in L^p(\wedge^k T^*M)$ is denoted $\essadm_p^k(\cS)$. We define the \emph{differential $p$-modulus} of $\cS$ by
\[
	\modform_p(\cS) = \inf_{\omega \in \essadm_p^k(\cS)} \int_M \abs{\omega}^p \vol_M.
\]
The case $k = 1$ corresponds to an unweighted version of the vector modulus defined by Aikawa and Ohtsuka \cite{Aikawa-Ohtsuka_vector-modulus}.
Unlike in the case of the standard modulus, the use of $p$-weakly admissible $k$-forms results in a different theory than if \eqref{eq:admissibility_forms} were required for every $\sigma \in \cS$, as noted by Aikawa and Ohtsuka \cite[p. 64]{Aikawa-Ohtsuka_vector-modulus}. 
For example, consider the case when $\cS = \{\sigma, \sigma'\}$, where $\sigma, \sigma' \colon \overline{B_{\R^k}(0, 1)} \to M$ and $\sigma(x_1, \dots, x_{k-1}, x_k) = \sigma'(x_1, \dots, x_{k-1}, -x_k)$. 
No form can possibly satisfy \eqref{eq:admissibility_forms} for both $\sigma$ and $\sigma'$, as the integrals are equal but with opposite signs. 
This would result in an infinite modulus for $\cS$ under a definition requiring \eqref{eq:admissibility_forms} for every element of $\cS$. However, assuming $k < n$, the shared image of $\sigma$ and $\sigma'$ has zero volume and hence $\moddens_p(\cS) = 0$. 
Consequently $\essadm_p^k(\cS) = L^p(\wedge^k M)$ and hence $\modform_p(\cS) = 0$.

The moduli $\moddens_p(\cS)$ and $\modform_p(\cS)$ may in general assume different values (see e.g.\ \cite[p. 201]{Freedman-He_Modulus-duality}). However, for every $\omega \in \essadm_p^k(\cS)$, we have that $\abs{\omega} \in \essadm_p(\cS)$ and 
\begin{equation}\label{eq:modulus_comparison}
	\moddens_p(\cS) \leq \modform_p(\cS).
\end{equation}

\subsection{Existence of Minimizers}
\begin{prop}\label{prop:minimizer_exists}
	Let $p \in [1, \infty)$ and let $\cS \subset C_k^{\text{lip}}(M;\bR)$. Then $\essadm_p^k(\cS)$ is a closed subset of $L^p(\wedge^k T^*M)$.  Additionally, if $\essadm_p^k(\cS) \neq \emptyset$, then there exists a form $\omega \in \essadm_p^k(\cS)$ such that
	\begin{equation}\label{eq:norm_minimizer}
		\int_M \abs{\omega}^p \vol_M = \modform_p(\cS).
	\end{equation}
	If $p > 1$, $\omega$ is unique.
\end{prop}
\begin{proof}
	For the closedness of $\essadm_p^k(\cS)$, if $\omega_i \in \essadm_p^k(\cS)$ and $\omega_i \to \omega$ in  $L^p(\wedge^k T^*M)$, then by Lemma \ref{lem:fuglede_for_forms},
	\[
		\int_\sigma \omega \geq \liminf_{i \to \infty} \int_\sigma \omega_{i} \geq 1
	\]
	for $p$-a.e.\ $\sigma \in \cS$.
	
	Let  $(\omega_i)$ be a norm-infimizing sequence in $\essadm_p^k(\cS)$. The  Banach-Alaoglu theorem yields a subsequence $(\omega_{i_j})$ and a limit $\omega \in L^p(\wedge^k M)$ such that $\omega_{i_j} \to \omega$ weakly. 
	By Mazur's lemma, there exists  a sequence $(\omega_l')$ such that every $\omega_l'$ is a convex combination of elements in the sequence $(\omega_{i_j})$ with $j \geq l$ and $\omega_l' \to \omega$ strongly. Since $\essadm_p^k(\cS)$ is closed under convex combinations, it follows that $\omega_l' \in \essadm_p^k(\cS)$ for every $l$. 
	We have already shown in this case that $\omega \in \essadm_p^k(\cS)$. 
	Finally,  $\norm{\omega}_{L^p} = \lim_{l \to \infty} \norm{\omega_l'}_{L^p} \leq \limsup_{j \to \infty} \smallnorm{\omega_{i_j}}_p$, which implies that $\omega$ satisfies \eqref{eq:norm_minimizer}.
	
	Assume now that $p > 1$, with goal of showing uniqueness of $\omega$. We may assume $\moddens_p(\cS) > 0$, as otherwise $\omega = 0$ is clearly the unique element of $\essadm_p^k(\cS)$ with the minimal $L^p$-norm. Suppose that $\omega' \in \essadm_p^k(\cS)$ is another minimizer of the $L^p$-norm in $\essadm_p^k(\cS)$. 
	Since $p > 1$, the $L^p$-norm is strictly convex. So either $\omega' = t\omega$ for some $t \in \R$ or $\norm{(\omega' + \omega)/2}_{L^p} < \norm{\omega'}_{L^p}/2 + \norm{\omega}_{L^p}/2$. 
	The latter is impossible since $(\omega' + \omega)/2 \in \essadm_p^k(\cS)$ and $\omega, \omega'$ are $L^p$-norm minimizers in $\essadm_p^k(\cS)$. 
	In the remaining case $\omega' = t\omega$ and $\abs{t} = 1$ since $\norm{\omega}_{L^p} = \norm{\omega'}_{L^p} \neq 0$. The admissibility of $\omega$ and $\omega'$ gives that $t = 1$.
\end{proof}

\subsection{Properties of Admissible Forms for Homology Classes}
We next prove that admissible forms for a homology family of surfaces are closed and satisfy appropriate boundary conditions.
Even though the family of admissible forms $\essadm_p^k(c)$ for $\modform_p(c)$, where $c$ is a homology class, includes no assumptions on their differentiability, any $p$-weakly admissible form $\omega$ will satisfy $d\omega = 0$ weakly.

\begin{lemma}\label{lem:admissible_are_weakly_closed}
	Let $c \in H_k^{\lip}(M,D; \Z)$ with $k < n$. Then every $\omega \in \essadm_p^k(c)$ satisfies $d\omega = 0$ in the weak sense and $\omega \in W^{d,p}_{T(D),\text{loc}}(\wedge^k T^*M)$.
\end{lemma}
\begin{proof}
	Suppose that either conclusion does not hold for $\omega \in \essadm_p^k(c)$. By Theorem \ref{thm:integral_characterization_thm}, there exists a family $\cS \subset C_k^{\text{lip}}(D;\bZ) +\partial C_{k+1}^{\lip}(M; \Z)$ such that $\cS$ is not $p$-exceptional and
	\[
		\int_\sigma \omega \neq 0,
	\]
	for all $\sigma \in \cS$. 
	For all $k \in \bN$, define the family $\cS_k$ by
	\begin{align*}
		\cS_k &= \left\{ \sigma \in C_k^{\text{lip}}(D;\bZ) +\partial C_{k+1}^{\lip}(M; \Z) : \int_\sigma \omega \leq -\frac{1}{k} \right\}.
	\end{align*}
	We also denote $-\cS_k = \{-\sigma : \sigma \in \cS_k\}$ and $a\cS_k = \{a\sigma : \sigma \in \cS_k\}$ for $a \in \Z$. 
	
	Since $\moddens_p$ is an outer measure and $\cS \subset \bigcup_{k \in \bN} \cS_k \cup (-\cS_k)$, we have that $\moddens_p(\cS_k) > 0$ or $\moddens_p(-\cS_k) > 0$ for some $k \in \bN$. However, $\moddens_p(\cS_k) = \moddens_p(-\cS_k)$, which implies that $\moddens_p(\cS_k) > 0$ for some $k \in \bN$. 
	It follows that for this specific $k$, $\moddens_p(a\cS_k) > 0$ for every $a \in \bN$.
	Additionally, $a\cS_k \subset C_k^{\text{lip}}(D;\bZ) + \partial C_{k+1}^{\lip}(M; \Z)$.
	
	For all $l \in \bN$, define the sets $\cC_l$ by
	\begin{align*}
		\cC_l &= \left\{ \sigma \in c : \int_\sigma \omega \leq l \right\}.
	\end{align*}
	We have that $c = \cC \cup \bigcup_{l \in \bN} \cC_l$, where $\cC$ is the $p$-exceptional set of elements of $c$ where the integral of $\omega$ is not defined. Since $\moddens_p(c) > 0$ by Proposition \ref{prop:no_exc_hom_classes}, the subadditivity of $\moddens_p$ again implies that $\moddens_p(\cC_l) > 0$ for some $l \in \bN$.
	
	Let $a$ be large enough so that $l - a/k < 1$ and consider the set $\cE = \cC_l + a\cS_k$. Then $\cE \subset c$ and moreover, for every $\sigma \in \cE$, we have
	\[
		\int_{\sigma} \omega \leq l - \frac{a}{k} < 1.
	\]
	Since $\omega \in \essadm_p^k(c)$, we must have that the set $\cE$ is $p$-exceptional. However, $\cE = \cC_l + a\cS_k$, where neither of the sets $\cC_l$ nor $a\cS_k$ is $p$-exceptional. This contradicts Lemma \ref{lem:sum_exceptional_set}, completing the proof of the claim.	
\end{proof}

We also have a converse result concerning exact forms.

\begin{lemma}\label{lem:adding_a_weakly_exact_form}
	Let $p \in [1, \infty)$ and $c \in H_k^{\lip}(M, D; \Z)$ with $k > 0$. If $\omega \in \essadm_p^k(c)$ and $\tau \in W^{d,p}_{T(D)}(\wedge^{k-1} T^*M)$, then $\omega + d\tau \in \essadm_p^k(c)$.
\end{lemma}
\begin{proof}
	We have, due to Lemma \ref{lem:chara_implication_3}, that the integral of $d\tau$ over $p$-a.e.\ $\sigma \in c$ is zero. It follows that the integrals of $\omega + d\tau$ and $\omega$ coincide over $p$-a.e.\ $\sigma \in c$ and hence $\omega + d\tau \in \essadm_p^k(c)$.
\end{proof}

\subsection{Minimizers are \emph{p}-harmonic}
For $p \in [1, \infty)$, a $k$-form $\omega \in L^p(\wedge^k M)$ is \emph{$p$-harmonic} if
\begin{gather}
	\label{enum:p-harm_d} d\omega = 0 \quad \text{and}\\
	\label{enum:p-harm_dstar} d( \abs{\omega}^{p-2} \hodge \omega) = 0.
\end{gather}

We next prove the following fact about cohomology minimizing forms on closed manifolds with boundary. Let $E = \partial M \setminus \intr D$, which is a closed Lipschitz $(n-1)$-submanifold of $\partial M$. 

\begin{lemma}\label{lem:p_harm_in_cohom_class}
	Let $p \in (1, \infty)$ and let $\omega \in L^p(\wedge^k T^*M)$ with $d\omega = 0$ weakly, where $0 < k < n$. Then there exists a unique $\omega_0 \in \omega + dW^{d,p}_{T(D)}(\wedge^{k-1} T^*M)$ such that $\omega_0$ is $p$-harmonic and $\abs{\omega_0}^{p-2} \omega_0 \in W^{d^*, q}_{N(E)}(\wedge^{k} T^* M)$, where $q^{-1} + p^{-1} = 1$. Moreover, $\omega_0$ is the unique $L^p$-norm minimizing element in $\omega + dW^{d,p}_{T(D)}(\wedge^{k-1} T^*M)$.  
\end{lemma}
\begin{proof}
	Let $\omega + d\tau_i$ be an $L^p$-norm minimizing sequence.  By the Banach-Alaoglu theorem, there exists a subsequence that converges to $\omega_0 \in L^p(\wedge^k T^*M)$. By taking convex combinations and applying Mazur's lemma, we may assume that the sequence converges strongly to a form $\omega_0 \in W^{d,p}(\wedge^k T^*M)$.
	
	We need to show that $\omega - \omega_0 \in dW^{d,p}_{T(D)}(\wedge^{k-1} T^*M)$, $\omega_0$ is $p$-harmonic with $\abs{\omega_0}^{p-2} \omega_0 \in W^{d^*, q}_{N(E)}(\wedge^{k} T^* M)$ and that no other element in $\omega + dW^{d,p}_{T(D)}(\wedge^{k-1} T^*M)$ minimizes the $L^p$-norm. Out of these, the uniqueness of the minimizer follows from the strict concavity of the $L^p$-norm.
	
	For exactness, consider $\sigma \in C_k^{\text{lip}}(M;\bZ)$ such that $\partial \sigma \in C_k^{\text{lip}}(D;\bZ)$.  By Lemma \ref{lem:fuglede_for_forms}, there exists a subsequence so that
	\begin{align*}
	    \int_\sigma \omega - \omega_0 = \lim_{j \to \infty} \int_\sigma d\tau_{i_j} = 0,
	\end{align*}
	for $p$-a.e.\ such $\sigma$.  By Theorem \ref{thm:integral_characterization_thm}, $\omega - \omega_0 = d\tau_0$ is exact and $\tau_0 \in W^{d,p}_{T(D)}(\wedge^{k-1} T^*M)$.
	Since $\omega$ is closed, we have that $\omega_0$ is also closed.
	
	To prove that $\omega_0$ is $p$-harmonic we use a variational argument.  Fix $\phi \in W^{d,\infty}_{T(\partial M)}(\wedge^{k-1}T^*M)$. By the minimizing property of $\omega_0$,
	\begin{align*}
	   0 &= \frac{d}{d\epsilon} \int_M |\omega_0 + \epsilon d\phi|^{p}\bigg|_{\epsilon = 0} \\
	   &= p \int_M d\phi \wedge |\omega_0|^{p-2}\hodge \omega_0.
	\end{align*}
	Since $\phi$ was an arbitrary flat form, we have that $\omega_0$ is $p$-harmonic.  
	In the case where $E = \partial M \setminus D$ is a closed Lipschitz $(n-1)$-submanifold of $\partial M$, let $\phi \in W^{d,\infty}_{T(D)}(\wedge^{k-1}T^*M)$. The calculation above still holds and we see that $|\omega_0|^{p-2}\hodge \omega_0$ has weakly vanishing tangential part on $E$. It follows that $\abs{\omega_0}^{p-2} \omega_0 \in W^{d^*, q}_{N(E)}(\wedge^{k} T^* M)$, completing the proof of existence for $\omega_0$ as in the statement. 

	For uniqueness, suppose that $\omega'$ is $p$-harmonic with $\abs{\omega'}^{p-2} \omega' \in W^{d^*, q}_{N(E)}(\wedge^{k} T^* M)$ and that $\omega_0 = \omega' + d\tau$ for some $\tau \in W^{d,p}_{T(D)}(\wedge^{k-1} T^*M)$. We compute that
	\begin{align*}
		\norm{\omega'}_{L^p}^p
		&= \int_M \omega' \wedge (\abs{\omega'}^{p-2}\hodge\omega')
		\\
		&= \int_M (\omega_0 - d\tau) \wedge (\abs{\omega'}^{p-2}\hodge\omega')\\
		&= \int_M \omega_0 \wedge (\abs{\omega'}^{p-2}\hodge\omega')
		+ (-1)^{k} \int_M \tau \wedge d(\abs{\omega'}^{p-2}\hodge\omega') \\
		&= \int_M \ip{\omega_0}{
		\abs{\omega'}^{p-2}\omega'}
		\leq \norm{\omega_0}_{L^p} \norm{\omega'}_{L^p}^{p-1}.
	\end{align*}
	Note that the integration by parts in the third line is justified: we can use Lemma \ref{lem:mixedboundarydata} to conclude that $\abs{\omega'}^{p-2}\hodge\omega' \wedge \tau$ has weakly vanishing tangential part on $\partial M$ and consequently the integral of $d(\abs{\omega'}^{p-2}\hodge\omega' \wedge \tau)$ over $M$ vanishes by a use of Theorem \ref{thm:integral_characterization_thm}.
	It follows that $\smallnorm{\omega'}_p \leq \norm{\omega_0}_{L^p}$.  Since $\omega_0$ is the unique $L^p$-norm minimizing element in $\omega + dW^{d,p}_{T(D)}(\wedge^{k-1} T^*M)$, we must have $\omega' = \omega_0$.
\end{proof}

Our observations regarding the set $\essadm_p^k(c)$ for a homology class $c$ imply that the minimal element in $\essadm_p^k(c)$ is $p$-harmonic.

\begin{prop}\label{prop:minimizer_is_p-harm_compact}
	Let $p \in (1, \infty)$ and let $c \in H_k^{\lip}(M, D; \Z)$ with $0 < k < n$. Suppose that $\essadm_p^k(c)\neq\emptyset$ and let $\omega \in \essadm_p^k(c)$ be the unique element with
	\[
		\modform_p(c) = \int_M \abs{\omega}^p \vol_M.
	\]
	Then $\omega$ is $p$-harmonic, $\omega$ has weakly vanishing tangential part on $D$ and $\abs{\omega}^{p-2}\omega$ has weakly vanishing normal part on $E$.
\end{prop}
\begin{proof}
	Recall that $\omega$ exists and is unique by Proposition \ref{prop:minimizer_exists}. 
	By Lemma \ref{lem:admissible_are_weakly_closed}, $d\omega = 0$ in the weak sense and $\omega$ has vanishing tangential part on $D$. For every $\tau \in W^{d,p}_{T(D)}(\wedge^{k-1} T^*M)$, we have that $\omega + d\tau \in \essadm_p^k(c)$ by Lemma \ref{lem:adding_a_weakly_exact_form}. Among such $\omega + d\tau$, there exists a unique $p$-harmonic form by Lemma \ref{lem:p_harm_in_cohom_class} with the correct boundary values. This $p$-harmonic element has the minimal $L^p$-norm. Since $\omega$ has the minimal $L^p$-norm in $\essadm_p^k(c)$, it follows that $\omega$ must be $p$-harmonic.
\end{proof}

\begin{proof}[Proof of Theorem \ref{thm:modcapconnection}]
    Proposition \ref{prop:minimizer_is_p-harm_compact} immediately implies the Theorem as long as $\essadm_p^k(c) \ne \emptyset$.  Since $c$ is not a torsion element, this follows from Proposition \ref{prop:de_Rham_thm_in_Lip_setting}.
\end{proof}

\section{Duality}\label{sec:duality}

In this section we connect the modulus-minimizing forms with their Poincar\'e duals using the Hodge duality of $p$-harmonic and $q$-harmonic forms. In particular, we show for compact $M$ that the $q$-harmonic Hodge dual of the modulus-minimizer of $c \in H_k^{\lip}(M,D; \Z)$ is a $q$-harmonic representative of the Poincar\'e dual $c^*_{PD} \in H^{n-k}_{\lip}(M,E; \Z)$ of $c$. For a general reference of such Poincar\'e duals, we refer to \cite[Theorem 3.43]{Hatcher_AlgTopo}.

For the duration of this section, let $M$ be a compact Lipschitz $n$-submanifold in an oriented Lipschitz Riemannian $n$-manifold $\cR$ without boundary. The set $D \subset \partial M$ will be a closed Lipschitz $(n-1)$-submanifold of $\partial M$ with boundary. Additionally, we have that $E = \partial M \setminus D$ is also a closed Lipschitz $(n-1)$-submanifold with boundary.

We begin with a Lemma connecting the Poincar\'e dual to integration.

\begin{lemma}\label{lem:poincare_duality_integration}
	Let $p, q \in [1, \infty]$ with $p^{-1} + q^{-1} \leq 1$, let $c \in H_k^{\lip}(M, D; \Z)$, let $c_\R$ be the element of $H_k^{\lip}(M, D; \R)$ containing $c$ and let $c^*_{PD, \R} \in H^{n-k}_{\lip}(M, E; \R)$ be the Poincar\'e dual of $c_{\R}$. Suppose that $c^*_{PD, \R}$ is mapped to $[\alpha] \in H_q^{n-k}(M, E)$ in the isomorphisms of Proposition \ref{prop:de_Rham_thm_in_Lip_setting} and Corollary \ref{cor:Lp_Linfty_cohom_eq} and suppose that $\omega \in W^{d,p}_{T(D), \loc}(\wedge^k T^*M)$ with $d\omega = 0$. Then for $p$-a.e.\ $\sigma \in c$, we have
	\[
		\int_\sigma \omega = \int_M \alpha \wedge \omega
	\]
\end{lemma}
\begin{proof}
	The Poincar\'e duality of $c_{\R}$ and $c^*_{PD, \R}$ means that $c_{\R} = [M] \cap c^*_{PD, \R} = [M] \cap [\alpha]$, where $[M]$ is the orientation class of $H_n^{\lip}(M, \partial M; \R)$. In particular, since the isomorphism of Proposition \ref{prop:poincare_dual_connection} takes the wedge product to the cup product and since the isomorphisms of Corollary \ref{cor:Lp_Linfty_cohom_eq} respect the wedge product $W^{d,p}_{T(D), \loc}(\wedge^k T^* M) \times W^{d, q}_{T(E), \loc}(\wedge^{n-k} T^* M) \to W^{d,1}_{T(\partial M), \loc}(\wedge^{n} T^* M)$ provided by Lemma \ref{lem:mixedboundarydata}, we have 
	\begin{align*}
		[\omega](c) = [\omega](c_{\R})
		= [\omega]([M] \cap [\alpha])
		= ([\alpha] \cup [\omega])([M])
		= [\alpha \wedge \omega](M)
		= \int_M \alpha \wedge \omega,
	\end{align*}
	where $[\omega]$ denotes the cohomology class of $\omega$ in $H_{p}^{k}(M, D)$. On the other hand, relying on Corollary \ref{cor:Lp_Linfty_cohom_eq}, we may select $\omega' \in [\omega] \cap W^{d,\infty}_{T(D), \loc}(\wedge^k T^* M)$, in which case $\omega = \omega' + d\tau$ for some $\tau \in W^{d,p}_{T(D), \loc}(\wedge^{k-1} T^* M)$. Thus, by Proposition \ref{prop:de_Rham_thm_in_Lip_setting} and Theorem \ref{thm:integral_characterization_thm}, we have
		\begin{align*}
		[\omega](c) &= [\omega'](c)
		= \int_\sigma \omega'
		= \int_{\sigma} \omega + d\tau
		= \int_{\sigma} \omega
	\end{align*}
	for $p$-a.e.\ $\sigma \in c$.
\end{proof}

\begin{prop}\label{prop:poincare_dual_connection}
	Let $c \in H_k^{\lip}(M,D; \Z)$, let $p, q \in (1, \infty)$ with $p^{-1} + q^{-1} = 1$. Suppose that $c\notin \tor(H_k^{\lip}(M,D; \Z))$. Then $\essadm_p^k(c)\neq \emptyset$ and for $q$-a.e.\ $\sigma' \in C_k^{\lip}(M; \Z)$ with $\partial \sigma' \in C_{k+1}^{\lip}(E; \Z)$, we have
	\[
		\int_{\sigma'} 	\frac{(-1)^{k(n-k)} \abs{\omega}^{p-2}\hodge\omega}{\norm{\omega}_{L^p}^p}
		= c_{\text{PD}}^*(\sigma'),
	\]
	where $c_{\text{PD}}^* \in H_{\lip}^{n-k}(M,E; \Z)$ is the Poincar\'e dual of $c$ and $\omega$ is the unique element in $\essadm_p^k(c)$ such that $\norm{\omega}_{L^p}^p = \modform_p(c)$.
\end{prop}
\begin{proof}
	Let $c_\R \in H_k^{\lip}(M, D; \R)$ be the real homology class containing the integral homology class $c$ and let $c_{PD, \R}^* \in H^{n-k}_{\lip}(M, E; \R)$ be the Poincar\'e dual of $c_\R$. Then $c_{PD}(\sigma') = c_{PD, \R}(\sigma')$ for every $\sigma' \in C_k^{\lip}(M; \Z)$ with $\partial \sigma' \in C_{k+1}^{\lip}(E; \Z)$. By the de Rham theorems of Proposition \ref{prop:de_Rham_thm_in_Lip_setting} and Corollary \ref{cor:Lp_Linfty_cohom_eq}, $c_{PD, \R}$ is represented by a unique $[\zeta] \in H_q(M, E)$. By Lemma \ref{lem:p_harm_in_cohom_class}, there exists a unique $q$-harmonic element in $[\zeta]$, which we may assume to be $\zeta \in W^{d,p}_{T(E), \loc}(\wedge^{n-k} T^* M)$. Since $c \notin \tor(H_k^{\lip}(M,D; \Z))$, we have $c_{PD, \R} \neq 0$ and consequently $\zeta$ is not identically zero.
	
	Our goal is to show that $\essadm_p^k(c)$ is non-empty and that the $L^p$-norm minimizing element $\omega$ satisfies
	\begin{equation}\label{eq:lp_hodge_dual}
		\zeta = \frac{(-1)^{k(n-k)} \abs{\omega}^{p-2}\hodge\omega}{\norm{\omega}_{L^p}^p}.
	\end{equation}
	We show \eqref{eq:lp_hodge_dual} by showing the equivalent equality:
	\begin{equation}\label{eq:lq_hodge_dual}
		\omega = \frac{\abs{\zeta}^{q-2}\hodge\zeta}{\norm{\zeta}_{L^q}^q}.
	\end{equation}
	Let $\omega' = \norm{\zeta}_{L^q}^{-q} \abs{\zeta}^{q-2}\hodge\zeta$, with the aim of showing $\omega' = \omega$. We note that
	\begin{equation}\label{eq:dual_norm_inverse}
		\norm{\omega'}_{L^p} = \frac{1}{\norm{\zeta}_{L^q}^{q}} 
		\left( \int_M \abs{\zeta}^{(q-1)p} \right)^\frac{1}{p}
		=  \frac{1}{\norm{\zeta}^{q}_{L^q}} \norm{\zeta}_{L^q}^{\frac{q}{p}}
		= \norm{\zeta}_{L^q}^{\frac{q(1-p)}{p}}
		= \norm{\zeta}_{L^q}^{-1}.
	\end{equation}

	Since $\zeta$ is $q$-harmonic, a calculation shows that $d\omega' = 0$. Hence, Lemma \ref{lem:poincare_duality_integration} yields that
	\[
		\int_{\sigma} \omega' = \int_M \zeta \wedge \omega' = \norm{\zeta}_{L^q}^{-q} \int_M \zeta \wedge \abs{\zeta}^{q-2}\hodge\zeta = 1
	\]
	for $p$-a.e.\ $\sigma \in c$. Therefore $\omega' \in \essadm_p^k(c)$ and notably $\essadm_p^k(c)$ is non-empty.
	
	Finally, we show that $\norm{\omega'}_{L^p} \leq \norm{\omega}_{L^p}$, which will imply that $\omega' = \omega$ since $\omega$ is the unique minimizer in $\essadm_p^k(c)$.
	We have again by Lemma \ref{lem:poincare_duality_integration} that
	\begin{align*}
		\int_\sigma \omega = \int_M \zeta \wedge \omega,
	\end{align*}
	for $p$-a.e.\ $\sigma \in c$.  Since $c$ is not $p$-exceptional by Proposition \ref{prop:no_exc_hom_classes}, we may choose a $\sigma \in c$ so that this holds and
	\begin{align*}
		1 \le \int_\sigma \omega \le \|\omega\|_{L^p}\|\zeta\|_{L^q}.
	\end{align*}
	By \eqref{eq:dual_norm_inverse},
	\begin{align*}
		\|\omega'\|_{L^p} \le \|\omega\|_{L^p},
	\end{align*}
	and the proof is therefore complete.

\end{proof}

The proof of Proposition \ref{prop:poincare_dual_connection} provides a way to find the minimizer of $\modform_p$ using the Hodge star and Poincar\'e duality.
Let $c \in H_k^{\lip}(M,D;\bZ)$, then $c_{PD}^*$ is represented by a $q$-harmonic form $\zeta$.  The minimizer of $\modform_p(c)$ is
\begin{align*}
    \omega = \frac{\abs{\zeta}^{q-2}\hodge\zeta}{\norm{\zeta}_{L^q}^q}.
\end{align*}
Since both of these operations are injective, we have shown the following proposition.
\begin{prop}\label{prop:modminimizerinjective}
    Let $\cF_p \colon H_k^{\lip}(M,D;\bZ) \to L^p(\wedge^kT^*M)$ give the unique minimizer of $\modform_p$ for a homology class. The operator $\cF_p$ is injective.
\end{prop}

We can now prove Theorem \ref{thm:mainthmrelative} for a general compact, oriented Lipschitz manifold $M$ of dimension $n$ that is a subset of $\cc R$, an $n$-dimensional Lipschitz Riemannian manifold without boundary.
\begin{named}{Theorem \ref{thm:mainthmrelative}}
    Let $p,q \in (1,\infty)$ satisfy $p^{-1} + q^{-1} = 1$.  If $c \in  H_k^{\lip}(M,D;\bZ)$ is a non-torsion element, then there exists a unique element $c' \in H_{n-k}^{\lip}(M,E;\bR)$ such that the Poincar\'e dual of $c$ maps $c'$ to $1$ and
    \begin{align*}
        \dMod_p(c)^{1/p}\dMod_q(c')^{1/q} = 1.
    \end{align*}
\end{named}
\begin{proof}
	Let $c_{PD}^*$ be the Poincar\'e dual of $c$. Then $c_{PD}^*$ is a nontorsion element of $H^{n-k}_{\lip}(M,E; \Z)$. 
	Let $\omega$ be the unique $p$-modulus minimizing element in $\essadm_p^k(c)$ and let $\zeta = (-1)^{k(n-k)} \norm{\omega}_{L^p}^{-p} \abs{\omega}^{p-2}\hodge\omega$.
	Due to Proposition \ref{prop:poincare_dual_connection}, $c_{PD}^*$ maps to $[\zeta]$ via the canonical map $H^{n-k}_{\lip}(M, E; \Z) \to H^{n-k}_{\lip}(M, E; \R)$ followed by the de Rham isomorphism.
	
	We let $c' \in H_{n-k}^{\lip}(M, E; \R)$ be the Poincar\'e dual of the cohomology class $[(-1)^{k(n-k)} \omega] \in H_p^k(M, D; \R)$. By Lemma \ref{lem:poincare_duality_integration}, we conclude that for $q$-a.e.\ every $\sigma' \in c'$, we have
	\[
	\int_{\sigma'} \zeta = \int_M (-1)^{k(n-k)} \omega \wedge \zeta = \int_M \omega \wedge \norm{\omega}_{L^p}^{-p} \abs{\omega}^{p-2} \hodge \omega = 1.
	\]
	It follows that $\zeta$ is weakly $q$-admissible for $c'$ and that $\zeta$ has integral $1$ over $q$-a.e.\ element of $c'$.
	
	Suppose that $\zeta'$ is the unique minimizer of the $q$-modulus for the class $c'$. Then $\zeta'$ is $q$-harmonic and has weakly vanishing tangential part on $E$ by Proposition \ref{prop:minimizer_is_p-harm_compact}. Since $c'$ is not $q$-exceptional by Proposition \ref{prop:no_exc_hom_classes}, we consequently have that $[\zeta'](c') = 1$. Hence, there exists $\sigma' \in c'$ such that the following computation, which again uses Lemma \ref{lem:poincare_duality_integration} along with \eqref{eq:dual_norm_inverse}, is valid: 
	\[
	1 = \int_{\sigma'} \zeta' = \int_M (-1)^{k(n-k)} \omega \wedge \zeta'
	\leq \norm{\omega}_{L^p} \norm{\zeta'}_{L^q} = \norm{\zeta}_{L^q}^{-1} \norm{\zeta'}_{L^q}.
	\]
	
	Thus, $\norm{\zeta}_{L^q} \leq \norm{\zeta'}_{L^q}$. But since $\zeta \in \essadm_q^{n-k}(c')$ and since $\zeta'$ is the unique $L^q$-norm minimizer in $\essadm_q^{n-k}(c')$, we must have $\zeta = \zeta'$. In conclusion, \eqref{eq:dual_norm_inverse} implies that
	\[
	\modform_p(c)^{\frac{1}{p}} \modform_q(c')^{\frac{1}{q}}
	= \norm{\omega}_{L^p} \norm{\zeta'}_{L^q}
	= \norm{\omega}_{L^p} \norm{\zeta}_{L^q}
	= 1.
	\]
	
	It remains to show that $c'$ is unique. For this, suppose that $a \in H_{n-k}^{\lip}(M,E;\bR), c_{PD}^*(a) = 1$ and
	\begin{align*}
		\modform_p(c)^{1/p}\modform_q(a)^{1/q} = 1.
	\end{align*}
	Let $\xi$ denote the $L^q$-norm minimizing element in $\essadm_q^{n-k}(a)$, in which case our assumption on the moduli yields $\norm{\xi}_{L^q} = \norm{\omega}_{L^p}^{-1}$. If we denote $\omega' = (-1)^{k(n-k)} \norm{\xi}_{L^q}^{-q} \abs{\xi}^{q-2} \hodge \xi$, then $d\omega' = 0$ and $\omega'$ has vanishing tangential part on $D$ by Proposition \ref{prop:minimizer_is_p-harm_compact}.
	Additionally, by \eqref{eq:dual_norm_inverse}, the norms must satisfy $\norm{\omega'}_{L^p} = \norm{\xi}_{L^q}^{-1} = \norm{\omega}_{L^p}$. Moreover, by Proposition \ref{prop:poincare_dual_connection}, $[\omega']$ is the Poincar\'e dual of $a$. 
	
	By our assumption that $c_{PD}^*(a) = 1$ and since $c_{PD}^*$ corresponds to $[\zeta]$ in the de Rham map, we may again select a suitable $\sigma' \in a$ and use Lemma \ref{lem:poincare_duality_integration} and \eqref{eq:dual_norm_inverse} to obtain that
	\[
	1 = \int_{\sigma'} \zeta = \int_M \omega' \wedge \zeta
	\leq \norm{\omega'}_{L^p} \norm{\zeta}_{L^q}
	= \norm{\omega}_{L^p} \norm{\omega}_{L^p}^{-1} = 1. 
	\]
	It follows that
	\[
	\int_M \omega' \wedge \zeta = \norm{\omega'}_{L^p} \norm{\zeta}_{L^q}.
	\]
	By the equality condition of H\"older's inequality,
	\[
	\norm{\omega'}_{L^p}^{-1} \omega' = (-1)^{k(n-k)} \norm{\zeta}_{L^q}^{-q+1} \abs{\zeta}^{q-2} \hodge \zeta.
	\] 
	By using \eqref{eq:dual_norm_inverse} and $\omega = \norm{\zeta}_{L^q}^{-q} \abs{\zeta}^{q-2} \hodge \zeta$, it follows that $\norm{\omega'}_{L^p}^{-1} \omega' = (-1)^{k(n-k)} \norm{\omega}_{L^p}^{-1} \omega$ and since $\norm{\omega'}_{L^p} = \norm{\omega}_{L^p}$, we therefore have $\omega' = (-1)^{k(n-k)} \omega$. Since $a$ is the Poincar\'e dual of $[\omega']$ and $c'$ is the Poincar\'e dual of $[(-1)^{k(n-k)} \omega]$, we conclude that $a = c'$.

\end{proof}

As a consequence, we prove Theorem \ref{thm:mainthmZ}, a duality result for a special class of manifolds.  As in Theorem $\ref{thm:mainthmrelative}$, let $M$ be a compact, oriented Lipschitz Riemannian $n$-manifold that is a subset of $\cc R$, a Lipschitz Riemannian $n$-dimensional manifold.
\begin{named}{Theorem \ref{thm:mainthmZ}}
	Let $p,q \in (1,\infty)$ satisfy $p^{-1} + q^{-1} = 1$.  If $H_k^{\lip}(M,D;\bZ) \cong H_{n-k}^{\lip}(M,E;\bZ) \cong \bZ$ and $c \in  H_k^{\lip}(M,D;\bZ)$ and $c' \in H_{n-k}^{\lip}(M,E;\bZ)$ are the generators of the homology groups, then
    \begin{align*}
        \dMod_p(c)^{1/p}\dMod_q(c')^{1/q} = 1.
    \end{align*}
\end{named}
\begin{proof}
	Let $a \in H_{n-k}^{\lip}(M, E; \R)$ be the unique element provided by Theorem \ref{thm:mainthmrelative} such that $c^*_{PD}(a) = 1$ and $\modform_p(c)^{1/p} \modform_q(a)^{1/q} = 1$. Since $c$ and $c'$ are generators, we must have $c^*_{PD}(c') = \pm 1$. Moreover, $c^*_{PD}$ when understood as a map $H_{n-k}^{\lip}(M, E; \R) \to \R$ is a non-zero linear map between 1-dimensional spaces, so it is necessarily injective. It follows that the only classes in $H_{n-k}^{\lip}(M, E; \R)$ that $c^*_{PD}$ maps to $\pm 1$ are given by $\pm c'$ after the relevant natural identifications. Hence, we must have $a = \pm c'$. In particular, $a$ has integer coefficients due to $c'$ having integer coefficients and the desired duality property holds for $c$ and $c'$ due to it holding for $c$ and $a$.	
\end{proof}

Corollary \ref{cor:modineq} follows immediately from Equation \eqref{eq:modulus_comparison}.  It is a generalization of \cite[Theorem 1.1]{Lohvansuu_duality}.

\section{Appendix}
\subsection{Proofs of Statements in Section 4}
In the first part of the appendix, we sketch the proofs of Proposition \ref{prop:wolfe_int_properties} and Lemma \ref{lem:bilip_change_of_vars}, which we expect to be well known to any readers familiar with the theory of integration of flat forms presented in the books of Federer \cite{Federer_book} and Whitney \cite{Whitney_book}. Our exposition relies mostly on the results of Federer \cite{Federer_book}.

\begin{proof}[{Proof of Proposition \ref{prop:wolfe_int_properties}}]
	Let $\sigma \colon \Delta_k \to \Omega$ be a singular Lipschitz $k$-simplex. Note that we may extend $\sigma$ to a map $\sigma \colon \R^k \to \R^n$ with the same Lipschitz constant by the Kirszbraun extension. The formula
	\[
		[\sigma](\omega) = \int_\sigma \omega = \int_{\Delta_k} \sigma^* \omega
	\]
	for $\omega \in C^\infty_0(\wedge^k \R^n)$ defines a $k$-current $[\sigma]$ on $\R^n$. In particular, by the integral formula \cite[4.1.25]{Federer_book}, this current is given by the push-forward $[\sigma] = \sigma_* [\Delta_k]$, where $[\Delta_k]$ denotes the $k$-current on $\R^k$ of integration over $\Delta_k$ and $\sigma_*$ is defined for Lipschitz maps as in \cite[4.1.14]{Federer_book}. Notably, the current $\sigma_* [\Delta_k]$ flat by \cite[4.1.14]{Federer_book}, since $[\Delta_k]$ is flat due to being normal. 
	
	By the result given in \cite[4.1.18]{Federer_book}, we can find differential forms $\xi \in L^1(\wedge^{n-k} T^* \R^n)$ and $\zeta \in L^1(\wedge^{n-k-1} T^* \R^n)$ such that $\spt(\xi) \cap \spt(\zeta)$ is compactly contained in $\Omega$ and
	\begin{equation}\label{eq:Federer_flat_representation}
		[\sigma](\omega) = \int_{\R^n} (\omega \wedge \xi + d\omega \wedge \zeta)
	\end{equation}
	for all $\omega \in C^\infty_0(\wedge^k T^* \R^n)$. This allows for the definition of the canonical integration operator, as the right hand side of \eqref{eq:Federer_flat_representation} defines an integral for $\omega \in W^{d, \infty}(\wedge^k T^* \Omega)$ over $\sigma$. This definition is independent on $\xi$ and $\zeta$: Indeed, if $\xi'$ and $\zeta'$ are another such choice, then \eqref{eq:Federer_flat_representation} implies that $d(\zeta - \zeta') = \xi - \xi'$ weakly and hence approximating $\zeta - \zeta'$ in $W^{d,1}_0(\wedge^{n-k-1} T^* \Omega)$ with smooth compactly supported forms shows that the defined integrals are the same. The above definition then generalizes to $\sigma \in C^{\lip}_k(\Omega; \R)$ linearly.
	
	Property \eqref{enum:wolfe_convergence} immediately follows due to dominated convergence, as the absolute value of the integrand in \eqref{eq:Federer_flat_representation} is dominated by $(\sup_j \norm{\omega_j}_{L^\infty}) \abs{\xi} + (\sup_j \norm{d\omega_j}_{L^\infty}) \abs{\zeta} \in L^1(\Omega)$. Property \eqref{enum:wolfe_smoothforms} also immediately follows from the definition, since if $\omega$ is smooth, then we can find a $\omega' \in C^\infty_0(\wedge^k T^* \R^n)$ such that $\omega' = \omega$ in $\spt(\xi) \cap \spt(\zeta)$. Property \eqref{enum:wolfe_restriction} is also immediate, since if the image of $\sigma$ is in $\Omega \cap \Omega'$, we may choose $\spt(\xi) \cap \spt(\zeta) \subset \Omega \cap \Omega'$.
	
	For property \eqref{enum:wolfe_stokes}, we refer again to \cite[4.1.14]{Federer_book} where it is pointed out that $\sigma_* \partial = \partial \sigma_*$. This implies that $[\partial \sigma]$ is the distributional boundary of the current $[\sigma]$. Hence, we get the desired Stokes theorem for $\omega \in C^\infty_0(\wedge^{k} T^* \R^n)$. A version for $\omega \in C^\infty(\wedge^{k} T^* \Omega)$ again follows by considering a $\omega' \in C^\infty_0(\wedge^{k} T^* \Omega)$ that coincides with $\omega$ on $\spt(\xi) \cap \spt(\zeta)$ and the full version for $\omega \in W^{d, \infty}(\wedge^k T^* \Omega)$ follows by mollifying $\omega$ and applying property \eqref{enum:wolfe_convergence}.
	
	For property \eqref{enum:wolfe_borelforms}, we again approximate the $\omega \in W^{d, \infty}(\wedge^k T^* \Omega)$ smoothly by mollification in some domain $\Omega' \subset \Omega$. The mollified $\omega_j$ satisfy property \eqref{enum:wolfe_smoothforms}. We have convergence on the left side by property \eqref{enum:wolfe_convergence} and after moving to a subsequence of $\omega_j$, the Fuglede Lemma \ref{lem:fuglede_for_forms_eucl} yields convergence on the right hand side on $p$-a.e.\ $\sigma \in C^{\lip}_k(\Omega'; \R)$, for every $p \in [1, \infty)$. Since the image of every $\sigma \in C^{\lip}_k(\Omega; \R)$ is compactly contained in $\Omega$, we hence get property \eqref{enum:wolfe_borelforms} by using an increasing sequence of subdomains.
	
	The final property \eqref{enum:wolfe_continuity} follows from the fact that \cite[4.1.18]{Federer_book} allows us to select $\xi \in L^1(\wedge^{n-k} T^* \R^n)$ and $\zeta \in L^1(\wedge^{n-k-1} T^* \R^n)$ for $\sigma + \partial \tau$ such that $\norm{\xi}_{L^1} + \norm{\zeta}_{L^1} \leq \nu_{\sigma}(\Omega) + \nu_{\tau}(\Omega) + \eps$ for any $\eps > 0$. This fact is stated using the flat norm of $\sigma + \partial \tau$, but the flat norm of $\sigma + \partial \tau$ is bounded from above by $\nu_\sigma(\Omega) + \nu_\tau(\Omega)$. 
\end{proof}

\begin{proof}[{Proof of Lemma \ref{lem:bilip_change_of_vars}}]
	Let $L$ be the bilipschitz constant of $f$. We consider first the case $p = \infty$. For this, fix a Lipschitz $k$-simplex $\sigma \colon \Delta_k \to \Omega$. We again mollify $f$ in some region $\Omega' \subset \Omega$ containing $\sigma(\Delta_k)$, achieving smooth maps $f_i \colon \Omega' \to \R^n$ such that $\abs{f_i - f} \to 0$ uniformly, $\abs{Df_i} \leq L$ a.e.\ and $Df_i \to Df$ pointwise a.e. We again refer to \cite[4.1.14 and 4.1.25]{Federer_book}, notably to the fact that $f_* \sigma_* = (f \circ \sigma)_*$ under Federer's definition of the push-forward, to obtain that $[f_* \sigma] = f_* [\sigma]$. Notably, for $\omega \in C^\infty_0(\wedge^k T^*\R^n)$, this definition immediately gives that $(f_* [\sigma])(\omega) = \lim_{i \to \infty} ((f_i)_* [\sigma])(\omega)$. 
	
	Let $\omega \in C^\infty_0(\wedge^k T^* \R^n)$, in which case $f^* \omega \in W^{d, \infty}(\wedge^k T^* \Omega)$ with $f^* d\omega = d f^* \omega$ by Lemma \ref{lem:bilip_chain_rule}. Moreover, the absolute continuity of $\omega$ yields that $f_i^* \omega \to f^* \omega$ in the sense of property \eqref{enum:wolfe_convergence}. Using these observations and a bilipschitz change of variables yields
	\begin{align*}
		\int_{f_* \sigma} \omega &= (f_* [\sigma])(\omega) \lim_{i \to \infty} (f_i)_*[\sigma](\omega) \\
		&= \lim_{i \to \infty} \int_\sigma f_i^* \omega\\
		&= \int_\sigma f^* \omega\\
		&= \int_\Omega (f^*\omega \wedge \xi + f^*d\omega \wedge \zeta) \\
		&= \int_{f(\Omega)} (\omega \wedge (f^{-1})^*\xi + d\omega \wedge (f^{-1})^*\zeta),
	\end{align*}
	where $\xi$ and $\zeta$ are as in \eqref{eq:Federer_flat_representation}. In particular, this shows that $\xi' = (f^{-1})^* \xi$ and $\zeta' = (f^{-1})^* \zeta$ yield a valid choice of $L^1$-forms to represent integration over $f_* \sigma$.
	
	For any $\omega \in W^{d,p}(\wedge^k T^* \Upsilon)$, we use Lemma \ref{lem:bilip_chain_rule} and a bilipschitz change of variables to compute that
	\begin{align*}
		\int_{f_* \sigma} \omega &= \int_{f(\Omega)} (\omega \wedge (f^{-1})^*\xi + d\omega \wedge (f^{-1})^*\zeta)\\
		&= \int_{\Omega} (f^*\omega \wedge \xi + f^*d\omega \wedge \zeta)\\
		&= \int_{\Omega} (f^*\omega \wedge \xi + df^*\omega \wedge \zeta)\\
		&= \int_{\sigma} f^* \omega.
	\end{align*}
	This completes the proof of the case $p = \infty$. The case $p < \infty$ then follows by approximating $\omega \in W^{d,p}(\wedge^k T^*\Omega)$ smoothly by $\omega_j \in C^\infty(\wedge^k T^* \Omega)$, using the case $p = \infty$  on $\omega_j$ in a subdomain $\Omega' \subset \Omega$ and by using Lemma \ref{lem:fuglede_for_forms_eucl}.
\end{proof}

\subsection{Proof of the Sobolev de Rham theorem in Section 5}

In this second part of the appendix, we sketch the main points of the proof of Proposition \ref{prop:de_Rham_thm_in_Lip_setting}. The reader is assumed to be familiar with the standard terminology of sheaf theory and sheaf cohomology.

\begin{proof}[Sketch of proof of Proposition \ref{prop:de_Rham_thm_in_Lip_setting}]
	By Lemma \ref{lem:canonical_int_vanishing_boundary_values}, the integration map over Lipschitz chains yields a linear presheaf morphism from the sheaves $\cW_{\infty, D}^k \colon U \mapsto W^{d, \infty}_{T(D), \loc}(\wedge^k T^* U)$ to the presheaves $\cL_{D}^k = \{ f \in \hom (C_k^{\lip}(U; \Z), \R) : f\vert_{C_k^{\lip}(U \cap D; \Z)} \equiv 0 \}$. Moreover, by the Stokes theorem of Proposition \ref{prop:Lipschitz_stokes}, this map commutes with the differentials of the corresponding complexes. The sheaves $\tilde{\cL}_D^k$ generated by $\cL_{D}^k$ are flabby, as Lipschitz cochains can be trivially zero-extended to larger domains. Moreover, like $\cW_{\infty, D}^k$, the sheaves $\tilde{\cL}_D^k$ also form a resolution of $\cR_D$, since Lipschitz cochains on balls/half-balls of $\R^n$ satisfy a Poincar\'e lemma. Hence, relying on \cite[II.4.2]{Bredon_book}, the integration map yields an isomorphism from $H^\infty(M, D; \R)$ to the cohomology of $\tilde{\cL}_{D}^k(M)$. Moreover, the cohomologies of $\tilde{\cL}_{D}^k(M)$ and $\cL_{D}^k(M)$ are naturally isomorphic by an argument based on subdivision of simplices: an exposition of the argument in the standard case of continuous singular cochains can be found in \cite[Section 5.32]{Warner_book}, requiring essentially no changes for our case of Lipschitz singular cochains vanishing on $D$.
	
	For the product structure, we note that $\wedge$ induces a product on $W^{d,\infty}_{\loc}(\wedge^* T^*U)$ which satisfies $f \wedge g = fg$ for functions $f, g \in W^{d, \infty}_{\loc}(\wedge^0 T^* U)$ and $d(\omega \wedge \omega') = d\omega \wedge \omega' + (-1)^k \omega \wedge d\omega'$ for $\omega \in W^{d, \infty}_{\loc}(\wedge^k T^* U)$, $\omega' \in W^{d, \infty}_{\loc}(\wedge^l T^* U)$. Since the graded differential sheaf $\cW_{\infty}^* : U \mapsto W^{d, \infty}_\loc(\wedge^* T^* U)$ forms a fine torsionless resolution of the constant sheaf on $M$, we have that $\cW_\infty^* \otimes \cA$ is a fine resolution for any sheaf $\cA$ of vector spaces (see e.g.\ \cite[Corollary II.9.18]{Bredon_book}). By the discussion in \cite[II.7.5]{Bredon_book}, it follows that $\wedge$ computes the cup product of sheaf cohomology. In particular, the wedge product $(\cW_{\infty}^* \otimes \cR_D)(M) \otimes (\cW_{\infty}^* \otimes \cR_{E})(M) \to (\cW_{\infty}^* \otimes \cR_{\partial M})(M)$ computes the cup product  $H^*(M; \cR_D) \otimes H^*(M; \cR_{E}) \to  H^*(M; \cR_{\partial M})$ in sheaf cohomology, where we use the identification $\cR_{D} \otimes \cR_{E} \cong \cR_{\partial M}$. 
	
	The space $(\cW_{\infty}^* \otimes \cR_D)(U)$ is isomorphic to $\{\omega \in W^{d,\infty}_{\loc}(\wedge^* T^* U) : \spt \omega \cap D = \emptyset\}$. Hence, $\cW_{\infty}^* \otimes \cR_D$ is naturally isomorphic to a sub-resolution of $\cW_{\infty, D}^*$. The same is true for $E$ and $\partial M$ and the embeddings respect wedge products. Since the wedge product $W^{d,\infty}_{T(D), \loc}(\wedge^* T^*M) \times W^{d,\infty}_{T(E), \loc}(\wedge^* T^*M) \to W^{d,\infty}_{T(\partial M), \loc}(\wedge^* T^*M)$ is well defined by Lemma \ref{lem:mixedboundarydata}, we may use \cite[II.4.2]{Bredon_book} to conclude that this wedge product computes the cup product $H^*(M; \cR_D) \times H^*(M; \cR_{E}) \to  H^*(M; \cR_{\partial M})$. 
	
	Similar ideas show that the cup product of singular Lipschitz cohomology computes the cup product of sheaf cohomology. Indeed, recall that the cup product of relative Lipschitz cohomology is defined by a cup product map $\cL_{D}^* \otimes \cL_{E}^* \to \cL_{D + E}^*$, where $\cL_{D + E}^*$ denotes cochains vanishing on every element of $C_*^{\lip}(D; \Z) + C_*^{\lip}(E; \Z)$. By a similar proof as for $\cL_{D}^*$, the sheaves $\tilde{\cL}_{D + E}^*$ generated by $\cL_{D + E}^*$ form a resolution of $\cR_{\partial M}$ and the cohomologies of $\tilde{\cL}_{D + E}^*(M)$ and $\cL_{D + E}^*(M)$ are naturally isomorphic. Since $\tilde{\cL}_{\emptyset}^*$ form a flabby resolution of the constant sheaf of $M$, the same argument as above leads to the conclusion that the cup product $(\tilde{\cL}_{\emptyset}^* \otimes \cR_D)(M) \otimes (\tilde{\cL}_{\emptyset}^* \otimes \cR_{E})(M) \to (\tilde{\cL}_{\emptyset}^* \otimes \cR_{\partial M})(M)$ computes the cup product of sheaf cohomology. We then observe similarly as before that $\tilde{\cL}_{\emptyset}^* \otimes \cR_D$, $\tilde{\cL}_{\emptyset}^* \otimes \cR_E$ and $\tilde{\cL}_{\emptyset}^* \otimes \cR_{\partial M}$ can be identified with sub-resolutions of $\tilde{\cL}_{D}^*$, $\tilde{\cL}_{E}^*$ and $\tilde{\cL}_{D+E}^*$, respectively. Hence, the cup products $\cL_{D}^*(M) \otimes \cL_{E}^*(M) \to \cL_{D + E}^*(M)$ and $(\tilde{\cL}_{\emptyset}^* \otimes \cR_D)(M) \otimes (\tilde{\cL}_{\emptyset}^* \otimes \cR_{E})(M) \to (\tilde{\cL}_{\emptyset}^* \otimes \cR_{\partial M})(M)$ yield the same cup product in map cohomology, completing the proof.
\end{proof}


\bibliographystyle{abbrv}
\bibliography{Surface_Modulus}

\end{document}